\documentclass[aos]{imsart}
\usepackage[utf8]{inputenc}

\usepackage{amsmath}
\usepackage{amsfonts}
\usepackage{amssymb}
\usepackage{amsthm}

\usepackage{enumitem} 
\usepackage{natbib}
\usepackage{caption}
\usepackage[english]{babel}
\usepackage{fontenc}
\usepackage{graphicx}
\usepackage{mathrsfs}
\usepackage{bbm}
\usepackage{array}
\usepackage{float}
\usepackage{textcomp}
\usepackage{rotating}
\usepackage{color}
\usepackage{hyperref}
\usepackage[ruled]{algorithm2e}

\usepackage[table]{xcolor}

\newcommand{\abs}[1]{\left| #1\right|}

\newcommand{\ind}{\mathbbm{1}}
\newcommand{\pr}{\mathbb{P}}
\renewcommand{\Pr}{\mathbb{P}}
\newcommand{\R}{\mathbb{R}}
\newcommand{\E}{\mathbb{E}}

\newcommand{\ceil}[1]{\left\lceil #1 \right\rceil}
\newcommand{\Cov}{\mathrm{Cov}}

\newcommand{\Var}{\mathrm{Var}}

\newcommand{\eqdist}{\stackrel{d}{=}}

\newcommand{\given}{\,|\,}
\newcommand{\independent}{\mbox{${}\perp\mkern-11mu\perp{}$}}

\usepackage{amsthm} 
\newtheorem{theorem}{Theorem}
\newtheorem{corollary}[theorem]{Corollary}
\newtheorem{proposition}[theorem]{Proposition}

\newtheorem{lemma}[theorem]{Lemma}

\usepackage{graphicx}  
\graphicspath{ {} } 

\begin{document}
	
	\begin{frontmatter}
		\title{Average partial effect estimation using double machine learning}
		\begin{aug}
			\author[A]{\fnms{Harvey}~\snm{Klyne}\ead[label=e1]{hck33@cantab.ac.uk}}
						\and
			\author[A]{\fnms{Rajen D.}~\snm{Shah}\ead[label=e2]{r.shah@statslab.cam.ac.uk}}
			\address[A]{Statistical Laboratory, University of Cambridge, United Kingdom\printead[presep={,\ }]{e1,e2}}
		\end{aug}
		\date{\today}
		
	\maketitle

\begin{abstract}
	Single-parameter summaries of variable effects in regression settings are desirable for ease of interpretation. However (partially) linear models for example, which would deliver these, may fit poorly to the data. On the other hand, an interpretable summary of the contribution of a given predictor is provided by the so-called average partial effect---the average slope of the regression function with respect to the predictor of interest. Although one can construct a doubly robust procedure for estimating this quantity, it entails estimating the derivative of the conditional mean and also the conditional score of the predictor of interest given all others, tasks which can be very challenging in moderate dimensions: in particular, popular decision tree based regression methods cannot be used.
	
	In this work we introduce an approach for estimating the average partial effect whose accuracy depends primarily on the estimation of certain regression functions, which may be performed by user-chosen machine learning methods that produce potentially non-differentiable estimates. Our procedure involves resmoothing a given first-stage regression estimator to produce a differentiable version, and modelling the conditional distribution of the predictor of interest through a location--scale model. We show that with the latter assumption, surprisingly the overall error in estimating the conditional score is controlled by a sum of errors of estimating the conditional mean and conditional standard deviation, and the estimation error in a much more tractable univariate score estimation problem. Our theory makes use of a new result  on the sub-Gaussianity of Lipschitz score functions that may be of independent interest. We demonstrate the attractive numerical performance of our approach in a variety of settings including ones with misspecification.
	An R package \texttt{drape} implementing the methodology is available on CRAN.
\end{abstract}
	\end{frontmatter}

	\section{Introduction}
	\label{sec:intro}
	
	A common goal of practical data analysis is to quantify the effect that a particular predictor or set of predictors $X$ has on a response $Y$, whilst accounting for the contribution of a vector of other predictors $Z$. Single-parameter summaries are often desirable for ease-of-interpretability, as demonstrated by the popularity of (partially) linear models. Such models, however, may not adequately capture the conditional mean of the response, potentially invalidating conclusions drawn. Indeed the successes of model-agnostic regression methods such as XGBoost~\citep{xgboost}, random forests~\citep{breiman2001random} and deep learning~\citep{goodfellow2016deep} 
	in machine learning competitions such as those hosted by Kaggle \citep{bojer2021kaggle} suggest that such models fitting poorly is to be expected in many contemporary datasets of interest.
	%
	
	When $X \in \R$ is a continuous random variable and the conditional mean $f(x, z) := \E(Y \given X=x, Z=z)$ is differentiable in the $x$-direction, a natural quantity of interest is the average of the slope with respect to $x$. This is known as the average partial effect (or average derivative effect), defined as 
	\begin{equation*}
		\theta := \E\bigg[\frac{\partial}{\partial x} f(X,Z)\bigg].
	\end{equation*}
Historically, motivation for this estimand came from semiparametric single-index models, i.e., where $f(x, z) = G(\beta x + \gamma^{\top} z)$; the coefficient $\beta$ is then proportional to $\theta$ \citep{StokerIndex,PowellIndex}. The average partial effect also recovers the linear coefficient in a partially linear model
		\begin{equation*}
			f(x,z)=\theta x + g(z).
		\end{equation*}
		Thus the average partial effect may be thought of as a generalisation of the coefficient in a partially linear model that appropriately measures the association of $X$ and the response, while controlling for $Z$, even when there may be complex interactions between $X$ and $Z$ present.
		Indeed, $\theta$ is also the average slope of the so-called partial dependence plot, popular in the field of interpretable machine learning and often used in conjunction with flexible regression methods that place no explicit restrictions on the form of regression functions to be estimated \citep{FriedmanGreedy, ZhaoInterpret, MolnarInterpret}.

	
	\citet{RothenhauslerICE} provide a causal interpretation of the average partial effect in the form of an average outcome change if the `treatment' $X$ of all subjects were changed by an arbitrarily small quantity. More precisely, let us denote by $Y(x)$ the potential outcome \citep{rubin1974estimating} were $X$ to be assigned the value $x$. Then under so-called weak ignorability, that is $\{Y(x) : x \in \R\} \independent X \given Z$ (in fact this can be weakened substantially to permit some forms of confounding, see \citet[Assumption~1]{RothenhauslerICE}), continuity of the joint density of $(X, Z)$ and conditional density of $X \given Z$,
	and mild regularity conditions,
	\[
	\theta = \lim_{\delta \to 0} \frac{1}{\delta} \E \{ Y(X + \delta) - Y(X)\}.
	\]
	In this sense, $\theta$ may be thought of as a continuous analogue of the well-studied average treatment effect functional
	$
	\E(Y(1) - Y(0)) = \E\{\E( Y \given Z, X=1) - \E(Y \given Z, X = 0)\} = \E\{f(1, Z) - f(0, Z)\}
	$
	in the case where $X$ is discrete, only taking values $0$ or $1$, and under analogous assumptions as above  \citep{RRZ94, RR95, SRR99}.
	
	While the average partial effect estimand is attractive from the perspective of interpretability, estimating the derivative of a conditional mean function is challenging. Regression estimators for $f$ which have been trained to have good mean-squared prediction error can produce arbitrarily bad derivative estimates, if they are capable of returning these at all. For example, highly popular tree-based methods give piecewise constant estimated regression functions and so clearly provide unusable estimates for the derivative of $f$.
	
	Moreover, even if the rate of convergence of the derivative estimates was comparable to the mean-squared prediction error when estimating $f$ nonparametrically, an estimator of $\theta$ formed through their empirical average would typically suffer from plug-in bias and fail to attain the parametric rate of convergence. As well as poor estimation quality, this would also make inference, that is, performing hypothesis tests or forming confidence intervals, particularly problematic. The rich theory of semiparametric statistics \citep{bickel1993efficient,tsiatis2006semiparametric} addresses the issue of such plug-in biases more generally, and supports the construction of debiased estimators based on (efficient) influence functions. This basic approach forms a cornerstone of what has become known as debiased or double machine learning \citep{ChernozhukovTreatment}: a collection of methodologies involving user-chosen machine learning methods to produce estimates of nuisance parameters that are used in the construction of estimators of functionals that enjoy parametric rates of convergence (see for example the review article \citet{kennedy2022semiparametric} and references therein).
	
	Procedurally, this often involves modelling both the conditional expectation $f$ and a function of the joint distribution of the predictors $(X, Z)$, with the bias of the overall estimator controlled by a product of biases relating to each of these models \citep{rotnitzky2021characterization}. 
	For the average partial effect, as shown by \citet{PowellIndex,NeweyEfficiency}, the predictor-based quantity to be estimated is the so-called score function (sometimes termed the negative score function)
	\begin{align*}
		\rho(x,z)&:= \frac{\frac{\partial}{\partial x} p(x\given  z) }{p(x\given  z)} = \frac{\partial}{\partial x} \log p(x \given z) 
	\end{align*}
	where $p(x \given z)$ is assumed to be differentiable in $x$. This has been studied in parametric settings \citep{StokerIndex} and in univariate settings (i.e. without any $Z$ present)  using estimators based on splines \citep{CoxSpline, NgSpline, BeraScore}. Nonparametric estimation of the score in the multivariate setting  is particularly challenging owing to the complex nature of potential interactions. Direct estimation through plugging in a kernel density estimate of the joint density $p(x, z)$, for example, can be plagued by stability issues where the estimated density is small. \citet{HardleAverageDerivative} consider thresholding such an estimate to avoiding dividing by quantities close to zero. Such an approach is studied further in \citet{wibisono2024optimal} where a version of the procedure is shown to achieve an expected squared error of order $n^{-\frac{2s}{p+3+2s}}$ up to polylogarithmic factors when $(X, Z) \in \R \times \R^p$ is sub-Gaussian, possesses a joint density and the full multivariate score (also differentiating the joint density with respect to $z$) is $(L, s)$-H\"older continuous.  Moreover, this rate is shown to be the minimax lower bound, thus formalising the difficulty of the problem when the dimension of $Z$ is large.
	 \citet{SriperumbudurInfinite} has considered an approach for multivariate score estimation based on infinite-dimensional exponential families parametrised by a reproducing kernel Hilbert space, and  
	\citet{chernozhukov2020adversarial,ChernozhukovRieszNet} have adapted deep learning architectures and tree splitting criteria to develop neural network and random forest-based approaches for estimating~$\rho$.
	
	
	 One approach to tackling the challenges associated with estimating the derivative of the regression function $f$ and the conditional score is to assume that $f$ and its derivative are sufficiently well-approximated by sparse linear combinations of basis functions \citep{RothenhauslerICE, ChernozhukovRiesz}. Similarly to the case with the debiased Lasso \citep{ZhangDebiased}, where regression coefficients can be estimated without placing explicit sparsity assumptions relating to the conditional distribution of $X$ given $Z$ (see for example \citet{shah2023double}), in this case, fewer assumptions need to be placed on the estimator of $\rho$ \citet[Remark~4.1]{ChernozhukovRiesz}. A related approach relies on $\rho$ itself being well-approximated by a sparse linear combination of basis functions; see  \citet{chernozhukov2020adversarial, ChernozhukovAutoNet, ChernozhukovRobust,ChernozhukovAutoDML, ChernozhukovRiesz, ChernozhukovGeneral} for examples of both of these approaches.
		\cite{HirshbergMinimax} assume that the regression estimation error lies within some absolutely convex class of functions, and perform a  convex optimisation to choose weights that minimise the worst-case mean-squared error over this class. In practice, the class of functions may often be taken as sparse linear combinations of basis functions, and in general it may not always be clear how such basis functions may be chosen.
		\citet{HirshbergSingle} and \citet{wooldridge2020inference} consider parametric single index models for the conditional expectation $f$; this results in a helpful simplification of the problem in the high-dimensional setting these works consider, but may appear overly restrictive in the more moderate-dimensional settings we have in mind here.
		The difficulties of estimating $\theta$ have led \citet{VansteelandtLean} and \citet{hines2021parameterising} to propose interesting alternative estimands that aim to capture some notion of a conditional association of $Y$ and $X$, given $Z$, but whose estimation avoids the challenges of nonparametric conditional score estimation.

	\subsection{Our contributions and organisation of the paper}
	In this paper, we take a different approach, and develop new methods for addressing the two main challenges in estimating the average partial effect $\theta$ using a double machine learning framework as outlined above, namely estimation of the derivative of the conditional mean function $f$ and the conditional score $\rho$.
	
		In Section~\ref{sect:general} we first give a uniform asymptotic convergence result for such doubly robust estimators of $\theta$ requiring user-chosen estimators for $f$ and $\rho$. We argue that uniform results as opposed to pointwise results are particularly important in nonparametric settings such as those considered here. Indeed, considering the problem of testing for a non-zero partial effect, one can show that this is fundamentally hard: when $Z$ is a continuous random variable, any test must have power at most its size. This comes as a consequence of noting that the null in question contains the null that $X \independent Y \given Z$, which is known to suffer from this form of impossibility \citep[Thm.~2]{ShahGCM}. This intrinsic hardness means that any non-trivial test must restrict the null further with the form of these additional conditions, which would be revealed in a uniform result but may be absent in a pointwise analysis, providing crucial guidance on the suitability of tests in different practical settings.
		
		In our case, the conditions for our results require rates of convergence for estimation of the conditional mean $f$, the score $\rho$ and also a condition on the quality of our implied estimate of the derivative of $f$. While estimation of conditional means is a task statisticians are familiar with tackling using machine learning methods for example, 
		the latter two remain challenging to achieve. In contrast to existing work, rather than relying on well-chosen basis function expansions or developing bespoke estimation tools, we aim to leverage once again the predictive ability of modern machine learning methods, which have a proven track record of success in practice. In particular, we wish to accommodate the use of tree-based regression methods such as random forest and XGboost, which produce piecewise constant regression functions and hence cannot directly be used to estimate the derivative of the true regression function. Such methods are popular due to their ability to model interactions and also regularise effectively in multivariate and high-dimensional settings, where other more classical approaches may suffer from the curse of dimensionality \citep{hastie2009elements}. Added practical benefits of these approaches are that they have very well-developed implementations \citep{wright2017ranger,xgboost}, and can handle mixtures of categorical and numerical data gracefully.

	For derivative estimation, we propose a post-hoc kernel smoothing procedure applied to the output of the chosen regression method for estimating $f$.
 In Section~\ref{sect:resmooth} we show that under mild conditions,
	our resmoothing method achieves consistent derivative estimation (in terms of mean-square error) at no asymptotic cost to estimation of $f$ when comparing to the convergence rate enjoyed by the original regression method. Importantly, we do not require the use of a specific differentiable estimator $\hat{f}$ or any explicit assumptions on its complexity or stability properties. 
	This contrasts in particular with some of the literature on estimation of the derivative of a regression function; see for example \citet{dai2016optimal} and references therein, and also \citet{da2008tree,fonseca2018boost} for  smoothing approaches  using sigmoid functions specific to tree-based estimators.
	
	Turning to score estimation, we aim to address the curse of dimensionality in multivariate score estimation such as formalised by the minimax lower bound of \citet{wibisono2024optimal} discussed earlier. To this end, we seek to reduce the problem of conditional score estimation to that of  univariate (unconditional) score estimation, which as explained above, is better studied and more tractable.
	In Section~\ref{sect:score} we advocate modelling the conditional distribution of $X\given  Z$ as a location--scale model (see for example \citet[Sec.~5]{kennedy2017non} who work with this in an application requiring conditional density estimation),
	\begin{equation} \label{eq:loc_scale}
		X=m(Z)+\sigma(Z)\varepsilon,
	\end{equation}
	where $\varepsilon$ is mean-zero and independent of $Z$, $m(z) := \E(X \given Z=z)$, and $\sigma(z): = \sqrt{\Var(X \given Z=z)}$. Through estimating the conditional mean $m$ and standard deviation $\sigma$ via some $\hat{m}$ and $\hat{\sigma}$, one can form scaled residuals $\{X-\hat{m}(Z)\} / \hat{\sigma}(Z)$ which may be fed into an unconditional score estimator. Model~\eqref{eq:loc_scale} has the following attractive features: (i) it is expansive enough to allow for heteroscedasticity, and (ii) by using flexible regression methods to estimate $m$ and $\sigma$, the challenging problem of conditional score estimation may be reduced to unconditional score estimation.
	
	Theoretically, we consider settings where $\varepsilon$ is sub-Gaussian and $\sigma$ is nonparametric, and also the case where $\varepsilon$ is allowed to be heavy-tailed and $\sigma=1$ (i.e.\ a location only model where the errors $X - m(Z) \independent X$). We also demonstrate good numerical performance in heterogeneous, heavy-tailed settings. Given how even the unconditional score involves a division by a density, one concern might be that any errors in estimating $m$ and $\sigma$ may propagate unfavourably to estimation of the score. We show however that the estimation error for the multivariate $\rho(x,z)$ may be bounded by the sum of the estimation errors for the conditional mean $m$, the conditional scale $\sigma$ and the unconditional score function for the residual $\varepsilon$ alone. In this way we reduce the problem of conditional score estimation to unconditional score estimation, plus regression and heterogeneous scale estimation, all of which may be relatively more straightforward.
	Our results rely on proving a sub-Gaussianity property of Lipschitz score functions, which may be of independent interest.

	While a location--scale model may well be a reasonable approximation to the true underlying conditional distribution, with large-scale data, the bias in this modelling assumption may become relevant. In Section~\ref{sec:local} we outline a scheme based on partitioning $\mathcal{Z}$ into regions where the normalised errors $\{X - m(Z)\} / \sigma(Z)$ are approximately independent of $Z$ such that the location--scale model can be applied locally. We note however that more generally, any conditional score estimation procedure may be used in our estimator.
		
	Numerical comparisons of our methodology to existing approaches are contained in Section~\ref{sect:numerical}, where we demonstrate in particular that the coverage properties of confidence intervals based on our estimator have favourable coverage over a range of settings, both where our theoretical assumptions are met and where they are not satisfied. We conclude with a discussion in Section~\ref{sect:discussion}. Proofs and additional results are relegated to the appendix. We provide an implementation of our methods in the R package \texttt{drape} (Doubly Robust Average Partial Effects) available from \url{https://cran.r-project.org/package=drape}.

	\subsection{Notation} \label{sec:notation}
	Let $(Y,X,Z)$ be a random triple taking values in $\R\times \R^d\times\mathcal{Z}$, where $\mathcal{Z}$ is an arbitrary measurable space: this allows for $\mathcal{Z}\subseteq \R^p$ but also permits $Z$ to include functional data, for example. We allow $X$ to be multivariate to enable estimation of average partial effects with respect to multiple variables simultaneously.
	In order to present results that are uniform over a class of distributions $P$ for $(Y,X,Z)$, we will often subscript associated quantities by $P$. For example when $(Y, X, Z) \sim P$, we denote by $\pr_P((Y, X, Z) \in A)$, the probability that $(Y, X, Z)$ lies in a (measurable) set $A$, and write $f_P(x,z) := \E_P(Y\given  X=x, Z=z)$ for the conditional mean function.
	
	Let $\mathcal{P}_0$ be the set of distributions $P$ for $(Y,X,Z)$ where both $f_P$ and  the conditional density of the predictors (with respect to Lebesgue measure) $p_P(x\given  z)$ exist, and the functions $x \mapsto f_P(x, z)$ and $x \mapsto p_P(x \given z)$ are differentiable on $\R^d$ and $p_P(x\given  z)$ has full support, for almost every $z\in\mathcal{Z}$ .

	Write $\nabla$ for the $d$-dimensional differentiation operator with respect to the $x$, which we replace with $'$ if we are considering the case $d=1$. For each $P\in\mathcal{P}_0$, define the  score function $\rho_P(x,z) := \nabla \log p_P(x\given  z)$, where $p_P(x\given  z)$ is the conditional density of $X$ given $Z$ under $P$.
	Denote by $\Phi$ the standard $d$-dimensional normal cumulative distribution function (c.d.f.), and understand inequalities between vectors to apply elementwise.

	We will sometimes introduce standard Gaussian random variables $W\sim N(0,1)$ independent of $(X,Z)$. Recall that a random variable $X\in \R$ is sub-Gaussian with parameter $\sigma$ if it satisfies $\E[\exp(\lambda X)] \leq \exp(\lambda^2\sigma^2/2)$ for all $\lambda\in\R$. A vector $V\in\R^d$ is sub-Gaussian with parameter $\sigma$ if $u^{\top}V$ is sub-Gaussian with parameter $\sigma$ for any $u\in\R^d$ satisfying $\|u\|_2=1$.
	
		As in \cite{LundborgPartial}, given a family of sequences of real-valued random variables $(W_{P,n})_{P\in\mathcal{P}, n\in\mathbb{N}}$ taking values in a finite-dimensional vector space and whose distributions are determined by $P\in\mathcal{P}$, we write
	$W_{P,n} = o_{\mathcal{P}}(1)$ if $\sup_{P\in\mathcal{P}}\Pr_P(|W_{P,n}| > \epsilon) \to 0$ for every $\epsilon > 0$. Similarly, we write $W_{P,n} = O_{\mathcal{P}}(1)$
	if, for any $\epsilon > 0$, there exist $M_\epsilon, N_\epsilon > 0$ such that $\sup_{n\geq N_\epsilon} \sup_{P\in\mathcal{P}} \Pr_P (|W_{P,n}| > M_\epsilon) < \epsilon$. Given a second family of sequences of random variables $(V_{P,n})_{P\in\mathcal{P}, n\in\mathbb{N}}$, we write
	$W_{P,n} = o_\mathcal{P}(V_{P,n})$ if there exists $R_{P,n}$ with $W_{P,n} = V_{P,n} R_{P,n}$ and $R_{P,n}=o_\mathcal{P}(1)$; likewise, we write $W_{P,n} = O_\mathcal{P}(V_{P,n})$
	if $W_{P,n} = V_{P,n} R_{P,n}$ and $R_{P,n} = O_\mathcal{P}(1)$. If $W_{P,n}$ is vector or matrix-valued, we write $W_{P,n} = o_\mathcal{P}(1)$ if $\|W_{P,n}\| = o_\mathcal{P}(1)$ for some norm, and similarly $O_\mathcal{P}(1)$. By the equivalence of norms for finite-dimensional vector spaces, if this holds for some norm then it holds for all norms.
	
	\section{Doubly robust average partial effect estimator}
	\label{sect:general}
	
	We consider a nonparametric model
	\begin{equation*}
		\E_P(Y\given  X,Z) =: f_P(X,Z),
	\end{equation*}
	for a response $Y\in \R$, continuous predictors of interest $X\in\R^d$, and additional predictors $Z\in\mathcal{Z}$ of arbitrary type. We assume that $(Y,X,Z)\sim P\in\mathcal{P}_0$ (see Section~\ref{sec:notation}), and so the conditional mean $f_P(x,z)$ and the conditional density $p_P(x\given  z)$ are differentiable with respect to $x$.  Our goal is to make inference on the average partial effect
	\begin{equation*}
		\theta_P := \E_P [ \nabla f_P(X,Z) ].
	\end{equation*}
	Recall that the score function $\rho_P : \R^d \times \mathcal{Z}\to \R^d$ plays an important role when considering estimation of $\theta_P$ because it acts like the differentiation operator in the following sense; see also \citet{StokerIndex, NeweyEfficiency}.
	\begin{proposition}
		\label{prop:intbyparts}
		Let the conditional density of the predictors $p_P(\cdot \given  z)$ with respect to Lebesgue measure exist, have full support and be differentiable in the $j$th coordinate $x_j$ for every $x_{-j} \in \R^{d-1}$ and almost every $z\in\mathcal{Z}$. Let $g:\R^d \times \mathcal{Z}\to \R$ similarly be differentiable with respect to $x_j$ at almost every $z\in\mathcal{Z}$ and satisfy
		\begin{equation} \label{eq:intbyparts}
			\E_P[|\nabla_j g(X,Z) + \rho_{P,j}(X,Z) g(X,Z)|]<\infty.
		\end{equation}
		Suppose that for almost every $(x_{-j}, z)$ there exist sequences $a_n\to -\infty$, $b_n\to\infty$ (potentially depending measurably on $(x_{-j}, z)$) such that
		\begin{equation}
		    \label{eq:boundary}
            \lim_{n\to \infty} \{g(b_n,x_{-j},z)p_P(b_n, x_{-j} \given  z) - g(a_n,x_{-j},z)p_P(a_n, x_{-j} \given  z)\} = 0;
		\end{equation}
		abusing notation, we write for example $g(a, x_{-j}, z)$ for $g(x_1, \ldots, x_{j-1}, a, x_{j+1},\ldots, x_d, z)$. Then
		\begin{equation*}
			\E_P[\nabla_j g(X,Z) + \rho_{P,j}(X,Z)g(X,Z)] = 0.
		\end{equation*}
	\end{proposition}
	Proposition~\ref{prop:intbyparts}, which follows from integration by parts, allows one to combine estimates of $f_P$ and $\rho_P$ to produce a doubly-robust estimator as we now explain. Suppose we have some fixed function estimates $(\hat f,\hat\rho)$, for example computed using some independent auxiliary data, and $g := f_P - \hat f$ obeys the conditions of Proposition~\ref{prop:intbyparts}. Then 
	\begin{equation*} 
		\begin{split}
	&	\E_P\big[\nabla \hat f(X,Z) - \hat\rho(X,Z)\big\{Y-\hat f(X,Z)\big\}\big] - \theta_P \\
	=&\, \E_P\big[ - \rho_P(X, Z) \hat{f}(X, Z)  - \hat\rho(X,Z)\big\{f_P(X, Z)-\hat f(X,Z)\big\}\big] + \E_P \big[ \rho_P(X, Z)f_P(X, Z)  \big] \\
= &\, \E_P\big[\big\{\rho_P(X,Z)-\hat\rho(X,Z)\big\}\big\{f_P(X,Z)-\hat f(X,Z)\big\}\big],
		\end{split}
	\end{equation*}
which will be zero if either $\hat{f}$ or $\hat\rho$ equal $f_P$ or $\rho_P$ respectively.
	Given independent, identically distributed (i.i.d.) samples $(y_i,x_i,z_i)~\sim~P$ for $i=1,\ldots,n$, this motivates an average partial effect estimator  of the form
	\begin{equation*} 
		\frac1n \sum_{i=1}^n \nabla \hat f(x_i,z_i) - \hat\rho(x_i,z_i)\big\{y_i-\hat f(x_i,z_i)\big\}.
	\end{equation*}
	From the penultimate display 
	and the Cauchy--Schwarz inequality (applied componentwise), we see that the squared-bias of such an estimator is at worst the product of the mean-squared error rates of the conditional mean and score function estimates. A consequence of this (see Theorem~\ref{thm:clt} below) is that the average partial effect estimate can achieve root-$n$ consistency even when both conditional mean and score function estimators converge at a slower rate. Such an estimator is typically called doubly robust \citep{SRR99, RobinsCommentMurphy, RobinsCommentBickel}.
	Note that this relies on the boundary condition \eqref{eq:boundary} being satisfied. Indeed, as noted in \citet{NeweyEfficiency}, in a setting where $p_p(\cdot \given z)$ has compact support and such a condition where $a_n$ and $b_n$ approach the boundary fails, root-$n$ consistent estimation of the average partial effect will not be possible.
	
	In practice, the function estimates $\hat f$ and $\hat \rho$ would not be fixed and must be computed from the same data. For our theoretical analysis, it is helpful to have independence between the function estimates and the data points on which they are evaluated. For this reason we mimic the setting with auxiliary data by employing a
	sample-splitting scheme known as cross-fitting \citep{schick1986asymptotically,ChernozhukovTreatment}, which works as follows.
	%
%
%
%
	Given a sequence of i.i.d.\ data sets $\{(y_i,x_i,z_i)\;:\; i=1,\ldots,n\}$, define a $K$-fold partition $\big(I^{(n,k)}\big)_{k=1,\ldots,K}$ of $\{1,\ldots,n\}$ for some $K$ fixed (in all our numerical experiments we take $K=5$). For simplicity of our exposition, we assume that $n$ is a multiple of $K$ and each subset is of equal size $n/K$. Let the pair of function estimates $\big(\hat f^{(n,k)},\hat\rho^{(n,k)}\big)$ be estimated using data
	\begin{equation*}
		D^{(n,k)} := \big\{(y_i,x_i,z_i)\;:\; i\in \{1,\ldots,n\}\setminus I^{(n,k)}\big\}.
	\end{equation*} The cross-fitted, doubly-robust estimator is
\begin{equation} \label{eqn:est}
		\hat\theta^{(n)} := \frac1n \sum_{k=1}^K  \sum_{i\in I^{(n,k)}} \left[ \nabla \hat f^{(n,k)}(x_i,z_i) - \hat\rho^{(n,k)}(x_i,z_i)\big\{y_i-\hat f^{(n,k)}(x_i,z_i)\big\} \right], 
	\end{equation}
	with corresponding variance estimator
	\begin{equation} \label{eq:Sigma_est}
		\begin{split}
					\hat \Sigma^{(n)} := \frac1n \sum_{k=1}^K \sum_{i\in I^{(n,k)}} &\Big[\nabla \hat f^{(n,k)}(x_i,z_i) - \hat\rho^{(n,k)}(x_i,z_i)\big\{y_i-\hat f^{(n,k)}(x_i,z_i)\big\} -\hat\theta^{(n)}\Big]\\
			&\cdot \Big[\nabla \hat f^{(n,k)}(x_i,z_i) - \hat\rho^{(n,k)}(x_i,z_i)\big\{y_i-\hat f^{(n,k)}(x_i,z_i)\big\} -\hat\theta^{(n)}\Big]^{\top}.
		\end{split}
	\end{equation}

	
	\subsection{Uniform asymptotic properties}
	The estimator \eqref{eqn:est} may be viewed as a special case of those considered in 
	a flurry of recent work \citep{chernozhukov2020adversarial,  ChernozhukovAutoNet, ChernozhukovRobust, ChernozhukovRieszNet, ChernozhukovAutoDML, ChernozhukovRiesz,    ChernozhukovGeneral}. These works look at doubly-robust inference on a broad range of functionals of the conditional mean function, satisfying a moment equation of the form
	\begin{equation*}
				\E_P[\Psi(X,Z; f_P)] = \beta_P,
			\end{equation*}
		for a known operator $\Psi$ and unknown target parameter $\beta_P$. This encompasses estimation of $\theta_P$ by taking $\Psi(x,z;\Delta) = \partial_x \Delta(x,z)$.
		The general theory in this line of work however typically relies on a 
		 mean-squared continuity assumption of the form
		\begin{equation*}
				\E_P\big[\{\Psi(X,Z; \Delta)\}^2\big] \leq C\big\{\E_P\big[\Delta^2(X,Z)\big]\big\}^q,
			\end{equation*} 
		for some $q>0$ and all $\Delta\in \mathcal{F}$, where $\mathcal{F}$ contains the conditional mean estimation errors $f_P - \hat f$. This therefore requires $\hat{f}$ to be differentiable, and its derivative to approximate that of $f_P$ sufficiently well. Such an assumption would not hold for popular tree-based estimators which are piecewise constant.
We address this crucial issue in Section~\ref{sect:resmooth}, but first
		give below a uniform asymptotic result relating to estimators $\hat{\theta}^{(n)}$ \eqref{eqn:est} and $\hat{\Sigma}^{(n)}$ \eqref{eq:Sigma_est},
%
%
%
	not claiming any substantial novelty, but  so as to introduce the quantities $A_f^{(n)}, A_\rho^{(n)}, E_f^{(n)}, E_\rho^{(n)}$ which we will seek to bound in later sections. We stress that it is in achieving these bounds in a model-agnostic way that our major contributions lie.
	
	Recall that the optimal variance bound is equal to the variance of the efficient influence function, which for the average partial effect in the nonparametric model $P\in\mathcal{P}_0$ is equal to
	\begin{equation*}
		\psi_P(y,x,z):=\nabla f_P(x,z) - \rho_P(x,z)\{y-f_P(x,z)\} - \theta_P,
	\end{equation*}
	provided that $\Sigma_P:=\E_P\big[\psi_P(Y,X,Z)\psi_P(Y,X,Z)^{\top}\big]$ exists and is non-singular \citep[Thm.~3.1]{NeweyEfficiency}.
	The theorem below shows that $\hat{\theta}^{(n)}$ achieves this variance bound asymptotically.
	\begin{theorem}
		\label{thm:clt}
		Define the following sequences of random variables:
		\begin{align*}
			A_f^{(n)} &:= \E_P \big[\{f_P(X,Z)-\hat f^{(n,1)}(X,Z)\}^2 \; \big| \; D^{(n,1)}\big],\\
			A_\rho^{(n)} &:= \max_{j=1,\ldots,d} \E_P\big[\{\rho_{P,j}(X,Z)-\hat \rho^{(n,1)}_j(X,Z)\}^2 \; \big| \; D^{(n,1)}\big],\\
			E_f^{(n)} &:= \max_{j=1,\ldots,d} \E_P\Big(\big[ \nabla_j f_P(X,Z)-\nabla_j \hat f^{(n,1)}(X,Z)  \\
			& \qquad +\rho_{P,j}(X,Z)\{f_P(X,Z)-\hat f^{(n,1)}(X,Z)\}\big]^2 \; \Big| \; D^{(n,1)}\Big),\\
			E_\rho^{(n)} &:= \max_{j=1,\ldots,d}\E_P\big[\{\rho_{P,j}(X,Z)-\hat \rho^{(n,1)}_j(X,Z)\}^2 \; \Var_P(Y \given  X,Z) \; \big| \; D^{(n,1)}\big];
		\end{align*}
	note that we have suppressed $P$-dependence in the quantities defined above.
		Let $\mathcal{P} \subset \mathcal{P}_0$ and the chosen regression method producing $\hat{f}^{(n,1)}$ be such that all of the following hold. The covariance matrix $\Sigma_P$ exists for every $P\in\mathcal{P}$, with minimum eigenvalue at least $c_1>0$. Furthermore, 
		\begin{align*}
			\sup_{P\in\mathcal{P}} \E_P\big[\|\psi_{P}(Y,X,Z)\|_2^{2+\eta}\big]&\leq c_2
		\end{align*}
		for some $c_2,\eta>0$. Finally, suppose the remainder terms defined above satisfy:
		\begin{equation}
			\label{eqn:cltrates}
			A_f^{(n)} = O_\mathcal{P}(1);\quad A_f^{(n)}A_\rho^{(n)} = o_{\mathcal{P}}(n^{-1});\quad
			E_f^{(n)}=o_{\mathcal{P}}(1);\quad
			E_\rho^{(n)}=o_{\mathcal{P}}(1).
		\end{equation}
		Then $\hat{\theta}^{(n)}$ \eqref{eqn:est} is root-$n$ consistent, asymptotically Gaussian and efficient:
		\begin{equation*}
			\lim_{n\to\infty}\sup_{P\in\mathcal{P}} \sup_{t\in\R^d} \Big| \Pr_P\big[\sqrt{n}(\Sigma_P)^{-1/2} (\hat\theta^{(n)}-\theta_P)\leq t\big]-\Phi(t)\Big| = 0.
		\end{equation*}
		Moreover the covariance estimate \eqref{eq:Sigma_est} satisfies $\hat\Sigma^{(n)} = \Sigma_P + o_{\mathcal{P}}(1)$, and one may perform asymptotically valid inference (e.g.\ constructing confidence intervals) using
		\begin{equation*}
			\lim_{n\to\infty}\sup_{P\in\mathcal{P}} \sup_{t\in\R^d} \Big| \Pr_P\big[\sqrt{n}(\hat\Sigma^{(n)})^{-1/2} (\hat\theta^{(n)}-\theta_P)\leq t\big]-\Phi(t)\Big| = 0.
		\end{equation*}
	\end{theorem}
	Let us discuss the assumptions of the result. Firstly, 	considering the case $d=1$ for simplicity, a sufficient condition for the variance bound $\Sigma$ to be finite is that $\rho_P'(x, z)$ is bounded, $\E_P f'(X, Z)^2 < \infty$ and $\E|Y-f(X, Z)|^{2+\eta} < \infty$; see Section~\ref{sec:var_bound} for a derivation of this fact. The assumption of $\rho_P$ being Lipschitz is something we discuss further in Section~\ref{sect:score}. 
	The minimum eigenvalue of $\Sigma_P$ being bounded away from zero and the stronger assumption of the $(2+\eta)$th moment of the influence function being bounded are required to obtain a uniform, rather than a weaker pointwise result, the latter condition in particular being required for an application of the Lyapunov central limit theorem.
	The assumptions on $A_f^{(n)}$, $A_\rho^{(n)}$ and $E_\rho^{(n)}$ are relatively weak and standard; for example they are satisfied if the conditional variance $\Var_P(Y \given  X,Z)$ is bounded almost surely and each of $A_f^{(n)}$, $A_\rho^{(n)}$ converge at the nonparametric rate $o_{\mathcal{P}}(n^{-1/2})$; see Section~\ref{sect:score} for our scheme on score estimation. For example, consider the case where $\mathcal{Z} = \R^p$ and $f_P$ is $(C,s)$-H\"older smooth for $C,s>0$, i.e., writing $m:=\ceil{s} -1$, for every $\alpha :=(\alpha_1, \ldots, \alpha_{d + p})$ with $\alpha_1 + \cdots + \alpha_d = m$ and $\alpha_j \in \mathbb{Z}_{\geq 0}$, the partial derivatives (assumed to exist) satisfy
	\begin{align*}
	&\abs{\frac{\partial^\alpha f_P (x, z)}{\partial^{\alpha_1} x_1 \cdots \partial^{\alpha_d} x_d \, \partial^{\alpha_{d+1}} z_1 \cdots \partial^{\alpha_{d+p}} z_p} - \frac{\partial^\alpha f_P (x', z')}{\partial^{\alpha_1} x_1 \cdots \partial^{\alpha_d} x_d \, \partial^{\alpha_{d+1}} z_1 \cdots \partial^{\alpha_{d+p}} z_p}}\\
	\leq &  \, C \| (x, z) - (x', z')\|_2^{s - m}
	\end{align*}
	for all $P \in \mathcal{P}$ and $(x, z), (x', z') \in \R^{d+p}$.
	Then we can expect that $A_f^{(n)} = O_{\mathcal{P}}(n^{-2s / (2s + d + p)})$ for appropriately chosen regression procedures; see for example \citet{gyorfi2002distribution}. Then when $s > (d+p)/2$, this is $o_{\mathcal{P}}(n^{-1/2})$.
	Moreover, a faster rate for $A_f^{(n)}$ permits a slower rate for $A_\rho^{(n)}$ and vice versa.
	
	However assuming $E_f^{(n)}=o_{\mathcal{P}}(1)$ needs justification. In particular, while the result aims to give guarantees for a version of $\hat{\theta}^{(n)}$ constructed using arbitrary user-chosen regression function and score estimators $\hat{f}^{(n,1)}$ and $\hat{\rho}^{(n,1)}$, in particular it requires $\hat{f}^{(n,1)}$ to be differentiable in the $x$ coordinates, which for example is not the case for regression tree-based estimates of $f_P$. In the next section, we  address this issue by proposing a resmoothing scheme to yield a suitable estimate of $f_P$ that can satisfy the requirements on $A_f^{(n)}$ and $E_f^{(n)}$ simultaneously.
	


	\section{Resmoothing}
	\label{sect:resmooth}

	In this section we propose performing derivative estimation via a kernel convolution applied to an arbitrary initial regression function estimate. This is inspired by similar approaches for edge detection in image analysis \citep{CannyEdge}. We do this operation separately for each coordinate $x_1,\ldots,x_d$ of interest, so without loss of generality in this section we take the dimension of $x$ to be $d=1$ (so $\nabla(\cdot) = (\cdot)'$). Motivated by Theorem~\ref{thm:clt}, we seek a class of differentiable regression procedures so that the errors
	\begin{align*}
		A_f^{(n)} &:= \E_P \big[\{f_P(X,Z)-\hat f^{(n,1)}(X,Z)\}^2 \; \big| \; D^{(n,1)}\big],\\
		E_f^{(n)} &:= \E_P\Big(\big[ f_P'(X,Z)-(\hat{f}^{(n,1)})'(X,Z) +\rho_{P}(X,Z)\{f_P(X,Z)-\hat f^{(n,1)}(X,Z)\}\big]^2 \; \Big| \; D^{(n,1)}\Big),
	\end{align*}
	satisfy \begin{equation*}
		A_f^{(n)} = O_{\mathcal{P}}(n^{-\alpha});\quad
		E_f^{(n)}=o_{\mathcal{P}}(1),
	\end{equation*}
	for $\alpha>0$ as large as possible.
	
	Consider regressing $Y$ on $(X,Z)$ using some favoured machine learning method, whatever that may be. By training on $D^{(n,k)}$ we get a sequence of estimators $\tilde f^{(n,k)}$ of the conditional mean function $f_P(x,z):=\E(Y\given  X=x, Z=z)$, which we expect to have a good mean-squared error convergence rate but that are not necessarily differentiable (or that their derivatives are hard to compute, or numerically unstable). Additional smoothness may be achieved by convolving $\tilde f^{(n,k)}$ with a kernel, yielding readily computable derivatives. Let $K:\R\to\R$ be a differentiable kernel function, i.e., such that $\int_{-\infty}^{\infty} K(u)\, du=1$. The convolution yields a new regression estimator
	\begin{equation*}
		\hat f^{(n,k)}(x,z) = \{\tilde f^{(n,k)}(\cdot, z)*K_{h_n}\}(x) = \int_{\R} \tilde f^{(n,k)}(u,z) K_{h_n}(x-u) \; du,
	\end{equation*}
	where $K_{h_n}(t)=h_n^{-1} K(h_n^{-1}t)$ for a sequence of bandwidths $h_n>0$. Here we will use a standard Gaussian kernel,\begin{equation*}
		K(u)=\frac1{\sqrt{2\pi}} \exp\bigg(-\frac{u^2}{2}\bigg);
	\end{equation*}
	the Gaussian kernel is positive, symmetric, and satisfies $ K'(u) = -u K(u)$, which makes it convenient for our theoretical analysis.
	
	\subsection{Theoretical results}
	Our goal in this section is to demonstrate that for some sequence of bandwidths $h_n$ and a class of distributions 
	 $\mathcal{P}\subset \mathcal{P}_0$ that will encode any additional assumptions we need to make, we have relationships akin to 
	  \begin{equation}
		A_f^{(n)} = O_{\mathcal{P}}\Big(\E_P\Big[\big\{f_P(X,Z)-\tilde f^{(n,1)}(X,Z)\big\}^2 \;\Big| \; D^{(n,1)}\Big]\Big) \quad \text{and} \quad E_f^{(n)} = o_{\mathcal{P}}(1). \label{eqn:hratetarget}
	\end{equation}
This means that we can preserve the mean squared error properties of the original $\tilde{f}^{(n,1)}$ but also achieve the required convergence to zero of the term $E_f^{(n)}$; see Theorem~\ref{thm:clt}. A result of this flavour is given by the following theorem. Recall that we have assumed that the dimension of $x$ is $1$ here; without loss of generality other components of $x$ may be absorbed into $z$.

	\begin{theorem}
		\label{thm:hrate}
		Let
		\[
		\tilde{A}_f^{(n)} := \E_P\Big[\big\{f_P(X,Z)-\tilde f^{(n,1)}(X,Z)\big\}^2 \;\Big| \; D^{(n,1)}\Big],
		\] where
		the randomness is over the training data set $D^{(n,1)}\sim P$. Let $\mathcal{P} \subset \mathcal{P}_0$ and the chosen regression method producing the fitted regression function $\tilde{f}^{(n, 1)}$ be such that all of the following hold.
		\begin{enumerate}[label=(\roman*)]
			\item The regression error of $\tilde f^{(n,1)}$ is bounded everywhere with high probability:
			\begin{equation*}
				\sup_{x,z}\big|f_P(x,z) - \tilde f^{(n,1)}(x,z)\big|  = O_{\mathcal{P}}(1).
			\end{equation*} 
			\item For each $P\in\mathcal{P}$ and almost every $z\in\mathcal{Z}$ the conditional density $p_P(\cdot\given  z)$ is twice continuously differentiable and
			\begin{equation} \label{eq:score_Lipschitz}
				\sup_{P\in \mathcal{P}} \sup_{x,z} \big| \rho_P'(x,z)\big| < \infty.
			\end{equation}
			\item There exists a class of functions $C_P:\mathcal{Z}\to\R$, $P\in\mathcal{P}$ such that $\sup_{P\in\mathcal{P}} \E_P\big[C_P^2(Z)\big] < \infty$ and \begin{equation}
				\sup_{x} \big|f_P''(x,z)\big| \leq C_P(z) \label{eqn:resmooththmderivbounds}
			\end{equation} 
			for almost every $z\in\mathcal{Z}$, for each $P\in\mathcal{P}$.
		\end{enumerate}
		If $\tilde{A}_f^{(n)} = O_{\mathcal{P}}\big(n^{-\alpha}\big)$ for some $\alpha > 0$
		then the choice $h_n = c n^{-\gamma}$ for any $c>0$ and $\gamma \in [\alpha/4, \alpha/2)$
		achieves $A_f^{(n)} = O_{\mathcal{P}}\big(n^{-\alpha}\big)$ and $E_f^{(n)} = O_{\mathcal{P}}\big(n^{2\gamma-\alpha}\big) = o_{\mathcal{P}}(1)$.
	\end{theorem}
The result shows that
we have a large range of possible bandwidth sequences, whose decay to zero varies in orders of magnitude, for which we have the desirable conclusion that the associated
smoothed estimator $\hat{f}^{(n,1)}$ enjoys $A_f^{(n)} = O_{\mathcal{P}}\big(n^{-\alpha}\big)$  as in the case of the original $\tilde{f}^{(n,1)}$, but crucially also $E_f^{(n)} = o_{\mathcal{P}}(1)$ as required by Theorem~\ref{thm:clt}. Note that in order to satisfy $A_f^{(n)}A_\rho^{(n)} = o_{\mathcal{P}}(n^{-1})$ \eqref{eqn:cltrates}, we may expect that $1/2 < \alpha \leq 1$, in which case $\gamma=1/4$ always lies within the permissible interval of bandwidths ensuring $E_f^{(n)} = o_{\mathcal{P}}(1)$.

The additional assumptions required by the result are a Lipschitz property of the score function \eqref{eq:score_Lipschitz} and a particular bound on the second derivative of the regression function $f_P$ \eqref{eqn:resmooththmderivbounds}, both of which we consider to be relatively mild. Nevertheless, in particular densities whose tails decay faster than that of a Gaussian are ruled out by the former assumption. Moreover for such densities the variance of the efficient influence function could still be finite, suggesting that this assumption is not necessary for efficient estimation of the average partial effect; see Section~\ref{sec:var_bound} for further discussion.

As  shown in  Theorem~\ref{thm:scoresubg} to follow, \eqref{eq:score_Lipschitz} implies a sub-Gaussianity property of the score function. To gain some intuition for how the condition on the second derivative of $f_P$ is helpful, observe that for instance
\[
f_P' - (\tilde{f}^{(n,1)} * K_h)' = \{f_P' - (f_P*K_h)'\} + \{(f_P-\tilde{f}^{(n,1)}) * K_h\}'.
\]
The first term on the right-hand side is zero when $f_P''=0$ and more generally can be controlled by the size of this second derivative. On the other hand, the second term may be controlled by the original regression error $f_P-\tilde{f}^{(n,1)}$. We give a brief sketch of how the assumed Lipschitz property of the score proves useful in our arguments in Section~\ref{sect:score}. 

%
%
%
	
	With this result on the insensitivity of the bandwidth choice with respect to the adherence to the conditions of Theorem~\ref{thm:clt} in hand, we now discuss a simple practical scheme for choosing an appropriate bandwidth.
	
	
	\subsection{Practical implementation}
	\label{sect:resmoothimplementation}
	
	There are two practical issues that require consideration. First, we must decide how to compute the convolutions. Second, we discuss bandwidth selection for the kernel for which we suggest a data-driven selection procedure that involves picking the largest resmoothing bandwidth achieving a cross-validation score within some specified tolerance of the original regression.
	
	Now for $W\sim N(0,1)$ independent of $(X,Z)$, we have
	\begin{align*}
		\hat f^{(n,k)}(x,z) &=  \E\big[\tilde f^{(n,k)}(x+hW,z) \; \big| \; D^{(n,k)}\big],\\
		\hat f^{(n,k)\prime}(x,z) &=  \frac1h \E\big[W\tilde f^{(n,k)}(x+hW,z)\; \big|\; D^{(n,k)}\big].
	\end{align*}
	The latter expression follows from differentiating under the integral sign; see Lemma~\ref{lem:convolution} in the Appendix for a derivation.
	While this indicates it is possible to compute Gaussian expectations to any degree of accuracy by Monte Carlo, the regression function $\tilde f^{(n,k)}$ may be expensive to evaluate too many times, and we have found that derivative estimates are sensitive to sample moments deviating from their population values. These issues can be alleviated by using antithetic variates, however we have found the simpler solution of grid-based numerical integration (as is common in image processing \citep{CannyEdge}) to be very effective. This involves finding a deterministic set of pairs $\{(w_j, q_j)\;:\;j=1,\ldots,J\}$ such that, for functions $g$, \begin{equation*}
		\E [g(W)] \approx \sum_{j=1}^J g(w_j)q_j.
	\end{equation*} We suggest taking the $\{w_j\}$ to be an odd number of equally spaced grid points covering the first few standard deviations of $W$, and $\{q_j\}$ to be proportional to the corresponding densities, such that $\sum_{j=1}^J q_j~=~1$. This ensures that the odd sample moments are exactly zero and the leading even moments are close to their population versions. In all of our numerical experiments presented in Section~\ref{sect:numerical} we take $J=101$.
	
	Recall that the goal of resmoothing is to yield a differentiable regression estimate without sacrificing the good prediction properties of the first-stage regression. With this intuition, we suggest choosing the largest bandwidth such that quality of the regression estimate in terms of squared error, as measured by cross-validation score, does not substantially decrease.
	
Specifically, the user first specifies a non-negative tolerance and a collection of positive trial bandwidths, for instance an exponentially spaced grid up to the empirical standard deviation of $X$. Next, we find the bandwidth $h_{\min}$ minimising the cross-validation error across the given set of bandwidths including bandwidth $0$ (corresponding to the original regression function). Then, for each positive bandwidth at least as large as $h_{\min}$, we find the largest bandwidth $h$ such that the corresponding cross-validation score $\text{CV}(h)$ exceeds $\text{CV}(h_{\min})$ by no more than some tolerance times an estimate of the standard deviation of the difference $\text{CV}(h)-\text{CV}(h_{\min})$; if no such $h$ exists, we pick the minimum positive bandwidth. Given a sufficiently small minimum bandwidth, this latter case should typically not occur.
	
	The procedure is summarised in Algorithm~\ref{alg:resmooth} below.
	We suggest computing all the required evaluations of $\tilde f^{(n,k)}$ at once, since this only requires loading a model once per fold. In all of our numerical experiments presented in Section~\ref{sect:numerical} we used $K=5$ and set the tolerance to be $2 \approx \Phi^{-1}(0.975)$, though the results were largely unchanged for a wide range of tolerances. Overall, the computation time for this was negligible compared to that required to train the regression models involved in the construction of our estimator.
	
\begin{algorithm}[h]
        \KwIn{Data set $D^{(n)}$, number of folds $K\in\mathbb{N}$, set of $L$ positive potential bandwidths $\mathcal{H} := \{h_1, \ldots, h_L\}$, tolerance $\text{tol} \geq 0$ controlling the permissible increase in regression error.}
	\KwOut{Bandwidth $\hat h \geq 0$.}
        Partition $D^{(n)}$ into $K$ folds.\\
		\For{each fold $k=1,\ldots,K$}{
			Train $\tilde f^{(n,k)}$ on the out-of-fold data $D^{(n,k)}$.\\
For each $i \in I^{(n,k)}$, set $\text{err}_{i}(0):= \{Y_i - \tilde f^{(n,k)}(X_i,Z_i)\}^2$.\\
			\For{each trial bandwidth $h\in\{h_1,\ldots, h_L\}$}{
				\For{each in-fold data point $(X_i,Z_i)$, $i\in I^{(n,k)}$}{
					Compute $\hat f^{(n,k)}(X_i,Z_i) =  \sum_{j=1}^J \tilde f^{(n,k)}(X_i+h w_j,Z_i)q_j$.\\
					Set $\text{err}_{i}(h) := \{Y_i - \hat f^{(n,k)}(X_i,Z_i)\}^2$
				}
			}
		}
		Writing $h_0 := 0$, for each $l=1,\ldots,L$, set $\text{CV}(h_l)$ to be the mean of the $\{\text{err}_i(h_l)\}_{i=1}^n$.\\
		Set $h_{\min} := \text{argmin}_h \text{CV}(h)$.\\
		For each $h \in \mathcal{H}$ such that $h \geq h_{\min}$, set $\text{se}(h)$ to be the empirical standard deviation of  $\{\text{err}_{i}(h_{\min}) - \text{err}_{i}(h)\}_{i=1}^n$ divided by $\sqrt{n}$.\\
		Set $\hat{h}$ to be the largest $h \in \mathcal{H}$ with $h \geq h_{\min}$ such that $\text{CV}(h) \leq \text{CV}(h_{\min}) + \text{tol} \times \text{se}(h)$, or set $\hat{h} := \min \mathcal{H}$ if no such $h$ exists.
		\caption{Cross validation selection procedure for the resmoothing bandwidth.}
  \label{alg:resmooth}
	\end{algorithm}

	\section{Score estimation}
	\label{sect:score}
	In this section we consider the problem of constructing an estimator of the score function $\rho_P$ of a random variable $X$ conditional on $Z$ as required in the estimator $\hat{\theta}^{(n)}$ \eqref{eqn:est} of the average partial effect $\theta_P$. However, score function estimation is also of independent interest more broadly, for example in distributional testing, particularly for assessing tail behaviour \citep{BeraScore}. More recently, there has been renewed interest in score estimation in the context of  generative modelling via diffusion models \citep{song2020score}.
	 The latter have achieved remarkable success in a variety of machine learning tasks including text-to-speech generation \citep{popov2021diffusion}, image inpainting and super-resolution \citep{song2021solving} for example. In these very high-dimensional applications, the score is estimated by deep neural networks that can leverage underlying structure in the data. We however also have in mind lower-dimensional problems that may have less structure and may benefit from a nonparametric approach.	
	This nonparametric conditional score estimation problem that we seek to address has received less attention than the simpler problem of an unconditional score estimation on a single univariate random variable; the latter may equivalently be expressed as the problem of estimating the ratio of the derivative of a univariate density and the density itself. In Section~\ref{sect:scorelsfamilies}, we propose a location--scale model that then reduces our original problem to the latter, and  in \ref{sec:location} by strengthening our modelling assumption to a location-only model, we weaken requirements on the tail behaviour of the errors. In Section~\ref{sec:local} we discuss how this simplifying assumption may be relaxed by only considering using it locally.
	
	Note that 
	\begin{align*}
		\rho_{P,j}(x,z)&=\nabla_j \log p_P(x\given  z)\\
		&= \nabla_j \log \{p_P(x_j \given  x_{-j}, z) \; p_P(x_{-j}\given  z)\}\\
		&= \nabla_j \log p_P(x_j \given  x_{-j}, z).
	\end{align*}
	Therefore each component of $\rho_P$ may be estimated separately using the conditional distribution of $X_j$ given $(X_{-j},Z)$. This means that we can consider each variable separately, so for the rest of this section we assume that $d=1$ without loss of generality. 
	
	Before we discuss location--scale families, we first present a theorem on the sub-Gaussianity of Lipschitz  score functions that is key to the results to follow and may be of independent interest.
	An interesting property of score functions is that their tail behaviour is ``nicer'' for heavy-tailed random variables. If a distribution has Gaussian tails, its score has linear tails. If a distribution has exponential tails, the score function has constant tails. If a distribution has polynomial (power law) tails, the score function tends to zero. This trade-off has a useful consequence: $\rho_P(X,Z)$ can be sub-Gaussian even when $X\given  Z$ is not. This is shown in the following result, which is proved using repeated integration by parts. We state it here for a univariate (unconditional) score. However we note that when the conditional distribution of $X$ given $Z$ satisfies the conditions of Theorem \ref{thm:scoresubg}, the same conclusions hold conditionally on $Z$.
	\begin{theorem}
		\label{thm:scoresubg}
		Let $X$ be a univariate random variable with density $p$ twice differentiable on $\R$ and score function $\rho$ satisfying $\sup_{x\in\R} |\rho'(x)| \leq C<\infty$. Then for all positive integers $k$, \begin{equation*}
			\E\big[\rho^{2k}(X)\big] \leq C^k (2k-1)!!,
		\end{equation*}
		where $m!!$ denotes the double factorial of $m$, that is the product of all positive integers up to $m$ that have the same parity as $m$. Furthermore,  the random variable $\rho(X)$ is sub-Gaussian with parameter $\sqrt{2C}$. If additionally $X$ is symmetrically distributed, then the sub-Gaussian parameter may be reduced to $\sqrt{C}$.
	\end{theorem}
	A Lipschitz assumption on the score is for example standard in the literature on sampling mentioned above. We also remark that given a distribution with a Lipschitz score, a family of distributions with Lipschitz score functions can  be generated via exponential tilting. Indeed, if a density $p_0$ has a Lipschitz score function, for any measurable function $s:\R \to \R$ with bounded second derivative, the family of densities
	\[
	p_{\vartheta} (x) := \frac{e^{s(x)\vartheta}p_0(x)}{\int_y e^{s(y)\vartheta}p_0(y) \, dy},
	\]
	where $\vartheta \in \R$ is such that the denominator above is finite, has a Lipschitz score.
	Nevertheless, as discussed following Theorem~\ref{thm:hrate}, densities with fast tail decay may fail to have Lipschitz scores, and in particular densities with bounded support will necessarily have scores that blow up approaching the boundary points.
	
	One straightforward implication of such sub-Gaussianity is that the moments of the score are bounded. More importantly however, this shows for example that the expectation of the exponential of the score is finite.
	As the following result shows, this quantity is particularly useful for bounding the expectation of a ratio of a density and a version shifted by a given amount.
	\begin{lemma}
		\label{lem:densityratiobound}
		If $p$ is a twice differentiable density on $\R$ with score $\rho$ defined everywhere such that $ \sup_{x\in\R} |\rho'(x)| \leq C$, then for any $x,u\in\R$ such that $p(x)>0$, \begin{equation*}
			\exp\bigg\{u\rho(x) - \frac{u^2}{2} C\bigg\} \leq \frac{p(x+u)}{p(x)} \leq \exp\bigg\{u\rho(x) + \frac{u^2}{2} C\bigg\}.
		\end{equation*}
	\end{lemma}
	\begin{proof}
		The inequality is proved via a Taylor expansion on $\log p(x+u)$ around $u=0$. Indeed,
		\begin{equation*}
			\log p(x+u) = \log p(x) + u \rho(x) + \frac{u^2}{2} \rho'(\eta)
		\end{equation*}
		for some $\eta\in[x-|u|,x+|u|]$. Rearranging and taking absolute values gives the bound
		\begin{equation*}
			\bigg| \log \bigg(\frac{p(x+u)}{p(x)}\bigg) 
			- u \rho(x) \bigg| \leq    
			\frac{u^2}{2} C.
		\end{equation*}
		Since the exponential function is increasing, this suffices to prove the claim.
	\end{proof}
	This indicates that while score estimation may appear to be highly delicate given that it involves the derivative and inverse of a density, it does in fact have a certain intrinsic robustness.
	Estimation of the score based on data corrupted by a perturbation, which in our case here would be our estimates of the errors in the location--scale model, can still yield estimates whose quality is somewhat comparable with those obtained using the original uncorrupted data, as Theorems~\ref{thm:knownscoreestimation}, \ref{thm:subgscoreestimation} and \ref{thm:homscoreestimation} to follow formalise.
	
	Another useful consequence of Lemma~\ref{lem:densityratiobound} is that when random variable $X$ has a Lipschitz score, for an arbitrary non-negative measurable function $h$ satisfying $\E [h(X)^{1+\delta}] < \infty$ for some $\delta > 0 $, we may obtain a bound on $\E [h(X + u)]$ as follows:
	\begin{align*}
		\E [h(X + u)]  &= \int_{-\infty}^\infty h(x+u) p(x) \, dx  \\
		&= \int_{-\infty}^\infty h(x) \frac{p(x-u)}{p(x)} p(x) \, dx \\
		&\leq \E \big[ h(X) \exp\big(-u \rho(X) + u^2 C / 2 \big) \big] \\
		&\leq \left(\E [h(X)^{1+\delta}]\right)^{\frac{1}{1+\delta}} \, \exp\big( 3(\delta^{-1} + 1 ) u^2C / 2 \big), 
	\end{align*}
	using H\"older's inequality and then Theorem~\ref{thm:scoresubg} and a bound on the moment generating function for sub-Gaussian random variables. This sort of bound turns out to be useful in the proof of Theorem~\ref{thm:hrate} when we wish to control the estimation error of the convolved regression function estimate $\hat{f}^{(n,1)}$ without appealing to any smoothness in $\tilde{f}^{(n,1)}$.

	\subsection{Estimation for location--scale families}
	\label{sect:scorelsfamilies}
	In this section, we consider  a location--scale model for $X$ on $Z$.
	Our goal is to reduce the conditions of Theorem~\ref{thm:clt} to conditions based on regression, scale estimation and univariate score estimation alone. The former two tasks are more familiar to analysts and amenable to the full variety of flexible regression methods that are available.

	We assume that we have access to a dataset $D^{(n)} := \{ (x_i,z_i) \;:\; i=1,\ldots,n\}$ of $n$ i.i.d.\ observations, with which to estimate the score.
	Let us write $\mathcal{P}_{\text{ls}}$ for the class of location--scale models of the form \begin{equation}
		X = m_P(Z) + \sigma_P(Z)\varepsilon_P, \label{eqn:locscale}
	\end{equation}
	where $\varepsilon_P$ is mean-zero with unit variance and independent of $Z$, both $m_P(Z)$ and $\sigma_P(Z)$ are square-integrable, and $\varepsilon_P$ has a differentiable density on $\R$. This enables us to reduce the problem of estimating the score of $X\given  Z$ to that of estimating the score function of the univariate variable $\varepsilon_P$ alone. Note that $\sigma_P^2(Z) = \Var_P(X \given Z)$.

	We denote the density and score function (under $P$) of the residual $\varepsilon_P$ by $p_{e}$ and $\rho_{e}$ respectively. Using these we may write the conditional density and score function as \begin{align*}
		p_P(x\given  z) &= p_{e}\bigg(\frac{x-m_P(z)}{\sigma_P(z)}\bigg)\\
		\rho_P(x, z) &= \frac{1}{\sigma_P(z)}\; \rho_{e}\bigg(\frac{x-m_P(z)}{\sigma_P(z)}\bigg).
	\end{align*}
	Recall that Theorem~\ref{thm:hrate} relies on a Lipschitz assumption for $\rho_P(x, z)$. Under the location--scale assumption above, this property will be satisfied provided $\inf_{z} \sigma_P(z)$ is bounded away from zero and $\rho_e$ has bounded derivative.
	
	Given conditional mean and (non-negative) scale estimates $\hat m^{(n)}$ and $\hat\sigma^{(n)}$,
	trained on $D^{(n)}$,
	define estimated residuals
	\begin{align*}
		\hat\varepsilon^{(n)} &:= \frac{X-\hat m^{(n)}(Z)}{\hat\sigma^{(n)}(Z)}\\
		&= \frac{\sigma_P(Z)\;\varepsilon_P+m_P(Z)-\hat m^{(n)}(Z)}{\hat\sigma^{(n)}(Z)}.
	\end{align*}
We will use the estimated residuals $\hat\varepsilon^{(n)}$ to construct a (univariate) residual score estimator $\hat \rho^{(n)}_{\hat\varepsilon}$, also
trained on $D^{(n)}$, which we combine into  a final estimate of $\rho_P(x,z)$:
\begin{equation} \label{eq:rho_est}
	\hat\rho^{(n)}(x,z) = \frac{1}{\hat \sigma^{(n)}(z)}\; \hat \rho_{\hat\varepsilon}^{(n)}\bigg(\frac{x-\hat m^{(n)}(z)}{\hat\sigma^{(n)}(z)}\bigg).
\end{equation}
While there are a variety of univariate score estimators available (see Section~\ref{sec:intro}), these will naturally have been studied in settings when supplied with  i.i.d.\ data from the distribution whose associated score we wish to estimate. Our setting here is rather  different in that we  wish to apply such a technique to an estimated set of residuals. In order to study how existing performance guarantees for score estimation may be translated to our setting, we let $p_{\hat{\varepsilon}}$ be the density of the distribution of $\hat\varepsilon^{(n)}$, conditional on $D^{(n)}$, and write $\rho_{\hat{\varepsilon}}(\epsilon) := p'_{\hat{\varepsilon}}(\epsilon) / p_{\hat{\varepsilon}}(\epsilon)$ for the associated score function. With this, let us define the following quantities,
	which are random over the sampling of $D^{(n)}\sim P$.
\begin{align*}
	A_\rho^{(n)} &:= \E_P\Big[\big\{\rho_{P}(X,Z)-\hat \rho^{(n)}(X,Z)\big\}^2 \; \Big| \; D^{(n)}\Big],\\
	A_m^{(n)} &:= \E_P\Bigg[\bigg\{\frac{m_P(Z) - \hat m^{(n)}(Z)}{\sigma_P(Z)}\bigg\}^2 \;\Bigg|\; D^{(n)}\Bigg],\\
	A_\sigma^{(n)} &:= \E_P\Bigg[\bigg\{\frac{\sigma_P(Z) - \hat \sigma^{(n)}(Z)}{\sigma_P(Z)}\bigg\}^2 \;\Bigg|\; D^{(n)}\Bigg],\\
	A_{\hat\varepsilon}^{(n)} &:= \E_P\Big[\big\{\rho_{\hat\varepsilon}(\hat\varepsilon^{(n)}) - \hat\rho_{\hat\varepsilon}^{(n)}\big(\hat\varepsilon^{(n)}\big)\big\}^2\;\Big|\; D^{(n)}\Big].
\end{align*}
The first quantity $A_\rho^{(n)}$ is what we ultimately seek to bound to satisfy the requirements of Theorem~\ref{thm:clt}. The final quantity $A_{\hat\varepsilon}^{(n)}$ is the sort of mean squared error we might expect to have guarantees on: note that this is evaluated with respect to the distribution of $\hat\varepsilon^{(n)}$, from which we can access samples. Indeed, \citet{wibisono2024optimal} provide a score estimation procedure that when $\hat{\varepsilon}^{(n)}$ is sub-Gaussian with parameter bounded by $v$ say, and its density is $(L, s)$-H\"older continuous,  $A_{\hat\varepsilon}^{(n)} = O_P(n^{-\frac{2s}{3+2s}} \text{polylog}(n))$.

We will assume throughout that $\sigma_P(z)$ is bounded away from zero for all $z \in \mathcal{Z}$, so $	A_m^{(n)}$ and $A_\sigma^{(n)}$ may be bounded above by multiples of their counterparts unscaled by $\sigma_P(Z)$. The former versions however allow for the estimation of $m_P$ and $\sigma_P$ to be poorer in regions where $\sigma_P$ is large. We also introduce the quantities
	\begin{equation*}
	u^{(n)}_\sigma(z) := \frac{\hat \sigma^{(n)}(z) - \sigma_P(z)}{\sigma_P(z)};\quad 
	u^{(n)}_m(z) := \frac{\hat m^{(n)}(z) - m_P(z)}{\sigma_P(z)},
\end{equation*}
which will feature in our conditions in the results to follow. Before considering the case where $\rho_{e}$ may be estimated nonparametrically, we first consider the case where the distribution of $\varepsilon_P$ is known.


	\subsubsection{Known family}
	
	The simplest setting is where the distribution of $\varepsilon_P$, and hence the function $\rho_{e}$, is known.
	
	\begin{theorem}
		\label{thm:knownscoreestimation}
		Let $\mathcal{P}\subset\mathcal{P}_{\text{ls}}$ be such that all of the following hold. Under each $P\in\mathcal{P}$, the location--scale model \eqref{eqn:locscale} holds with $\varepsilon_P \overset{d}{=} \varepsilon$ fixed such that the density of $\varepsilon$ is twice differentiable 
	and the corresponding score satisfies $\sup_{\epsilon\in\R} |\rho_{e}'(\epsilon)| < \infty$ and $\E[\rho_{e}^2(\varepsilon)]<\infty$.
		Assume $\inf_{P\in\mathcal{P}} \inf_{z\in\mathcal Z} \sigma_P(z)>0$
		and the ratio $\sigma_P / \hat\sigma^{(n)}$ and the regression error $u_m^{(n)}$ are bounded with high probability:
		\begin{equation} \label{eq:sigma_m_cond}
			\sup_{z\in\mathcal Z}\frac{\sigma_P(z)}{\hat\sigma^{(n)}(z)} = O_{\mathcal{P}}(1)  \quad \text{and} \quad \sup_{z\in\mathcal Z}\big|u_m^{(n)}(z)\big| = O_{\mathcal{P}}(1).
		\end{equation} 
		Set 
		\begin{equation*}
			\hat\rho^{(n)}(x,z) := \frac{1}{\hat \sigma^{(n)}(z)}\; \rho_{e}\bigg(\frac{x-\hat m^{(n)}(z)}{\hat\sigma^{(n)}(z)}\bigg).
		\end{equation*}
		Then \begin{equation*}
			A_\rho^{(n)} = O_{\mathcal{P}}\big(A_m^{(n)} + A_\sigma^{(n)}\big).
		\end{equation*}
	\end{theorem}
We see that in this case, the error $A_\rho^{(n)}$ we seek to control is bounded by mean squared errors in estimating $m_P$ and $\sigma_P$.

	\subsubsection{Sub-Gaussian family}
	\label{sect:subgscoreestimation}
	
	We now consider the case where $\varepsilon_P$ follows some unknown sub-Gaussian distribution and $\rho_{e}$ is Lipschitz. This for example encompasses the case where $\varepsilon_P$ has a Gaussian mixture distribution.

	\begin{theorem}
		\label{thm:subgscoreestimation}
		Let $\mathcal{P}\subset\mathcal{P}_{\text{ls}}$ and constants $C_\varepsilon, C_\rho, >0$ be such that all of the following hold. Under each $P\in\mathcal{P}$, the location--scale model \eqref{eqn:locscale} holds where $\varepsilon_P$ is sub-Gaussian with parameter $C_\varepsilon$. Furthermore the density $p_{e}$ of $\varepsilon_P$ is twice differentiable on $\R$, with $\sup_{\epsilon\in\R} |\rho_{e}'(\epsilon)|\leq C_\rho$
		and $p_{e}'$ and $ p_{e}''$ are both bounded.
		Assume $\inf_{P\in\mathcal{P}} \inf_{z\in\mathcal Z} \sigma_P(z)>0$, \eqref{eq:sigma_m_cond} holds and
		with high probability $D^{(n)}$ is such that the
		scale estimation error $u_\sigma^{(n)}$ is bounded in the sense that
		\begin{equation} \label{eq:u_cond}
			\lim_{n\to\infty} \sup_{P\in\mathcal{P}} \Pr_P\Big(\sup_{z\in\mathcal Z}\big|u_\sigma^{(n)}(z)\big| >  \{18 \sqrt{C_{\rho} }C_{\varepsilon}\}^{-1} \Big) = 0.
		\end{equation}
		Then \begin{equation*}
			A_\rho^{(n)} = O_{\mathcal{P}} \big(A_m^{(n)} + A_\sigma^{(n)} + A_{\hat\varepsilon}^{(n)}\big).
		\end{equation*}
	\end{theorem}
We see that in addition to requiring that $A_m^{(n)}$ and $A_\sigma^{(n)}$ are well-controlled, the mean squared error $ A_{\hat\varepsilon}^{(n)}$ associated with the univariate  score estimation problem also features in the upper bound.
We also have a  condition \eqref{eq:u_cond} on $\sup_{z}\big|u_\sigma^{(n)}(z)\big|$ that is stronger than $\sup_{z}\big|u_\sigma^{(n)}(z)\big| = O_{\mathcal{P}}(1)$, but weaker than $\sup_{z}\big|u_\sigma^{(n)}(z) \big| = o_{\mathcal{P}}(1)$.

	\subsection{Estimation for location families}
	\label{sec:location}
	Theorem \ref{thm:subgscoreestimation} assumes that $\varepsilon_P$ is sub-Gaussian, which we use to deal with the scale estimation error $u_\sigma^{(n)}$. If $X$ only depends on $Z$ through its location (i.e. $\sigma_P$ is constant) then the same proof approach works even for heavy-tailed $\varepsilon_P$. Consider the location only model
	\begin{equation}
		X = m_P(Z) + \varepsilon_P, \label{eqn:location}
	\end{equation}
	where $\varepsilon_P$ is independent of $Z$, $m_P(Z)$ is square-integrable, and $\varepsilon_P$ has a differentiable density on $\R$. Compared to Section~\ref{sect:subgscoreestimation}, we have assumed $\sigma_P$ does not depend on $Z$, and have relabelled $\sigma_P\varepsilon_P\mapsto \varepsilon_P$. We fix $\hat\sigma^{(n)}(z) \equiv 1$.

	\begin{theorem}
		\label{thm:homscoreestimation}
		Let $\mathcal{P}\subset\mathcal{P}_{\text{ls}}$ and uniform constant $C_\rho>0$ be such that all of the following hold. Under each $P\in\mathcal{P}$, the location model \eqref{eqn:location} holds. Assume furthermore that the density $p_{e}$ of $\varepsilon_P$ is twice differentiable on $\R$, with
		$\sup_{\epsilon\in\R} |\rho_{e}'(\epsilon)|\leq C_\rho$, $p_{e}'$ and $ p_{e}''$ both bounded, and $\sup_{z}\big|u_m^{(n)}(z)\big| = O_{\mathcal{P}}(1)$.
		Then \begin{equation*}
			A_\rho^{(n)} = O_{\mathcal{P}} \big(A_m^{(n)} + A_{\hat\varepsilon}^{(n)}\big).
		\end{equation*}
	\end{theorem}
As is to be expected, compared to Theorem~\ref{thm:subgscoreestimation}, here there is no $A_\sigma^{(n)}$ term in the upper bound on $A_\rho^{(n)}$.
		Note that $\varepsilon_P$ need not have any finite moments and $\varepsilon_P$ may follow a Cauchy distribution, for example.

\subsection{Local location--scale modelling} \label{sec:local}
	While we have seen that a location--scale assumption can make the problem of conditional score estimation much more tractable in the sense that the predictive power of modern flexible regression techniques can be brought to bear on the problem, the location--scale model is unlikely to hold precisely in practice. For large dataset sizes, this may result in a non-negligible bias in the final estimator. One possible approach to mitigating this issue involves only aiming to make use of the location--scale model locally, as we now briefly outline.
	
	Were the location--scale assumption to hold, we would expect $\hat{\varepsilon}^{(n)}$ and $Z$ to be independent; were it to hold only approximately, these should be close to independent. Furthermore, we might expect there to be a partition of $\mathcal{Z}$ into regions $R_1,\ldots,R_J$ such that we approximately have $\hat{\varepsilon}^{(n)} \independent Z \given \{Z \in R_j\}$ for each $j$. Trees constructed using a two-sample tests as a split criterion such as those that form the basis of distributional random forests \citep{cevid2022distributional} are specifically designed to seek out such partitions. One may then apply univariate score estimation separately to those estimated scaled residuals whose associated $Z$ values lie in a given region to arrive at a final conditional score estimate. We however leave further investigation of this and related approaches to future work.
	
	\section{Numerical experiments}
	\label{sect:numerical}
	
	We demonstrate that confidence intervals derived from the cross-fitted, doubly robust average partial effect estimator (\ref{eqn:est}) and associated variance estimator \eqref{eq:Sigma_est} constructed using the approaches of Sections~\ref{sect:resmooth} and \ref{sect:score} is able to maintain good coverage across a range of settings. As competing methods, we consider a version of (\ref{eqn:est}) using a simple numerical difference for the derivative estimate (as suggested in \citet[\S S5.2]{ChernozhukovRiesz}) and a quadratic basis approach for score estimation (similar to \cite{RothenhauslerICE}); the method of \cite{RothenhauslerICE}; and the doubly-robust partially linear regression (PLR) of \citet[\S 4.1]{ChernozhukovTreatment}. Theorem~\ref{thm:basisgaussian} suggests that the basis approaches of \citet[\S 2]{ChernozhukovRiesz} and \cite{RothenhauslerICE} are similar, and since the latter is easier to implement we use this as a reasonable proxy for the approach of \citet{ChernozhukovRiesz}. While our estimator may be used with any plug-in machine learning regression, here we make use of XGBoost \citep{xgboost} for its good predictive power, perform scale estimation via decision trees so that our estimates are bounded away from zero, and perform unconditional score estimation via a penalised smoothing spline \citep{CoxSpline, NgSpline, NgSplineImplementation}, which has the attractive property of smoothing towards a Gaussian in the sense of \citet[Thm.~4]{CoxSpline}. The precise implementation details are given in Section~\ref{sect:simulationimplementation}. We note in particular that as XGBoost is based on decision trees, it returns a piecewise constant regression surface.
	
	For a sanity check we also include the ordinary least squares (OLS), which is expected to do very poorly in general. Code to reproduce our experiments is available from \url{https://github.com/harveyklyne/drape}.

	\subsection{Settings}
	
	In all cases we generate $Y = f_P(X,Z) + N(0,1)$ using a known regression function $f_P$ and predictor distribution $(X,Z)\sim P$, so that we may compute the target parameter $\theta_P = \E_P[f'_P(X,Z)]$ to any degree of accuracy using Monte Carlo. The predictors $(X,Z) \in \R \times \R^p$ are either generated synthetically from a location--scale family or taken from a real data set.
	
	\subsubsection{Location--scale families}
	For these fully simulated settings, we fix $n=1000$ and 
	\begin{align*}
		Z &\sim N(0, \Sigma) \in \R^9 \text{ where } \Sigma_{jj}=1 \text{, } \Sigma_{jk} = 0.5 \text{ for } j\neq k;\\
		X &= m_P(Z) + \sigma_P(Z)\varepsilon_P,
	\end{align*}
	for the following choices of $m_P$, $\sigma_P$, $\varepsilon_P$. We use two step functions for $m_P, \sigma_P:\mathcal{Z}\to \R$.
	\begin{align*}
		m_{P}(z) &= \ind_{(0, \infty)}(z_1)\\
		\sigma_{P}(z) &= \begin{cases}
			\sqrt{\frac32} & \text{ if } z_3 < 0;\\
			\frac1{\sqrt2} & \text{ if } z_3 \geq 0.
		\end{cases}
	\end{align*}
	Note that $\E_P\big[\sigma_{P}^2(Z)\big]=1$. We use the following options for the noise $\varepsilon_P$.
	\begin{align}
		\varepsilon_{\text{norm}} &= N(0,1) \label{eqn:normsetting}\\
		\varepsilon_{\text{mix2}} &= N\bigg(\pm \frac{1}{\sqrt2}, \frac{1}{2}\bigg) \text{ equiprobably}; \label{eqn:mix2setting}\\
		\varepsilon_{\text{mix3}} &= N\bigg(\pm \frac{\sqrt3}{\sqrt2}, \frac{1}{3}\bigg) \text{ equiprobably}; \label{eqn:mix3setting}\\
		\varepsilon_{\text{log}} &= \text{Logistic} \bigg( 0, \frac{\sqrt3}{\pi} \bigg); \label{eqn:logsetting}\\
		\varepsilon_{\text{t4}} &= \frac1{\sqrt2} t_4. \label{eqn:t4setting}
	\end{align}
	In all cases $\varepsilon_P$ is independent of $(X,Z)$, and has zero mean and unit variance. Since $\sigma_P$ is not constant, the heavy-tailed settings $\varepsilon_{\text{log}}, \varepsilon_{\text{t4}}$ are not covered by the results in Section~\ref{sect:score}. The score functions for these random variables are plotted in Figure~\ref{fig:score}.
	
	\begin{figure}
		\centering
		\includegraphics[width=0.5\textwidth]{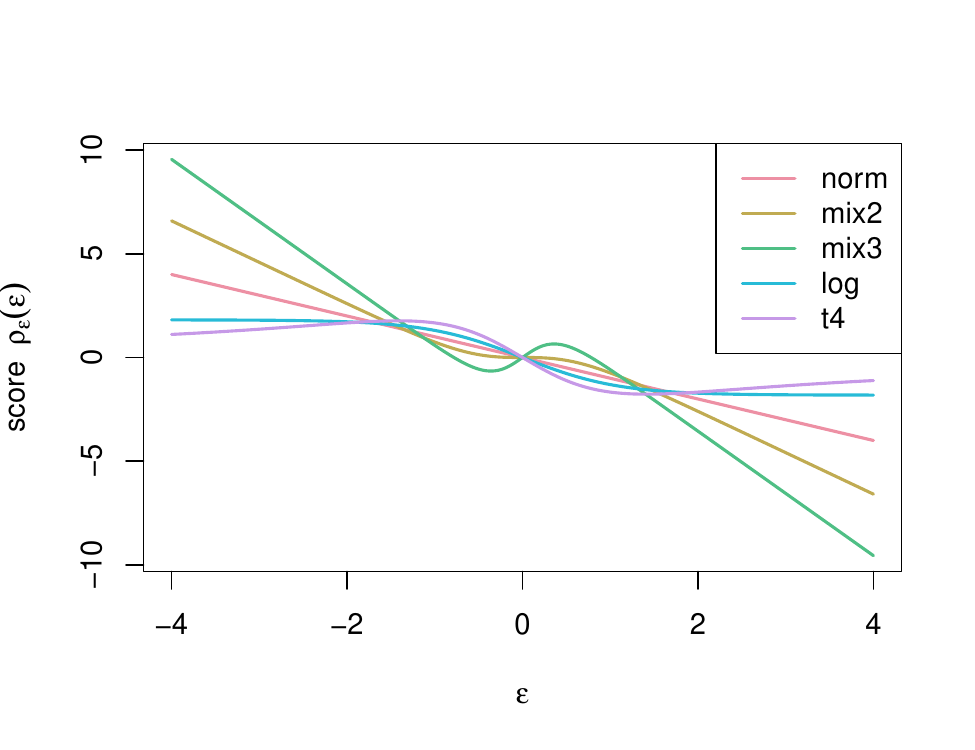}
		\caption{Score functions for the choices of distribution for $\varepsilon_P$ in our numerical experiments. All these distributions have mean zero and variance one. The pink line corresponds to a standard Gaussian \eqref{eqn:normsetting}, the gold and green lines to Gaussian mixtures \eqref{eqn:mix2setting} and \eqref{eqn:mix3setting}, the blue line to the logistic distribution \eqref{eqn:logsetting}, and the purple line to the Student's $t$ distribution with 4 degrees of freedom \eqref{eqn:t4setting}.  }
		\label{fig:score}
	\end{figure}

	\subsubsection{401k dataset}
	\label{sect:401k}
	To examine misspecification of the location--scale model for $(X,Z)$, we import the 401k data set from the \texttt{DoubleML} R package \citep{DoubleML}. We take $X$ to be the income feature, and $Z$ to be age, education, family size, marriage, two-earner household, defined benefit pension, individual retirement account, home ownership, and 401k availability, giving $p=10$. We make use of all the observations ($n=9915$), and centre and scale the predictors before generating the simulated response variables.

	\subsubsection{Simulated responses}
	
	For the choices of regression function $f_P$, first define the following sinusoidal and sigmoidal families of functions:
	\begin{align*}
		f_{\text{sino}}(u; a) &= \exp(-u^2/2)\sin(au);\\
		f_{\text{sigm}}(u; s) &= (1+\exp(-su))^{-1};
	\end{align*} for $u\in \R$, $a, s>0$. We use the following choices for $f_P:\R\times\mathcal{Z}\to \R$, giving partially linear, additive, and interaction settings:
	\begin{align}
		f_{\text{plm}}(x,z)&=x + f_{\text{sigm}}(z_2; s=1) + f_{\text{sino}}(z_2; a=1) \label{eqn:plmsetting};\\
		f_{\text{add}}(x,z)&=f_{\text{sigm}}(x; s=1) + f_{\text{sino}}(x; a=1) + f_{\text{sino}}(z_2; a=3) \label{eqn:addsetting};\\
		f_{\text{int}}(x,z)&=f_{\text{sigm}}(x; s=3) + f_{\text{sino}}(x; a=3) + f_{\text{sino}}(z_2; a=3) + x \times z_2.\label{eqn:intsetting}
	\end{align}
	
	\subsection{Results}
	
	We examine the coverage of the confidence intervals for each of the 5 methods in each of the 18 settings described. Figures~\ref{fig:plm}, \ref{fig:add}, and \ref{fig:int} show nominal 95\% confidence intervals from each of 1000 repeated experiments. We find that our method achieves at least 87\% coverage in each of the 18 settings trialed and in most settings coverage is close to the desired 95\%. Moreover, the estimator appears to be unbiased in most cases, with those confidence intervals not covering the parameter being equally likely to fall above or below the true parameter. In additional results which we do not include here, we find that our multivariate score estimation procedure reduces the bias as compared to the high-dimensional basis approach, and our resmoothing reduces the variance compared to numerical differencing. Taken together, our proposed estimator performs well in all settings considered.
	
	The numerical difference and quadratic basis approach also performs reasonably well: it typically has a slightly lower variance than the resmoothing and spline-based approach, as evidenced by the slightly shorter median confidence interval widths in the additive and partially linear model examples. However, this comes at the expense of introducing noticeable bias in most settings. Indeed, the bulk of the confidence intervals failing to cover the true parameter often lie on one side of it.
	 The confidence intervals also tend to undercover, in the worst case achieving a coverage of 78\%. In some of the settings where the error distribution has heavier tails, a proportion of the confidence intervals produced are very wide (extending far beyond the margins of the plot).
	 
	 As one would expect, the doubly-robust partially linear regression does very well when the partially linear model is correctly specified \eqref{eqn:plmsetting}, achieving full coverage with narrow confidence intervals, but risks completely losing coverage when the response is non-linear in $X$. Interestingly, the quadratic basis approach of \citet{RothenhauslerICE} displayed a similar tendency. As expected, the ordinary least squares approach did not achieve close to the specified coverage in any setting.

	
	\begin{figure}
		\centering
		\includegraphics[width=\textwidth]{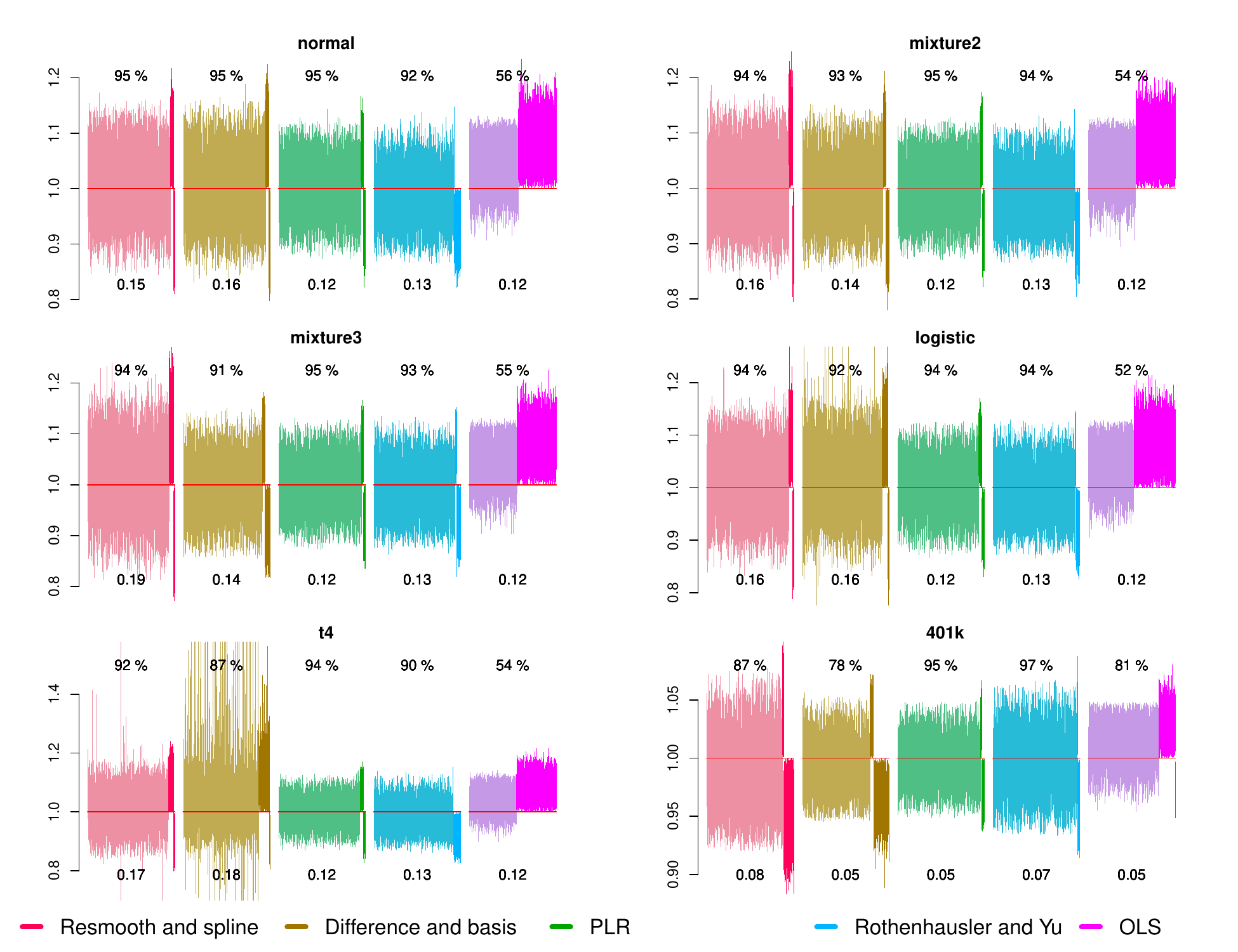}
		\caption{Estimated confidence intervals from the partially linear model experiment \eqref{eqn:plmsetting}. The subplots correspond to different settings for predictor $(X,Z)$ generation. The different colours refer to the different methods. The red horizontal lines correspond to the population level target parameter. The percentages above each subplot refer to the achieved coverage (specified level 95\%), and the decimals below give the median confidence interval width. In the bottom left subplot --- setting \eqref{eqn:t4setting} --- some of the confidence intervals for the numerical difference and basis score approach extend significantly beyond the plotting limits.}
		\label{fig:plm}
	\end{figure}
	
	\begin{figure}
		\centering
		\includegraphics[width=\textwidth]{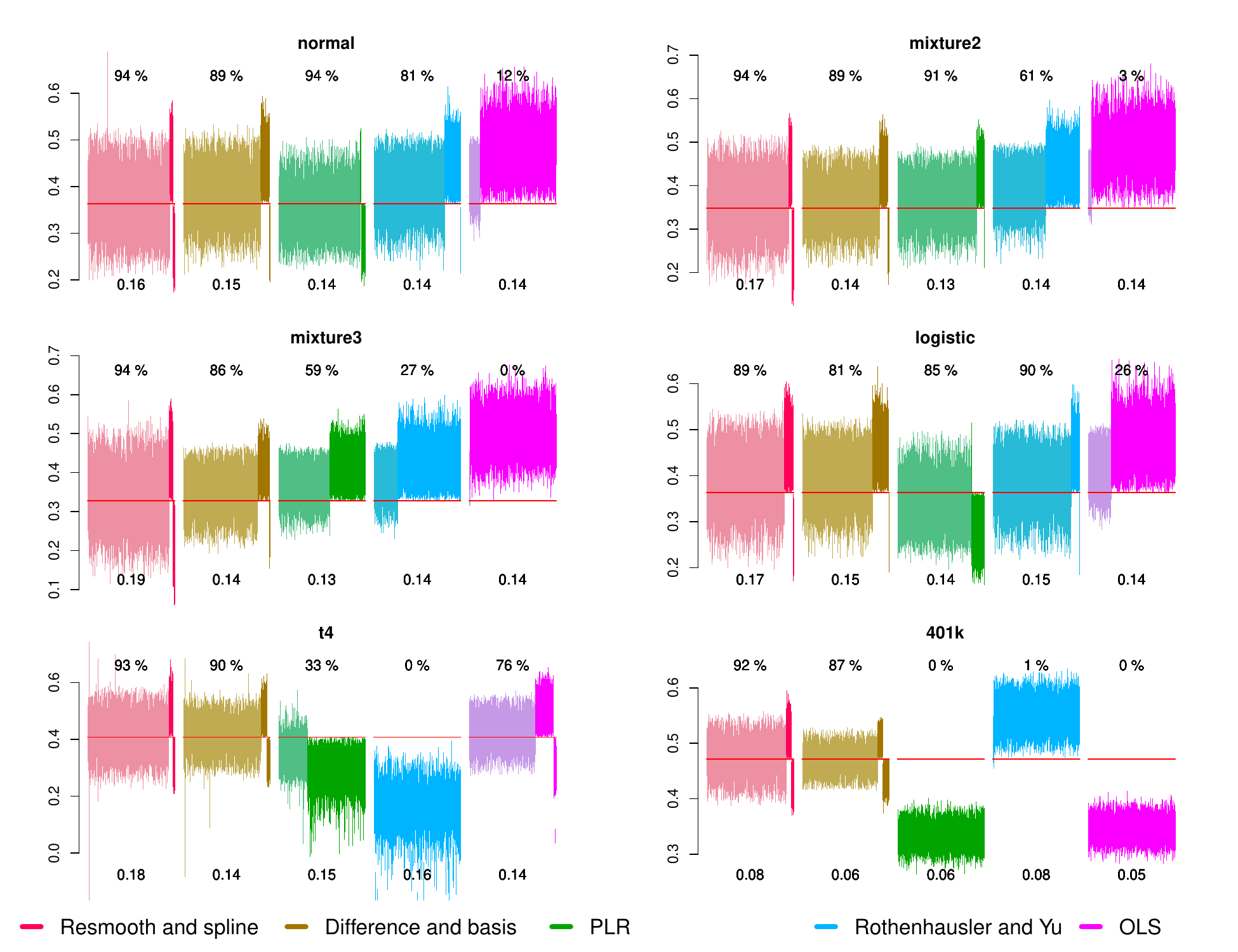}
		\caption{Estimated confidence intervals from the additive model experiment \eqref{eqn:addsetting}. The subplots correspond to different settings for predictor $(X,Z)$ generation. The different colours refer to the different methods. The red horizontal lines correspond to the population level target parameter. The percentages above each subplot refer to the achieved coverage (specified level 95\%), and the decimals below give the median confidence interval width. }
		\label{fig:add}
	\end{figure}
	
	\begin{figure}
		\centering
		\includegraphics[width=\textwidth]{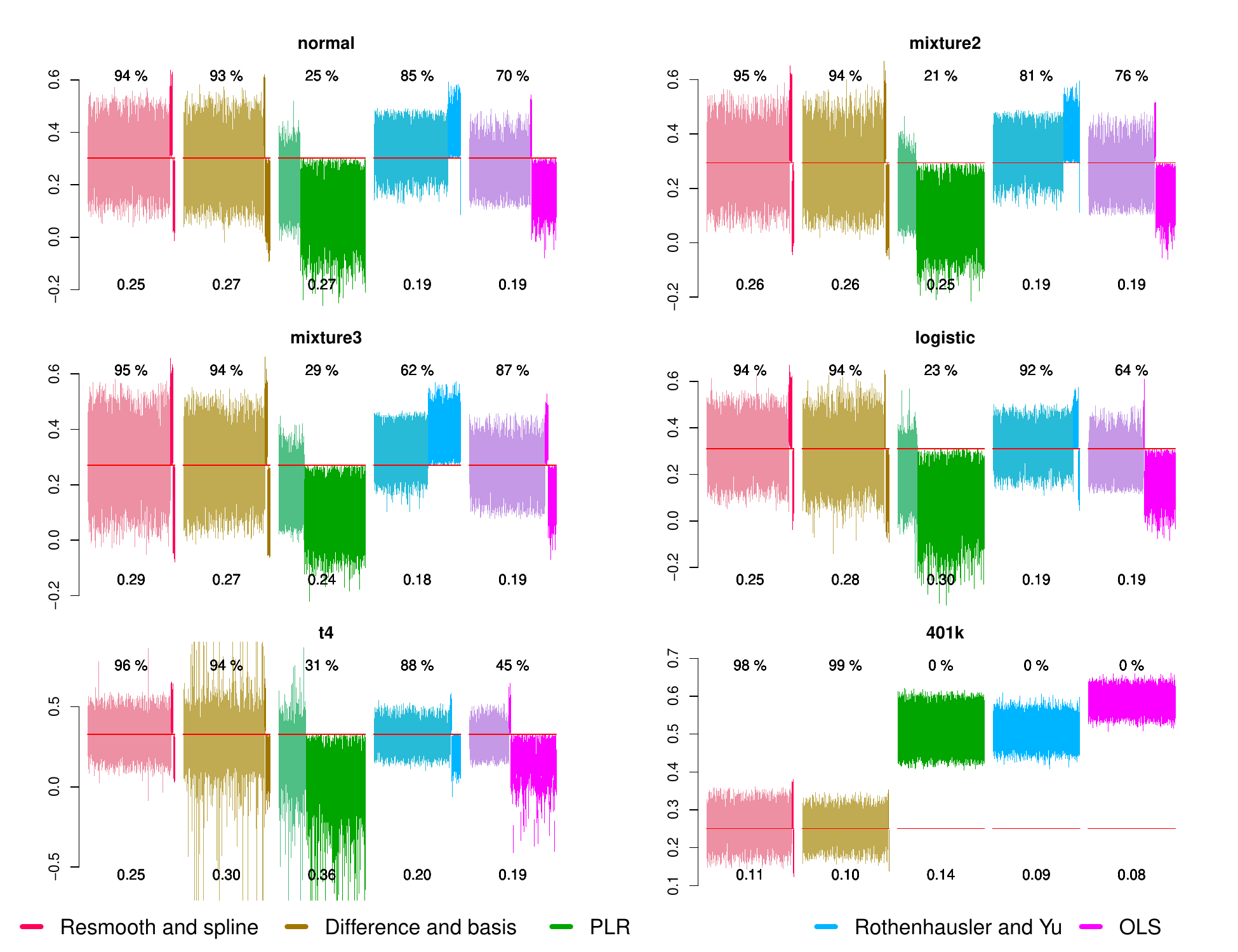}
		\caption{Estimated confidence intervals from the interaction model experiment \eqref{eqn:intsetting}. The subplots correspond to different settings for predictor $(X,Z)$ generation. The different colours refer to the different methods. The red horizontal lines correspond to the population level target parameter. The percentages above each subplot refer to the achieved coverage (specified level 95\%), and the decimals below give the median confidence interval width. In the bottom left subplot, setting \eqref{eqn:t4setting}, some of the confidence intervals for the numerical difference and basis score approach extend significantly beyond the plotting limits.}
		\label{fig:int}
	\end{figure}

    \subsection{401k data analysis} \label{sec:401k}
    We also apply all the methods to the 401k data set with the response variable taken as net financial assets (as opposed to a simulated response variable in Section~\ref{sect:401k}). We estimate the average partial effect of income on net financial assets, and find that all methods agree that the effect is significantly positive but disagree on the magnitude. The methods which assume a linear relationship (OLS and PLR) produce substantially higher estimates than the non-linear models. In the light of our theoretical and numerical results, we view this as suggesting that the partial effect of income on net financial assets may be non-linear. The estimates (standard errors) of each method are as follows: Resmooth and spline 0.46 (0.03); Difference and basis 0.39 (0.05); PLR 0.86 (0.11); \cite{RothenhauslerICE} 0.27 (0.07); OLS 0.93 (0.03).

	\section{Discussion}
	\label{sect:discussion}
	
	The average partial effect is of interest in nonparametric regression settings, giving a single parameter summary of the effect of a predictor. In this work we have suggested a framework to enable the use of arbitrary machine learning regression procedures when doing inference on average partial effects. We propose kernel resmoothing of a first-stage machine learning method to yield a new, differentiable regression estimate. Theorem~\ref{thm:hrate} demonstrates the attractive properties of this approach for a range of kernel bandwidths. We further advocate location--scale modelling for multivariate (conditional) score estimation, which we prove reduces this challenging problem to the better studied, univariate case (Theorems~\ref{thm:knownscoreestimation}, \ref{thm:subgscoreestimation}, and \ref{thm:homscoreestimation}). Our proofs rely on a novel result of independent interest: that Lipschitz score functions yield sub-Gaussian random variables (Theorem~\ref{thm:scoresubg}).
	
	Our work suggests several directions for further research. In Theorem~\ref{thm:clt}, the dimension $d$ of $X$ is considered fixed. It would be of interest to extend this to permit $d$ increasing with $n$. In such a setting, a Gaussian multiplier bootstrap approach (for example, as employed in Section~3.2 of \citet{ShahGCM}) should allow for simultaneous confidence intervals for all components of $\theta_P$. A It would also  be of interest to develop theoretical guarantees for the approach outlined in Section~\ref{sec:local} on using a local location--scale model.
	
	\subsection*{Acknowledgements}
	HK was supported by an EPSRC studentship and RDS was supported in part by the EPSRC programme grant EP/N031938/1.  The authors are grateful to three referees whose constructive comments helped to improve the paper.
	
	
	\bibliographystyle{apalike}	
	\bibliography{references}
	
	\appendix
	
	\newpage
	\section{Proofs in Section~\ref{sect:general}}
	
	\subsection{Proof of Proposition~\ref{prop:intbyparts}}
	
	\begin{proof}
		Fix some $(x_{-j},z) \in \R^{d-1} \times \mathcal{Z}$ where $z$ has positive marginal density. By the product rule,
		\begin{align*}
			\nabla_j \big(g(x,z)p_P(x\given  z)\big) &= \nabla_j g(x,z)p_P(x\given  z) + g(x,z)\nabla_j p_P(x\given  z)\\
			&= \nabla_j g(x,z)p_P(x\given  z) + \rho_{P,j}(x,z)g(x,z)p_P(x\given  z).
		\end{align*}
		Therefore writing
		\[
		q_{a,b}(x, z):= \{\nabla_j g(x,z)p_P(x\given  z) + \rho_{P,j}(x,z)g(x,z)p_P(x\given  z)\}\ind_{[a,b]}(x_j),
		\]
		for any $-\infty<a<b<\infty$ and noting that for almost every $(x_{-j}, z)$ the above must be integrable over $x_j$ by \eqref{eq:intbyparts}, we have by the Fundamental Theorem of Calculus that for such $(x_{-j}, z)$,
		\begin{align*}
			\int_{-\infty}^\infty q_{a,b}(x, z) \; dx_j &= \int_a^b \nabla_j \big(g(x,z)p_P(x\given  z)\big)\; dx_j\\
			& = g(b,x_{-j}, z)p_P(b, x_{-j}\given  z)-g(a,x_{-j},z)p_P(a, x_{-j}\given  z).
		\end{align*}   
Let sequences $a_n(x_{-j}, z)$ and $b_n(x_{-j}, z)$ satisfy the assumptions of the result. By the dominated convergence theorem we have that for almost every $(x_{-j}, z)$,
\begin{align*}
	& \int_{-\infty}^\infty \{ \nabla_j g(x,z)p_P(x \given  z) + \rho_{P,j}(x,z)g(x,z)p_P(x \given  z)\} \, dx_j  \\
	=&  \int_{-\infty}^\infty \lim_{n \to \infty}q_{a_n(x_{-j},z), b_n(x_{-j},z)} \, dx_j \\
	=&\lim_{n \to \infty} \int_{-\infty}^\infty q_{a_n(x_{-j},z), b_n(x_{-j},z)} \, dx_j  = 0.
\end{align*}
Now, by Fubini's theorem, writing $\mu$ for the marginal distribution of $Z$,
\begin{align*}
				&\E_P[\nabla_j g(X,Z) + \rho_{P,j}(X,Z)g(X,Z)] \\
				=& \int_{\mathcal{Z}} \int_{\R^{d-1}} \int_{\R} \{ \nabla_j g(x,z)p_P(x \given  z) + \rho_{P,j}(x,z)g(x,z)p_P(x \given  z)\} \, dx_j \, dx_{-j} \, d\mu(z) =0,
\end{align*}
as required.
	\end{proof}
	
	\subsection{Proof of Theorem~\ref{thm:clt}}
	\begin{proof}
		In an abuse of notation, we refer to the quantities \begin{align*}
			A_f^{(n,k)} := \E_P \big[\{f_P(X,Z)-\hat f^{(n,1)}(X,Z)\}^2 \; \big| \; D^{(n,1)}\big],
		\end{align*} for each fold $k=1,\ldots,K$. Each $A_f^{(n,k)}$ satisfies the same probabilistic assumptions as $A_f^{(n)} = A_f^{(n,1)}$ due to the equal partitioning and i.i.d.\ data. Likewise we define $A_\rho^{(n,k)}, E_f^{(n,k)}, E_\rho^{(n,k)}$.
		
		To show the first conclusion we first highlight the term which converges to a standard normal distribution, and then deal with the remainder. Note that the lower bound on the minimum eigenvalue of $\Sigma_P$ corresponds to an upper bound on the maximal eigenvalue of $(\Sigma_P)^{-1/2}$. Denote the random noise in $Y$ as\begin{align*}
			\xi_P &= Y - f_P(X,Z);\\
			\xi_{P,i} &= y_i - f_P(x_i, z_i),
		\end{align*}
		so that $\E_P(\xi_P\given  X,Z) = 0$.
		
		With these preliminaries, we have
		\begin{equation*}
			\sqrt{n}(\Sigma_P)^{-1/2}\Big(\hat\theta^{(n)}-\theta_P\Big) = \frac{1}{\sqrt{n}}\sum_{i=1}^n (\Sigma_P)^{-1/2} \psi_P(y_i,x_i,z_i) + (\Sigma_P)^{-1/2} \sum_{k=1}^K R_P^{(n,k)},
		\end{equation*} where the uniform central limit theorem (Lemma~\ref{lem:uniformCLT}) applies to the first term and \begin{equation*}
			R_P^{(n,k)} := \frac{1}{\sqrt{n}} \sum_{i\in I^{(n,k)}}  \nabla \hat f^{(n,k)}(x_i,z_i) - \hat\rho^{(n,k)}(x_i,z_i) \{y_i-\hat f^{(n,k)}(x_i,z_i)\}  - \nabla f_P(x_i,z_i) + \rho_P(x_i,z_i)\xi_{P,i}.
		\end{equation*}
		Note that, conditionally on $D^{(n,k)}$, each summand of $R_P^{(n,k)}$ is i.i.d. To show that $R_P^{(n,k)}=o_{\mathcal{P}}(1)$, we fix some element $j\in\{1,\ldots, d\}$ and decompose \begin{equation}
			\label{eqn:psidecomposition}
			R_{P,j}^{(n,k)} = a^{(n,k)} - b_f^{(n,k)} + b_\rho^{(n,k)},
		\end{equation}
		where
		\begin{align*}
			a^{(n,k)}&:= \frac{1}{\sqrt{n}} \sum_{i\in I^{(n,k)}} \{\rho_{P,j}(x_i,z_i)-\hat\rho_j^{(n,k)}(x_i,z_i)\}\{f_P(x_i,z_i)-\hat f^{(n,k)}(x_i,z_i)\};\\
			b_f^{(n,k)}&:= \frac{1}{\sqrt{n}} \sum_{i\in I^{(n,k)}} [\nabla_j f_P(x_i,z_i)-\nabla_j \hat f^{(n,k)}(x_i,z_i) +\rho_{P,j}(x_i,z_i)\{f_P(x_i,z_i)-\hat f^{(n,k)}(x_i,z_i)\}];\\
			b_\rho^{(n,k)}&:= \frac{1}{\sqrt{n}} \sum_{i\in I^{(n,k)}} \{\rho_{P,j}(x_i,z_i)-\hat\rho_j^{(n,k)}(x_i,z_i)\} \xi_{P,i}.
		\end{align*}
		We now show that each term is $o_\mathcal{P}(1)$, so Lemma \ref{lem:uniformSlutsky} yields the first conclusion.
		
		By the Cauchy--Schwarz inequality, we have
		\begin{align*}
			\E_P[|a^{(n,k)}|\given  D^{(n,k)}] &\leq  \sqrt{n} \E_P[|\rho_{P,j}(X,Z)-\hat\rho_j^{(n,k)}(X,Z)||f_P(X,Z)-\hat f^{(n,k)}(X,Z)|\given  D^{(n,k)}]\\
			&\leq \sqrt{n A_f^{(n,k)}A_\rho^{(n,k)}} =o_{\mathcal{P}}(1),
		\end{align*}
		so $a^{(n,k)}$ is $o_{\mathcal{P}}(1)$ by Lemma~\ref{lem:uniformMarkov}. Note that each summand of $b^{(n,k)}_\rho$ is mean-zero conditionally on $X$ and $Z$. This means that
		\begin{align*}
			\E_P[(b^{(n,k)}_\rho)^2\given  D^{(n,k)}]&=\E_P[\{\rho_{P,j}(X,Z)-\hat \rho^{(n,k)}_j(X,Z)\}^2 \E_P(\xi_P^2 \given  X,Z)\given  D^{(n,k)}]\\
			&\leq E_\rho^{(n,k)}=o_{\mathcal{P}}(1).
		\end{align*}
		Again using Lemma~\ref{lem:uniformMarkov} we have that $b^{(n,k)}_\rho = o_\mathcal{P}(1)$.
		
		We now apply a similar argument to $b^{(n,k)}_f$, using Proposition~\ref{prop:intbyparts} to show that each summand is mean zero. Given $\epsilon >0$, noting that both $A_f^{(n,k)}$ and $E_f^{(n,k)}$ are $O_\mathcal{P}(1)$, we have there exist $M$ and $N \in \mathbb{N}$ such that for sequences of $D^{(n,k)}$-measurable events $\Omega_{P,n}$ with $\pr_P(\Omega_{P,n}) \geq 1-\epsilon$, for all $n \geq N$,
		\begin{align}
			\E_P\Big[\big|f_{P}(X,Z) - \hat f^{(n,k)}(X,Z)\big|\;\Big|\;D^{(n,k)}  \Big] \ind_{\Omega_{P,n}} &< M;\label{eqn:fl1bound}\\
			\E_P\Big[\big|\nabla_j f_{P}(X,Z) - \nabla_j \hat f^{(n,k)}(X,Z) + \rho_{P,j}\big\{f_{P}(X,Z) - \hat f^{(n,k)}(X,Z)\big\}\big|\;\Big|\;D^{(n,k)} \Big]\ind_{\Omega_{P,n}} &< M.\label{eqn:dfrhol1bound}
		\end{align}
		Now fixing $z,x_{-j}$ (where $x_{-j} \in \R^{d-1}$), we have that the function
		\begin{equation*}
			\Big(f_{P}\big((\cdot,x_{-j}),z\big) - \hat f^{(n,k)}\big((\cdot,x_{-j}),z\big)\Big)p\big((\;\cdot\;,x{_-j}) \; \big| \; z\big)
		\end{equation*}
		is continuous, where we understand $(u,x_{-j}) = (x_1,\ldots, x_{j-1},u,x_{j+1},\ldots, x_d)$ for $u\in\R$. We may therefore apply Lemma~\ref{lem:boundaryconvergent}
		to both
		\begin{align*}
			t &\mapsto \Big(f_{P}\big((t,x_{-j}),z\big) - \hat f^{(n,k)}\big((t,x_{-j}),z\big)\Big)p\big((t,x_{-j}) \;\big|\;  z\big) \\
			t &\mapsto \Big(f_{P}\big((-t,x_{-j}),z\big) - \hat f^{(n,k)}\big((-t,x_{-j}),z\big)\Big)p\big((-t,x_{-j}) \;\big|\;  z\big).
		\end{align*}
		This, in combination with \eqref{eqn:fl1bound} implies that on $\Omega_{P,n}$ and for each $n \geq N$, for almost every $(x-j, z)$, there exist $D^{(n,k)}$-measurable sequences $a_{P,m}(x_{-j},z) \to - \infty$, $b_{P,m}(x_{-j},z) \to \infty$ such that
		\begin{align}
			\lim_{m\to\infty} \bigg\{\Big(f_{P}\big((b_{P,m}(x_{-j},z),x_{-j}),z\big) - \hat f^{(n,k)}\big((b_{P,m}(x_{-j},z),x_{-j}),z\big)\Big)p\big((b_{P,m}(x_{-j},z),x_{-j}) \;\big|\;  z\big) \phantom{\bigg\}}&\nonumber \\
			- \big(f_{P}\big((a_{P,m}(x_{-j},z),x_{-j}),z\big) - \hat f^{(n,k)}\big((a_{P,m}(x_{-j},z),x_{-j}),z\big)\Big)p\big((a_{P,m}(x_{-j},z),x_{-j}) \;\big|\;  z\big)\bigg\} &= 0. \label{eqn:ferrorboundaryconvergent}
		\end{align}
		Equations (\ref{eqn:dfrhol1bound} and \ref{eqn:ferrorboundaryconvergent}) verify that we may apply Proposition~\ref{prop:intbyparts} conditionally on $D^{(n,k)}$. Therefore, for all $n$ sufficiently large,
		\begin{equation*}
			\E_P\Big[\nabla_j f_{P}(X,Z) - \nabla_j \hat f^{(n,k)}(X,Z) + \rho_{P,j}\big\{f_{P}(X,Z) - \hat f^{(n,k)}(X,Z)\big\}\;\Big|\;D^{(n,k)} \Big]\ind_{\Omega_{P,n}}\ = 0
		\end{equation*}
		and hence
		\begin{align*}
			& \E_P[(b^{(n,k)}_f)^2\given  D^{(n,k)}]\ind_{\Omega_{P,n}} \\
			=\;& \E_P([ \nabla_j f_P(X,Z)-\nabla_j \hat f^{(n,k)}(X,Z) +\rho_{P,j}(X,Z)\{f_P(X,Z)-\hat f^{(n,k)}(X,Z)\}]^2\given  D^{(n,k)}) \ind_{\Omega_{P,n}}\\
			\leq \;  & E_f^{(n,k)}.
		\end{align*}
		Now
		\[
		\pr_P(b^{(n,k)}_f > \epsilon) \leq \pr_P(b^{(n,k)}_f\ind_{\Omega_{P,n}} > \epsilon) + \pr_P(\Omega_{P,n}^c) \leq \pr_P(b^{(n,k)}_f\ind_{\Omega_{P,n}} > \epsilon) + \epsilon.
		\]
		Lemma~\ref{lem:uniformMarkov} shows that the first term above converges to $0$, uniformly in $P$, and so 
		$b_f^{(n,k)}$ is $o_\mathcal{P}(1)$.
		
		Turning now to the second conclusion, we aim to show that $\hat \Sigma^{(n)}-\Sigma_P=o_{\mathcal{P}}(1)$. We introduce notation for the following random functions:
		\begin{equation*}
			\hat\psi^{(n,k)}(y,x,z) := \nabla \hat f^{(n,k)}(x,z) - \hat\rho^{(n,k)}(x,z)\{y-\hat f^{(n,k)}(x,z)\} -\hat\theta^{(n)}.
		\end{equation*}
		We will focus on an individual element $(\hat\Sigma^{(n)}-\Sigma_P)_{l,m}$, $1\leq l,m\leq d$, and make use of Lemma \ref{lem:uniformWLLN}. We first check that
		\begin{equation*}
			\sup_{P\in\mathcal{P}}\E_P\Big[\big|\psi_{P,l}(Y,X,Z)\psi_{P,m}(Y,X,Z) - \E_P[\psi_{P,l}(Y,X,Z)\psi_{P,m}(Y,X,Z)]\big|^{1+\tilde\eta}\Big]\leq \tilde c,
		\end{equation*} for some $\tilde \eta, \tilde c>0$. Indeed, due to the convexity of $x\mapsto |x|^{1+\tilde \eta}$,
		\begin{align*}
			\E_P\Big[\big|\psi_{P,l}(Y,X,Z)\psi_{P,m}(Y,X,Z) &- \E_P[\psi_{P,l}(Y,X,Z)\psi_{P,m}(Y,X,Z)]\big|^{1+\tilde\eta}\Big] \\
			&\leq 2^{\tilde \eta} \Big\{ \E_P\Big[\big|\psi_{P,l}(Y,X,Z)\psi_{P,m}(Y,X,Z)\big|^{1+\tilde\eta}\Big] \\
			&\phantom{\leq}  + \Big| \E_P[\psi_{P,l}(Y,X,Z)\psi_{P,m}(Y,X,Z)]\Big|^{1+\tilde\eta}\Big\}\\
			&\leq 2^{1+\tilde \eta} \E_P\Big[\big|\psi_{P,l}(Y,X,Z)\psi_{P,m}(Y,X,Z)\big|^{1+\tilde\eta}\Big]\\
			&\leq 2^{1+\tilde \eta}\E_P\Big[\|\psi_{P}(Y,X,Z)\|_2^{2+2\tilde\eta}\Big].
		\end{align*}
		The first inequality is $|(a+b)/2|^{1+\tilde \eta} \leq (|a|^{1+\tilde \eta} + |b|^{1+\tilde \eta})/2$, the second is Jensen's inequality, and the final inequality is $|ab|\leq (a^2 + b^2)/2$. Therefore the condition is satisfied for $\tilde \eta=\eta/2$, $\tilde c =2^{1+\eta/2} c_2$.
		
		We are now ready to decompose the covariance estimation error.
		\begin{align*}
			(\hat\Sigma^{(n)}-\Sigma_P)_{l,m} &= \frac1n\sum_{k=1}^K \sum_{i\in I^{(n,k)}} \hat\psi_l^{(n,k)}(y_i,x_i,z_i)\hat\psi_m^{(n,k)}(y_i,x_i,z_i) - \E_P[\psi_{P,l}(Y,X,Z)\psi_{P,m}(Y,X,Z)]\\
			&= \frac1n\sum_{i=1}^n\Big[ \psi_{P,l}(y_i,x_i,z_i)\psi_{P,m}(y_i,x_i,z_i) - \E_P[\psi_{P,l}(Y,X,Z)\psi_{P,m}(Y,X,Z)]\Big] + \frac1K\sum_{k=1}^K S_P^{(n,k)},
		\end{align*} 
		where the first term is $o_{\mathcal{P}}(1)$ by Lemma \ref{lem:uniformWLLN} and
		\begin{equation*}
			S_P^{(n,k)} := \frac{K}{n} \sum_{i\in I^{(n,k)}} \Big[ \hat\psi_l^{(n,k)}(y_i,x_i,z_i)\hat\psi_m^{(n,k)}(y_i,x_i,z_i) - \psi_{P,l}(y_i,x_i,z_i)\psi_{P,m}(y_i,x_i,z_i) \Big].
		\end{equation*}
		
		We show that $S_P^{(n,k)}=o_{\mathcal{P}}(1)$ using the following identity for $a_1, a_2, b_1, b_2 \in \R$,
		\begin{equation*}
			a_1b_1 - a_2b_2 = (a_1-a_2)(b_1-b_2) + a_2(b_1-b_2) + b_2(a_1-a_2),
		\end{equation*}
		and then applying the Cauchy--Schwarz inequality to each term.
		\begin{align*}
			\big|S_P^{(n,k)}\big| &= \bigg| \frac{K}{n} \sum_{i\in I^{(n,k)}} \hat\psi_l^{(n,k)}(y_i,x_i,z_i)\hat\psi_m^{(n,k)}(y_i,x_i,z_i) - \psi_{P,l}(y_i,x_i,z_i)\psi_{P,m}(y_i,x_i,z_i) \bigg|\\
			&\leq \bigg|\frac{K}{n} \sum_{i\in I^{(n,k)}} \Big\{\hat\psi_l^{(n,k)}(y_i,x_i,z_i)-\psi_{P,l}(y_i,x_i,z_i)\Big\}\Big\{\hat\psi_m^{(n,k)}(y_i,x_i,z_i) - \psi_{P,m}(y_i,x_i,z_i)\Big\}\bigg|\\
			&\phantom{=} + \bigg|\frac{K}{n} \sum_{i\in I^{(n,k)}} \Big\{\hat\psi_l^{(n,k)}(y_i,x_i,z_i)-\psi_{P,l}(y_i,x_i,z_i)\Big\} \psi_{P,m}(y_i,x_i,z_i)\bigg|\\
			&\phantom{=} + \bigg|\frac{K}{n} \sum_{i\in I^{(n,k)}} \Big\{\hat\psi_m^{(n,k)}(y_i,x_i,z_i)-\psi_{P,m}(y_i,x_i,z_i)\Big\} \psi_{P,l}(y_i,x_i,z_i)\bigg|\\
			&\leq \bigg[\frac{K}{n} \sum_{i\in I^{(n,k)}} \Big\{\hat\psi_l^{(n,k)}(y_i,x_i,z_i)-\psi_{P,l}(y_i,x_i,z_i)\Big\}^2\bigg]^{1/2} \\
			&\phantom{=} \cdot \bigg[\frac{K}{n} \sum_{i\in I^{(n,k)}}\Big\{\hat\psi_m^{(n,k)}(y_i,x_i,z_i) - \psi_{P,m}(y_i,x_i,z_i)\Big\}^2\bigg]^{1/2}\\
			&\phantom{=} + \bigg[\frac{K}{n} \sum_{i\in I^{(n,k)}} \Big\{\hat\psi_l^{(n,k)}(y_i,x_i,z_i)-\psi_{P,l}(y_i,x_i,z_i)\Big\}^2\bigg]^{1/2} \bigg[\frac{K}{n} \sum_{i\in I^{(n,k)}} \psi_{P,m}(y_i,x_i,z_i)^2\Bigg]^{1/2}\\
			&\phantom{=} + \bigg[\frac{K}{n} \sum_{i\in I^{(n,k)}} \Big\{\hat\psi_m^{(n,k)}(y_i,x_i,z_i)-\psi_{P,m}(y_i,x_i,z_i)\Big\}^2\bigg]^{1/2} \bigg[\frac{K}{n} \sum_{i\in I^{(n,k)}} \psi_{P,l}(y_i,x_i,z_i)^2\bigg]^{1/2}.
		\end{align*}
		Therefore it suffices to show that, for each $l=1,\ldots,d$,
		\begin{align*}
			T_{P,1}^{(n,k)}&:= \frac{K}{n} \sum_{i\in I^{(n,k)}} \psi_{P,l}(y_i,x_i,z_i)^2 = O_{\mathcal{P}}(1);\\
			T_{P,2}^{(n,k)}&:= \frac{K}{n} \sum_{i\in I^{(n,k)}} \Big\{\hat\psi_l^{(n,k)}(y_i,x_i,z_i)-\psi_{P,l}(y_i,x_i,z_i)\Big\}^2 = o_{\mathcal{P}}(1).
		\end{align*}
		To this end, Lemma \ref{lem:uniformWLLN} gives $T_{P,1}^{(n,k)} = (\Sigma_{P})_{(l,l)} + o_\mathcal{P}(1)$.
		Moreover, similarly to equation \eqref{eqn:psidecomposition} and using the inequality $(a+b+c+d)^2\leq 4(a^2+b^2+c^2+d^2)$,
		\begin{align*}
			T_{P,2}^{(n,k)} &= \frac{K}{n} \sum_{i\in I^{(n,k)}} \Big\{\hat\psi_l^{(n,k)}(y_i,x_i,z_i)-\psi_{P,l}(y_i,x_i,z_i)\Big\}^2 \\
			&= \frac{K}{n} \sum_{i\in I^{(n,k)}} \Big[ \Big\{\rho_{P,l}(x_i,z_i)-\hat\rho_l^{(n,k)}(x_i,z_i)\Big\}\Big\{f_{P}(x_i,z_i)-\hat f^{(n,k)}(x_i,z_i)\Big\} \\
			& \phantom{=} - \nabla_l f_{P}(x_i,z_i)+\nabla_l \hat f^{(n,k)}(x_i,z_i) -\rho_{P,l}(x_i,z_i)\Big\{f_{P}(x_i,z_i)-\hat f^{(n,k)}(x_i,z_i)\Big\}\\
			& \phantom{=} + \Big\{\rho_{P,l}(x_i,z_i)-\hat\rho_l^{(n,k)}(x_i,z_i)\Big\} \xi_{P,i}  - \hat\theta_l^{(n)}+\theta_{P,l}\Big]^2 \\
			&\leq 4\{\tilde a^{(n,k)} + \tilde b_f^{(n,k)} + \tilde b_\rho^{(n,k)} + (\hat\theta_l^{(n)}-\theta_{P,l})^2\},
		\end{align*}
		where
		\begin{align*}
			\tilde a^{(n,k)}&:= \frac{K}{n} \sum_{i\in I^{(n,k)}} \Big\{\rho_{P,l}(x_i,z_i)-\hat\rho_l^{(n,k)}(x_i,z_i)\Big\}^2\Big\{f_P(x_i,z_i)-\hat f^{(n,k)}(x_i,z_i)\Big\}^2;\\
			\tilde b_f^{(n,k)}&:= \frac{K}{n} \sum_{i\in I^{(n,k)}} \left[\nabla_l f_P(x_i,z_i)-\nabla_l \hat f^{(n,k)}(x_i,z_i) +\rho_{P,l}(x_i,z_i)\left\{f_P(x_i,z_i)-\hat f^{(n,k)}(x_i,z_i)\right\}\right]^2;\\
			\tilde b_\rho^{(n,k)}&:= \frac{K}{n} \sum_{i\in I^{(n,k)}} \Big\{\rho_{P,l}(x_i,z_i)-\hat\rho_l^{(n,k)}(x_i,z_i)\Big\}^2 \xi_{P,i}^2.
		\end{align*}
		
		Since $n^{-1/2}(\hat\theta^{(n)}-\theta_P)$ is uniformly asympotically Gaussian, we have that $(\hat\theta_l^{(n)}-\theta_{P,l})^2=O_\mathcal{P}(n^{-1})$. For $\tilde a^{(n,k)}$, $\tilde b_f^{(n,k)}$ and $\tilde b_\rho^{(n,k)}$ we use Lemma~\ref{lem:uniformMarkov}, noting that conditionally on $D^{(n,k)}$ each summand is i.i.d.
		
		Using the identity $\sum_i a_i b_i \leq (\sum_i a_i)(\sum_i b_i) $ for positive sequences $(a_i)$ and $(b_i)$, we have
		\begin{equation*}
			\big|\tilde a^{(n,k)}\big| \leq \frac{n}{K} \tilde a_\rho^{(n,k)} \tilde a_f^{(n,k)},
		\end{equation*}
		for \begin{align*}
			\tilde a_\rho^{(n,k)}&:= \frac{K}{n} \sum_{i\in I^{(n,k)}} \Big\{\rho_{P,l}(x_i,z_i)-\hat\rho_l^{(n,k)}(x_i,z_i)\Big\}^2;\\
			\tilde a_f^{(n,k)}&:= \frac{K}{n} \sum_{i\in I^{(n,k)}} \Big\{f_P(x_i,z_i)-\hat f^{(n,k)}(x_i,z_i)\Big\}^2.
		\end{align*}
		
		Finally,
		\begin{align*}
			\E_P\Big(\big|\tilde a_f^{(n,k)}\big|\;\Big|\; D^{(n,k)}\Big) &=  \E_P\Big[ \big\{f_P(X,Z)-\hat f^{(n,k)}(X,Z)\big\}^2 \;\Big|\; D^{(n,k)}\Big]  \leq  A_f^{(n,k)};\\
			\E_P\Big(\big|\tilde a_\rho^{(n,k)}\big|\;\Big|\; D^{(n,k)}\Big) &=  \E_P\Big[ \big\{\rho_{P,l}(X,Z)-\hat\rho_l^{(n,k)}(X,Z)\big\}^2 \;\Big|\; D^{(n,k)}\Big] \leq  A_\rho^{(n,k)};\\
			\E_P\Big(\big|\tilde b_f^{(n,k)}\big| \;\Big|\; D^{(n,k)}\Big) &=  \E_P\Big(\big[\nabla_l f_P(X,Z)-\nabla_l \hat f^{(n,k)}(X,Z) \\
			& \qquad+\rho_{P,l}(X,Z)\big\{f_P(X,Z)-\hat f^{(n,k)}(X,Z)\big\}\big]^2 \;\Big|\; D^{(n,k)}\Big)\\
			&\leq  E_f^{(n,k)};\\
			\E_P\Big(\big|\tilde b_\rho^{(n,k)}\big|\;\Big|\; D^{(n,k)}\Big) &=  \E_P\Big[ \big\{\rho_{P,l}(X,Z)-\hat\rho_l^{(n,k)}(X,Z)\big\}^2 \xi_P^2 \;\Big|\; D^{(n,k)}\Big] \leq  E_\rho^{(n,k)}.
		\end{align*}
		This suffices to show that $T^{(n,k)}_{P,2} = o_\mathcal{P}(1)$, so $\hat\Sigma^{(n)}-\Sigma_P = o_\mathcal{P}(1)$.
		
		It remains to show the final conclusion. By Lemma \ref{lem:uniformSlutsky}, it is enough to show that \begin{equation*}
			(\Sigma_P)^{-1/2}(\hat \Sigma^{(n)})^{1/2} = I+o_\mathcal{P}(1).
		\end{equation*}
		We have that the maximal eigenvalue of $(\Sigma_P)^{-1/2}$ is uniformly bounded above, and further that 
		\begin{align*}
			(\Sigma_P)^{-1/2}(\hat \Sigma^{(n)})^{1/2} &=(\Sigma_P)^{-1/2}\big\{(\hat \Sigma^{(n)})^{1/2} - (\Sigma_P)^{1/2} + (\Sigma_P)^{1/2}\big\}\\
			&=I+(\Sigma_P)^{-1/2}\big\{(\hat \Sigma^{(n)})^{1/2} - (\Sigma_P)^{1/2} \big\}.
		\end{align*}
		Therefore it remains to check that 
		\begin{equation*}
			(\hat \Sigma^{(n)})^{1/2} = (\Sigma_P)^{1/2}+o_\mathcal{P}(1).
		\end{equation*} By \citet[Eqn.~(7.2.13)]{HornMatrix}, \begin{align*}
			\big\|(\hat \Sigma^{(n)})^{1/2}  
			- (\Sigma_P)^{1/2} \big\|_2 &\leq \big\|(\Sigma_P)^{-1/2} \big\|_2  \|\hat \Sigma^{(n)} - \Sigma_P \|_2\\
			&\leq c_1^{-1/2}  \|\hat \Sigma^{(n)} - \Sigma_P \|_2.
		\end{align*}
		Hence $(\hat \Sigma^{(n)})^{1/2} - (\Sigma_P)^{1/2} = o_\mathcal{P}(1)$.
		This completes the proof.
	\end{proof}
	
	\subsection{Auxiliary lemmas}
	\begin{lemma}[{\citet[Supp. Lem.~18]{ShahGCM}}, vectorised]
		\label{lem:uniformCLT}
		Let $\mathcal{P}$ be a family of distributions for $\zeta\in\R^d$ and suppose $\zeta_1, \zeta_2,\dots$ are i.i.d.\ copies. For each $n\in\mathbb{N}$, let $S_n=n^{-1/2}\sum_{i=1}^n \zeta_i$. Suppose that for all $P\in\mathcal{P}$, we have $E_P(\zeta)=0$, $\Var_P(\zeta)=I$, and $\E_P(\|\zeta\|_2^{2+\eta})\leq c$ for some $c, \eta>0$. Then we have that \begin{equation*}
			\lim_{n\to\infty} \sup_{P\in\mathcal{P}} \sup_{t\in\R^d} |\Pr_P(S_n\leq t)-\Phi(t)| = 0.
		\end{equation*}
	\end{lemma}
	\begin{proof}
		For each $n$, let $P_n\in\mathcal{P}$ satisfy \begin{equation*}
			\sup_{P\in\mathcal{P}} \sup_{t\in\R^d} |\Pr_P(S_n\leq t)-\Phi(t)| \leq \sup_{t\in\R^d} |\Pr_{P_n}(S_n\leq t)-\Phi(t)| + n^{-1}.
		\end{equation*}
		Let $Y_{n,i}$ be equal in distribution to $n^{-1/2}\zeta_i$ under $P_n$. We check the conditions to apply \citet[Prop.~2.27]{vanderVaartAsymptotic}. Indeed, $Y_{n,1},\dots, Y_{n,n}$ are i.i.d.\ for each $n$, and $\sum_{i=1}^n \Var(Y_{n,i}) = \sum_{i=1}^n n^{-1} \Var_{P_n}(\zeta)=I$. Finally, for any $\epsilon>0$ we have \begin{align*}
			\sum_{i=1}^n \E\big(\|Y_{n,1}\|_2^2 \ind_{\{\|Y_{n,i}\|_2>\epsilon\}} \big) &= \E_{P_n}\big(\|\zeta\|_2^2 \ind_{\{\|\zeta\|_2>\sqrt{n} \epsilon\}}\big)\\
			&\leq \big[\E_{P_n}(\|\zeta\|_2^{2+\eta})\big]^{2/(2+\eta)} \big[\E_{P_n}\big(\ind_{\{\|\zeta\|_2>\sqrt{n} \epsilon\}}^{(2+\eta)/\eta}\big)\big]^{\eta/(2+\eta)}\\
			&\leq c^{2/(2+\eta)} [\Pr_{P_n}(\|\zeta\|_2>\sqrt{n} \epsilon)]^{\eta/(2+\eta)}\\
			&\leq c^{2/(2+\eta)} [\E_{P_n}(\|\zeta\|_2)/(\sqrt{n}\epsilon)]^{\eta/(2+\eta)}\\
			&\leq c \epsilon^{-\eta/(2+\eta)}  n^{-\eta/(4+2\eta)} \to 0.
		\end{align*}
		Here the first inequality is due to H\"older, the third due to Markov and the second and fourth are applying the assumption $\E_P(\|\zeta\|_2^{2+\eta})\leq c$.
	\end{proof}
	
	\begin{lemma}[{\citet[Supp. Lem.~19]{ShahGCM}}]
		\label{lem:uniformWLLN}
		Let $\mathcal{P}$ be a family of distributions for $\zeta\in\R$ and suppose $\zeta_1, \zeta_2,\dots$ are i.i.d.\ copies. For each $n\in\mathbb{N}$, let $S_n=n^{-1}\sum_{i=1}^n \zeta_i$. Suppose that for all $P\in\mathcal{P}$, we have $E_P(\zeta)=0$, and $\E_P(|\zeta|^{1+\eta})\leq c$ for some $c, \eta>0$. Then we have that for all $\epsilon>0$,
		\begin{equation*}
			\lim_{n\to\infty} \sup_{P\in\mathcal{P}} \Pr_P(|S_n|>\epsilon) = 0.
		\end{equation*}
	\end{lemma}
	
	\begin{lemma}
		\label{lem:uniformSlutsky}
		Let $\mathcal{P}$ be a family of distributions that determines the law of sequences $(V_n)_{n\in\mathbb{N}}$ and $(W_n)_{n\in\mathbb{N}}$ of random vectors in $\R^d$ and $(M_n)_{n\in\mathbb{N}}$ random matrices in $\R^{d\times d}$. Suppose \begin{equation*}
			\lim_{n\to\infty} \sup_{P\in\mathcal{P}} \sup_{t\in\R^d} |\Pr_P(V_n\leq t)-\Phi(t)| = 0.
		\end{equation*}
		Then we have the following.
		\begin{enumerate}[label=(\alph*)]
			\item If $W_n=o_\mathcal{P}(1)$ we have
			\begin{equation*}
				\lim_{n\to\infty} \sup_{P\in\mathcal{P}} \sup_{t\in\R^d} |\Pr_P(V_n + W_n\leq t)-\Phi(t)| = 0.
			\end{equation*}
			\item If $M_n= I + o_\mathcal{P}(1)$ we have \begin{equation*}\lim_{n\to\infty} \sup_{P\in\mathcal{P}} \sup_{t\in\R^d} |\Pr_P(M_n^{-1}V_n\leq t)-\Phi(t)| = 0.\end{equation*}
		\end{enumerate}
	\end{lemma}
	\begin{proof}
		We first show that for any $\delta>0$, $\sup_{t\in\R^d} |\Phi(t+\delta)-\Phi(t)|\leq d\delta$. Indeed, letting $Z\sim N(0,I)$ in $\R^d$, \begin{align*}
			|\Phi(t+\delta)-\Phi(t)| &= \Pr(\cap_j \{Z_j\leq t_j + \delta\}) - \Pr(\cap_j \{Z_j\leq t_j\})\\
			&= \Pr(\cup_j\{Z_j\in(t_j,t_j+\delta]\})\\
			&\leq \sum_j \Pr(Z_j \in (t_j,t_j+\delta])\\
			&\leq d\delta.
		\end{align*}
		The final line follows because the univariate standard normal c.d.f.\ has Lipschitz constant $1/\sqrt{2\pi}~<~1$.
		
		Now consider the setup of (a). Given $\epsilon>0$ let $N$ be such that for all $n\geq N$ and for all $P\in \mathcal{P}$,
		\begin{equation*}
			\sup_{t\in\R^d} |\Pr_P(V_n\leq t)-\Phi(t)|<\epsilon/3 \text{\;\; and \;\;} \Pr_P[\|W_n\|_\infty > \epsilon/(3d)] < \epsilon/3.
		\end{equation*}
		Then \begin{align*}
			\Pr_P(V_n+W_n\leq t)-\Phi(t) &= \Pr_P(\cap_j \{V_{nj}+W_{nj} \leq t_j\}) - \Phi(t)\\
			&\leq \Pr_P[(\cap_j \{V_{nj}\leq t_j + \epsilon/(3d)\})\cup (\cup_j \{W_{nj} < -\epsilon/(3d)\})] - \Phi(t)\\
			&\leq \Pr[V_n\leq t+\epsilon/(3d)] + \Pr[\|W_n\|_\infty > \epsilon/(3d)] - \Phi(t)\\
			&< \epsilon/3 + \Phi[t+\epsilon/(3d)] - \Phi(t) + \epsilon/3 < \epsilon,
		\end{align*}
		and \begin{align*}
			\Pr_P(V_n+W_n\leq t)-\Phi(t) &= 1-\Phi(t) - \Pr_P(\cup_j\{V_{nj}+W_{nj} > t_j\})\\
			&\geq 1-\Phi(t) - \Pr_P[\cup_j(\{V_{nj} > t_j-\epsilon/(3d)\}\cup\{W_{nj} > \epsilon/(3d)\})]\\
			&= 1-\Phi(t) - \Pr_P[(\cup_j\{V_{nj} > t_j-\epsilon/(3d)\})\cup\{\|W_{n}\|_\infty > \epsilon/(3d)\}]\\
			&\geq 1-\Phi(t) - \Pr_P(\cup_j\{V_{nj} > t_j-\epsilon/(3d)\}) - \Pr_P[\|W_{n}\|_\infty > \epsilon/(3d)]\\
			&> \Pr_P[V_n\leq t-\epsilon/(3d)] - \Phi(t) - \epsilon/3\\
			&> -\epsilon/3 + \Phi[t-\epsilon/(3d)] - \Phi(t) - \epsilon/3 > -\epsilon.
		\end{align*}
		Thus for all $n\geq N$ and $P\in\mathcal{P}$,
		\begin{equation*}
			\sup_{t\in\R^d} |\Pr_P(V_n + W_n\leq t)-\Phi(t)|<\epsilon.
		\end{equation*}
		
		To prove (b), it suffices to show that $(M_n^{-1}-I)V_n = o_\mathcal{P}(1)$ and then apply (a). We have that $M_n-I$ is $o_\mathcal{P}(1)$, and so the sequence
		\begin{equation*}
			\|M_n-I\|_\infty := \sup_{x:\|x\|_\infty=1} \|(M_n-I)x\|_\infty = o_\mathcal{P}(1).
		\end{equation*}
		By \citet[Thm.~2.3.4]{GolubMatrix}, when $\|M_n-I\|_\infty <1$, then $M_n$ is nonsingular and \begin{equation*}
			\|M_n^{-1}-I\|_\infty \leq \frac{\|M_n-I\|_\infty}{1-\|M_n-I\|_\infty}.
		\end{equation*}
		Now \begin{align*}
			\sup_{P\in\mathcal{P}} \Pr_P(\|M_n^{-1}-I\|_\infty>\epsilon) & \leq \sup_{P\in\mathcal{P}} \Pr_P\bigg(\frac{\|M_n-I\|_\infty}{1-\|M_n-I\|_\infty}>\epsilon\bigg)\\
			&= \sup_{P\in\mathcal{P}}\Pr_P(\|M_n-I\|_\infty>\epsilon/(1+\epsilon)) \to 0,
		\end{align*}
		so $\|M_n^{-1}-I\|_\infty$ is also $o_\mathcal{P}(1)$. 
		
		Now we can show that the sequence $\|(M_n^{-1}-I)V_n\|_\infty$ is $o_\mathcal{P}(1)$. Indeed given $\epsilon>0$, let $\delta>0$ be such that $\Phi(\epsilon/\delta)>1-\epsilon/3$, and let $N$ be such that for all $n\geq N$ and for all $P\in\mathcal{P}$,
		\begin{equation*}
			\sup_{t\in\R^d} |\Pr_P(V_n\leq t)-\Phi(t)|<\epsilon/3 \text{\;\; and \;\;} \Pr_P(\|M_n^{-1}-I\|_\infty>\delta) < \epsilon/3.
		\end{equation*}
		Then 
		\begin{align*}     
			\Pr_P(\|(M_n^{-1}-I)V_n\|_\infty>\epsilon)&\leq \Pr_P(\|M_n^{-1}-I\|_\infty\|V_n\|_\infty>\epsilon)\\
			&\leq \Pr_P(\{\|M_n^{-1}-I\|_\infty>\delta\}\cup \{\|V_n\|_\infty>\epsilon/\delta\})\\
			&\leq \Pr_P(\|M_n^{-1}-I\|_\infty>\delta) +  1-\Pr_P(V_n\leq\epsilon/\delta)\\
			&< \epsilon/3 + 1-\Phi(\epsilon/\delta) + \epsilon/3 < \epsilon.
		\end{align*}
		This suffices to show that the sequence of random vectors $(M_n^{-1}-I)V_n$ is $o_\mathcal{P}(1)$, so we are done by (a).
	\end{proof}

	\begin{lemma}
		\label{lem:uniformMarkov}
		Let $X_m$ and $Y_m$ be sequences of random vectors governed by laws in some set $\mathcal{P}$, let $\|\cdot\|$ be any norm and $q\geq 1$.
		\begin{enumerate}[label=(\alph*)]
			\item If $\E_P(\|X_m\|^q\given  Y_m) = o_{\mathcal{P}}(1)$, then $\|X_m\|=o_{\mathcal{P}}(1)$.
			\item If $\E_P(\|X_m\|^q\given  Y_m) = O_{\mathcal{P}}(1)$, then $\|X_m\|=O_{\mathcal{P}}(1)$.
		\end{enumerate}
	\end{lemma}
	\begin{proof}
		In both cases we work with a bounded version of $\|X_m\|$, and apply Markov's inequality. 
		
		Let $\E_P(\|X_m\|^q\given  Y_m) = o_{\mathcal{P}}(1)$. Given $\epsilon>0$,
		\begin{align*}
			\Pr_P[\|X_m\|>\epsilon] &= \Pr_P[\|X_m\|^q>\epsilon^q]\\
			&= \Pr_P[(\|X_m\|^q\wedge 2\epsilon^q)>\epsilon^q]\\
			&\leq \epsilon^{-q} \E_P[\|X_m\|^q\wedge 2\epsilon^q]\\
			&= \epsilon^{-q} \E_P[\E_P(\|X_m\|^q\given  Y_m) \wedge 2\epsilon^q].
		\end{align*}
		Writing $W_m =\E_P(\|X_m\|^q\given  Y_m) \wedge 2\epsilon^q$, we have that $W_m=o_\mathcal{P}(1)$ and $|W_m|\leq2\epsilon^q$ almost surely. Taking supremum over $\mathcal{P}$ and applying \citet[Supp. Lem.~25]{ShahGCM} (uniform bounded convergence), we have \begin{equation*}
			\sup_{P\in\mathcal{P}} \Pr_P(\|X_m\|>\epsilon) \leq \epsilon^{-q} \sup_{P\in\mathcal{P}} \E_P(W_m) \to 0.
		\end{equation*}
		
		The second conclusion is similar. Let $\E_P(\|X_m\|^q\given  Y_m) = O_{\mathcal{P}}(1)$. Given $\epsilon>0$ and for $M$ to be fixed later, we have
		\begin{align*}
			\Pr_P[\|X_m\|>M] &= \Pr_P[\|X_m\|^q>M^q]\\
			&= \Pr_P[(\|X_m\|^q\wedge 2M^q)>M^q]\\
			&\leq M^{-q} \E_P[\|X_m\|^q\wedge 2M^q]\\
			&= M^{-q} \E_P[\E_P(\|X_m\|^q\given  Y_m) \wedge 2M^q].
		\end{align*}
		Now let $W_m:= \E_P(\|X_m\|^q\given  Y_m) \wedge 2M^q$. Note that for any $\tilde M>0$,
		\begin{align*}
			W_m &= W_m \ind_{\{\E_P(\|X_m\|^q\given  Y_m)\leq \tilde M\}} + W_m \ind_{\{\E_P(\|X_m\|^q\given  Y_m)> \tilde M\}}\\
			&\leq \tilde M + 2M^q \ind_{\{\E_P(\|X_m\|^q\given  Y_m)> \tilde M\}}
		\end{align*}
		almost surely. Since $\E_P(\|X_m\|^q\given  Y_m) = O_{\mathcal{P}}(1)$, we may choose $\tilde M$ so that \begin{equation*}
			\sup_{m\in\mathbb{N}} \sup_{P\in\mathcal{P}} \Pr_P[\E_P(\|X_m\|^q\given  Y_m)>\tilde M] < \epsilon/3,
		\end{equation*}
		and then choose $M>(3\tilde M/\epsilon)^{1/q}$. Again applying \citet[Supp. Lem.~25]{ShahGCM}, we have \begin{align*}
			\sup_{m\in\mathbb{N}} \sup_{P\in\mathcal{P}} \Pr_P(\|X_m\|>M) &\leq M^{-q} \sup_{m\in\mathbb{N}}\sup_{P\in\mathcal{P}} \E_P(W_m)\\
			&\leq M^{-q}(\tilde M + 2M^q \epsilon/3)\\
			&<\epsilon/3 + 2\epsilon/3 = \epsilon. \qedhere
		\end{align*}
	\end{proof}

	\section{Proof of Theorem~\ref{thm:hrate}}
	\begin{proof}
		Let \begin{equation*}
			\sup_{P\in \mathcal{P}} \sup_{x,z} \big| \rho_P'(x,z)\big| =:C < \infty.
		\end{equation*}
		Using the inequality $(a+b)^2\leq 2\big(a^2+b^2\big)$, we decompose the quantities of interest \eqref{eqn:hratetarget} as follows:
		\begin{align}
			A_f^{(n)} &= \E_P \Big(\big[f_P(X,Z)-\{f_P(\cdot,Z)*K_h\}(X) + \{f_P(\cdot,Z)*K_h\}(X) - \hat f^{(n,1)}(X,Z)\big]^2 \; \Big| \; D^{(n,1)}\Big)\nonumber\\
			&\leq 2 \E_P \big([f_P(X,Z)-\{f_P(\cdot,Z)*K_h\}(X)]^2\big)\nonumber\\
			&\phantom{\leq} + 2\E_P\Big(\big[\{f_P(\cdot,Z)*K_h\}(X) - \big\{\tilde f^{(n,1)}(\cdot,Z)*K_h\big\}(X)\big]^2 \; \Big| \; D^{(n,1)}\Big)\nonumber\\
			&= 2 \E_P ([f_P(X,Z)-\{f_P(\cdot,Z)*K_h\}(X)]^2)\nonumber\\
			&\phantom{\leq} + 2\E_P\Big(\big[\big\{f_P(\cdot,Z) - \tilde f^{(n,1)}(\cdot,Z)\big\}*K_h\big]^2(X)  \; \Big| \; D^{(n,1)}\Big)\nonumber\\
			&\leq 2 \E_P\Big(\sup_{x} [f_P(x,Z)-\{f_P(\cdot,Z)*K_h\}(x)]^2\Big)\nonumber\\
			&\phantom{\leq} + 2\E_P\Big(\big[\big\{f_P(\cdot,Z) - \tilde f^{(n,1)}(\cdot,Z)\big\}*K_h\big]^2(X)  \; \Big| \; D^{(n,1)}\Big). \label{eqn:Afbound}
		\end{align}
		Similarly,
		\begin{align}
			E_f^{(n)} &\leq 2 \E_P\Big[ \big( f_P'(X,Z)- \{f_P(\cdot,Z)*K_h\}'(X)\nonumber\\
			&\phantom{\leq 2 \E_P\Big[}+\rho_{P}(X,Z)[f_P(X,Z)-\{f_P(\cdot,Z)*K_h\}(X)]\big)^2\Big]\nonumber\\
			&\phantom{\leq} + 2 \E_P\bigg[ \Big(  \big[\big\{f_P(\cdot,Z) - \tilde f^{(n,1)}(\cdot,Z)\big\}*K_h\big]'(X)\nonumber\\
			&\phantom{\leq \leq 2 \E_P\bigg[}+\rho_{P}(X,Z) \big[\big\{f_P(\cdot,Z) - \tilde f^{(n,1)}(\cdot,Z)\big\}*K_h\big](X)\Big)^2 \; \bigg| \; D^{(n,1)}\bigg]\nonumber\\
			&\leq 4 \E_P\Big(\sup_{x} [ f_P'(x,Z)-\{f_P(\cdot,Z)*K_h\}'(x)]^2\Big)\nonumber\\
			&\phantom{\leq }+ 4\E_P\Big(\E_P\big[\rho^2_{P}(X,Z) \; \big| \; Z \big] \sup_{x} [f_P(x,Z)-\{f_P(\cdot,Z)*K_h\}(x)]^2\Big)\nonumber\\
			&\phantom{\leq} + 4 \E_P\bigg[ \big[\big\{f_P(\cdot,Z) - \tilde f^{(n,1)}(\cdot,Z)\big\}*K_h\big]'^2(X) \; \bigg| \; D^{(n,1)}\bigg]\nonumber\\
			&\phantom{\leq }+ 4 \E_P\Big(\rho^2_{P}(X,Z) \big[\big\{f_P(\cdot,Z) - \tilde f^{(n,1)}(\cdot,Z)\big\}*K_h\big]^2(X)\; \Big| \; D^{(n,1)}\Big).\label{eqn:Bfbound}
		\end{align}
		
		By Theorem~\ref{thm:scoresubg}, $\E_P\big[\rho^2_P(X,Z) \; \big|\; Z=z\big]$ is bounded by $C$ for almost every $z\in\mathcal{Z}$. We wish to apply Lemma~\ref{lem:convolution} to the quantities \begin{equation*}
			f_P(x,z)-\{f_P(\cdot,z)*K_h\}(x); \quad f_P'(x,z)-\{f_P(\cdot,z)*K_h\}'(x).
		\end{equation*}
		To this end, note that by a Taylor expansion
		\begin{equation}
			|f_P(x+hw,z)| \leq |f_P(x,z)| + h|w||f_P'(x,z)|+\frac{C_P(z)}{2}w^2 h^2. \label{eqn:fptaylor}
		\end{equation}
		Since $f_P,f_P'$ are real-valued, both $|f_P(x,z)|$ and $|f_P'(x,z)|$ are finite for any fixed $x,z$. The conditions for Lemma~\ref{lem:convolution} follow. Now we have that
		\begin{align*}
			\big|\{f_P(\cdot,z)*K_h\}(x) - f_P(x,z)\big|&= \big|\E[f_P(x+hW,z) - f_P(x,z)]\big|\\
			&\leq  \bigg|\E\bigg[ hW f_P'(x,z) + \frac{h^2W^2}{2} \sup_{t\in\R} \big| f_P''(t,z)\big|\bigg]\bigg|\\
			&\leq \frac{C_P(z)}{2}h^2. 
		\end{align*}
		In the second line we have applied equation \eqref{eqn:fptaylor} and the third line $\E(W)=0$, $\E\big(W^2\big)=1$.
		Similarly, 
		\begin{align*}
			\big| \{f_P(\cdot,z)*K_h\}'(x) -  f_P'(x,z)\big|&= \bigg|\frac1h \E[Wf_P(x+hW,z) - h f_P'(x,z)]\bigg|\\
			&\leq  \bigg|\frac1h\E\bigg[ W f_P(x,z) + h\big(W^2-1\big)  f_P'(x,z) + \frac{h^2|W^3|}{2} \sup_{t\in\R} \big| f_P''(t,z)\big|\bigg]\bigg|\\
			&\leq \frac{\sqrt{2}C_P(z)}{\sqrt{\pi}}h, 
		\end{align*}
		noting that $\E\big(|W|^3\big)=2\sqrt{2/\pi}$.
		The choice of $h = cn^{-\gamma}$ for any \begin{equation}
			\gamma \geq \alpha/4 \label{eqn:gammaalpha4}
		\end{equation} yields the desired rates on the respective terms in equations (\ref{eqn:Afbound},~\ref{eqn:Bfbound}).

		Write \begin{equation*}
			\Delta_{P,n}(x,z) := f_P(x,z) - \tilde f^{(n,1)}(x,z).
		\end{equation*} It remains to demonstrate the following:
		\begin{align}
			\E_P\big[\{\Delta_{P,n}(\cdot,Z)*K_h\}^2(X) \;\big|\; D^{(n,1)} \big] = O_{\mathcal{P}}\big(\tilde{A}_f^{(n)}\big); \label{eqn:g*K}\\
			\E_P\big[\{\Delta_{P,n}(\cdot,Z)*K_h\}'^2(X) \;\big|\; D^{(n,1)} \big] = o_{\mathcal{P}}(1); \label{eqn:g*Kderiv}\\
			\E_P\big[\rho_P^2(X,Z)\{\Delta_{P,n}(\cdot,Z)*K_h\}^2(X) \;\big|\; D^{(n,1)}\big] = o_{\mathcal{P}}(1) ;\label{eqn:g*Krho2}
		\end{align}
		which we do by proving bounds in terms of
		\begin{align}
			\E_P\big[\Delta_{P,n}^2(X,Z)\;\big|\; D^{(n,1)}\big] &= \tilde{A}_f^{(n)}; \nonumber\\
			\big(\E_P\big[|\Delta_{P,n}(X,Z)|^{4}\;\big|\; D^{(n,1)}\big]\big)^{\frac{1}{2}} &=: \tilde{B}_f^{(n)}.\label{eqn:Bgdefn}
		\end{align}
		Lemma~\ref{lem:eta_exists} below shows that $\tilde{B}_f^{(n)} = O_{\mathcal{P}}(n^{-\alpha/2})$. 
		We work on the event that $\Delta_{P,n}$ is bounded over $(x,z)$, which happens with high probability by assumption. This will enable us to use dominated convergence to exchange various limits below. Recall that we are not assuming any smoothness of $\tilde f^{(n,1)}$ or $\Delta_{P,n}$. Since we are working under the event that $\Delta_{P,n}$ is bounded, we have that $\Delta_{P,n}*K_h$ is bounded by the same bound as $\Delta_{P,n}$, and also due to Lemma~\ref{lem:convolution} we have that $(\Delta_{P,n}*K_h)'$ exists and is bounded. By assumption and Theorem~\ref{thm:scoresubg} we have that $\rho_P'(x,z)$ and all moments of $\rho_P(X,Z)$ are bounded.

		We first show that \eqref{eqn:g*Krho2} follows from (\ref{eqn:g*K}, \ref{eqn:g*Kderiv}). Due to the aforementioned bounds and Lemma~\ref{lem:pconvergent}, we can apply Proposition~\ref{prop:intbyparts} as follows.
		\begin{align*}
			&\E_P\big[\rho_P^2(X,Z)\{\Delta_{P,n}(\cdot,Z)*K_h\}^2(X) \;\big|\; D^{(n,1)} \big]\\ &\qquad = -\E_P\big[\rho_P'(X,Z)\{\Delta_{P,n}(\cdot,Z)*K_h\}^2(X) \\
			&\qquad\phantom{= -\E_P\big[}+ 2\rho_P(X,Z)\{\Delta_{P,n}(\cdot,Z)*K_h\}'(X)\{\Delta_{P,n}(\cdot,Z)*K_h\}(X) \;\big|\; D^{(n,1)} \big]\\
			&\qquad\leq \sup_{x,z}|\rho_P'(x,z)| \E_P\big[\{\Delta_{P,n}(\cdot,Z)*K_h\}^2(X)\;\big|\; D^{(n,1)} \big]\\
			&\qquad\phantom{=} + \sup_{x,z} | \Delta_{P,n}(x,z)|\big(\E_P[\rho^2_P(X,Z)]\big)^{1/2}\big(\E_P\big[\{\Delta_{P,n}(\cdot,Z)*K_h\}'^2(X) \;\big|\; D^{(n,1)} \big]\big)^{1/2}
		\end{align*}
		The second line is due to the H\"older and Cauchy--Schwarz inequalities. All the random quantities above are integrable due to the stated bounds.
		It remains to show (\ref{eqn:g*K}, \ref{eqn:g*Kderiv}).
		
		We start with \eqref{eqn:g*K}. By Lemma~\ref{lem:convolution}, conditional Jensen's inequality, and the Fubini theorem,
		\begin{align*}
			\E_P\big[\{\Delta_{P,n}(\cdot,Z)*K_h\}^2(X)\;\big|\; D^{(n,1)}\big] &=\E_P\big[\E\{\Delta_{P,n}(X+hW,Z)\given  X,Z,D^{(n,1)}\}^2\;\big|\; D^{(n,1)}\big]\\
			&\leq \E_P\big[\E\big\{\Delta_{P,n}^2(X+hW,Z)\given  X,Z,D^{(n,1)}\big\}\;\big|\; D^{(n,1)}\big]\\
			&= \E\big[\E_P\big\{\Delta_{P,n}^2(X+hW,Z)\;\big|\; W, D^{(n,1)}\big\} \;\big|\; D^{(n,1)} \big].
		\end{align*}
		Define a new function $\phi_{P,n}:\R\to\R$ by $\phi_{P,n}(t) =  \E_P\big[\Delta_{P,n}^2(X+t,Z)\;\big|\; D^{(n,1)}\big]$, so $\phi_{P,n}(0)=\tilde{A}_f^{(n)}$. We will show later in the proof that $\phi_{P,n}$ is twice differentiable, which we
		assume to be true for now.
		By a Taylor expansion, for each fixed $h>0$, $w\in \R$ we have
		\begin{equation*}
			\phi_{P,n}(hw) \leq \phi_{P,n}(0)+hw \phi_{P,n}'(0) + \frac{h^2w^2}{2} \sup_{|t|\leq h|w|} \big|\phi_{P,n}''(t)\big|.
		\end{equation*}
		We will also show later that the remainder term is integrable with respect to the Gaussian density. Taking expectations over $W$ yields
		\begin{align}
			\E\big[\phi_{P,n}(hW)\;\big|\; D^{(n,1)}\big] &\leq \phi_{P,n}(0) + h\E(W) \phi_{P,n}'(0) + \frac{h^2}{2}\int_\R w^2K(w) \sup_{|t|\leq h|w|} \big|\phi_{P,n}''(t)\big| \; dw \nonumber\\
			&= \phi_{P,n}(0) + \frac{h^2}{2}\int_\R w^2K(w) \sup_{|t|\leq h|w|} \big|\phi_{P,n}''(t)\big| \; dw. \label{eqn:phihwtaylor}
		\end{align}
		In the final line we have used $\E(W)=0$.
		
		Now considering the quantity \eqref{eqn:g*Kderiv}, Lemma \ref{lem:convolution} implies that \begin{equation*}
			\{\Delta_{P,n}(\cdot,z)*K_h\}'(x) = \frac1h \E\big[W\Delta_{P,n}(x+hW,z)\;\big|\; D^{(n,1)}\big].
		\end{equation*}
		Similarly to the above,
		\begin{align*}
			\E_P\big([\{\Delta_{P,n}(\cdot,Z)*K_h\}'(X)]^2\;\big|\; D^{(n,1)}\big) &= h^{-2}\E_P\big[\E\big\{W \Delta_{P,n}(X+hW,Z)\;\big|\; X,Z, D^{(n,1)}\}^2\;\big|\; D^{(n,1)}\big]\\
			&\leq h^{-2}\E_P\big[\E\big\{W^2 \Delta_{P,n}^2(X+hW,Z)\;\big|\; X,Z,D^{(n,1)} \big\}\;\big|\; D^{(n,1)}\big]\\
			&= h^{-2}\E\big[W^2 \E_P\big\{\Delta_{P,n}^2(X+hW,Z)\;\big|\; W, D^{(n,1)}\big\}\;\big|\; D^{(n,1)}\big].
		\end{align*}
		Moreover, 
		\begin{align}
			h^{-2}\E\big[W^2\phi_{P,n}(hW)\;\big|\; D^{(n,1)}\big] &= h^{-2} \E\big(W^2\big)\phi_{P,n}(0) + h^{-1}\E\big(W^3\big) \phi_{P,n}'(0) \nonumber\\
			&\phantom{=} +  \frac12 \int_\R w^4K(w) \sup_{|t|\leq h|w|} \big|\phi_{P,n}''(t)\big| \; dw \nonumber\\
			&= h^{-2}\phi_{P,n}(0) + \frac12 \int_\R w^4K(w) \sup_{|t|\leq h|w|} \big|\phi_{P,n}''(t)\big| \; dw. \label{eqn:w2phihwtaylor}
		\end{align}
		In the final line we have used $\E\big(W^2\big)=1$, $\E\big(W^3\big)=0$.
		
		It remains to check that $\phi_{P,n}$ is twice differentiable and compute its derivatives. By a change of variables $u=x + t$,
		\begin{align*}
			\phi_{P,n}(t) &= \E_P\Big[\E_P\big\{\Delta_{P,n}^2(X+t,Z) \;\big|\; Z, D^{(n,1)} \big\}\;\Big|\; D^{(n,1)}\Big]\\
			&= \E_P\bigg[\int_{\R} \Delta_{P,n}^2 (x+t,Z) p_P(x\given  Z) \; dx  \;\bigg|\; D^{(n,1)}\bigg]\\
			&= \E_P\bigg[\int_{\R} \Delta_{P,n}^2 (u,Z) p_P(u-t\given  Z) \; du\;\bigg|\; D^{(n,1)}\bigg].
		\end{align*}
		The conditional density $p_P$ is assumed twice differentiable, so the integrand is twice differentiable with respect to $t$. The bound on $\Delta_{P,n}$ and conclusion of Lemma~\ref{lem:pderivlocallybounded} allow us to interchange the differentiation and expectation operators using \citet[Thm.~20.4]{AliprantisAnalysis}. Differentiating $\phi_{P,n}$ twice gives \begin{align}
			\phi_{P,n}''(t) &= \partial_{t}^2\E_P\bigg[\int_{\R} \Delta_{P,n}^2 (u,Z) p_P(u-t\given  Z) \; du\;\bigg|\; D^{(n,1)}\bigg] \nonumber\\
			&= \E_P\bigg[ \int_{\R} \Delta_{P,n}^2 (u,Z) p_P''(u-t\given  Z) \; du\;\bigg|\; D^{(n,1)}\bigg].\label{eqn:phi''int}
		\end{align}
		
		Note that 
		\begin{align}
			\rho_{P}'(x,z) &=  \bigg(\frac{p_P'(x\given  z)}{p_P(x\given  z)}\bigg)' \nonumber\\
			&= \frac{p_P''(x\given  z)}{p_P(x\given  z)} - \bigg(\frac{p_P'(x\given  z)}{p_P(x\given  z)}\bigg)^2 \nonumber\\
			&= \frac{p_P''(x\given  z)}{p_P(x\given  z)} - \rho^2_{P}(x, z). \label{eqn:p''identity}
		\end{align} 
		Applying equation \eqref{eqn:p''identity}, the Lipschitz property of $\rho_P$, and Lemma~\ref{lem:densityratiobound} to the interior of \eqref{eqn:phi''int} yields \begin{align*}
			\bigg| \int_\R &\Delta_{P,n}^2 (u,z) p_P''(u-t\given  z) \; du  \bigg| \\
			&= \bigg| \int_\R \Delta_{P,n}^2 (u,z) \big\{\rho_{P}'(u-t, z) + \rho_{P}^2(u-t, z)\big\} \;p_P(u-t\given  z)  \; du \bigg|\\
			&= \bigg| \int_\R \Delta_{P,n}^2 (u,z) \big[\rho_{P}'(u-t, z) + \{\rho_{P}(u-t, z) - \rho_{P}(u, z) + \rho_{P}(u, z)\}^2\big] \\
            &\phantom{=\bigg| \int_\R} \cdot \frac{p_P(u-t\given  z)}{p_P(u\given  z)}p_P(u\given  z)  \; du \bigg| \\
			&\leq \int_\R \Delta_{P,n}^2 (u,z) \big[C + \{C|t| + \rho_{P}(u, z)\}^2\big] \;\exp\bigg(-t\rho_P(u,z)+\frac{C}{2}t^2\bigg)p_P(u\given  z)  \; du\\
			&\leq \int_\R \Delta_{P,n}^2 (u,z) \big\{C + 2C^2t^2 + 2\rho^2_{P}(u, z)\big\} \;\exp\bigg(-t\rho_P(u,z)+\frac{C}{2}t^2\bigg)p_P(u\given  z)  \; du\\
			&= \E_P \Bigg[ \Delta_{P,n}^2 (X,z) \big\{C + 2C^2t^2 + 2\rho^2_{P}(X, z)\big\} \\
            & \phantom{=\E_P \Bigg[} \cdot \exp\bigg(-t\rho_P(X,z)+\frac{C}{2}t^2\bigg) \; \Bigg|\; Z=z, D^{(n,1)}\Bigg].
		\end{align*}
		The penultimate line uses $(a+b)^2\leq 2(a^2+b^2)$. Plugging this in to \eqref{eqn:phi''int} and using the Fubini theorem gives
		\begin{align}
			\big| \phi_{P,n}''(t) \big| &\leq \E_P \Bigg[ \Delta_{P,n}^2 (X,Z) \big\{C + 2C^2t^2 + 2\rho^2_{P}(X, Z)\big\} \\ & \phantom{= \E_P \Bigg[}\cdot\exp\bigg(-t\rho_P(X,Z)+\frac{C}{2}t^2\bigg) \; \Bigg| D^{(n,1)}\Bigg] \nonumber\\
			&= \big(C + 2C^2t^2\big)\exp\bigg(\frac{C}{2}t^2\bigg) \E_P \Big[ \Delta_{P,n}^2 (X,Z) \;\exp\big(-t\rho_P(X,Z)\big) \; \Big| D^{(n,1)}\Big] \nonumber\\
			&\phantom{=} + 2 \exp\bigg(\frac{C}{2}t^2\bigg) \E_P \Big[ \Delta_{P,n}^2 (X,Z)\rho^2_P(X,Z) \;\exp\big(-t\rho_P(X,Z)\big) \; \Big| D^{(n,1)}\Big]. \label{eqn:phi''bound}
		\end{align}
	Applying the Cauchy--Schwarz inequality twice and appealing to the monotonicity of $L_p$ norms, we obtain
		\begin{align*}
			\big| \phi_{P,n}''(t) \big| &\leq \big(C + 2C^2t^2\big)\exp\bigg(\frac{C}{2}t^2\bigg) \Big(\E_P \big[ \exp\big(-2 t \rho_P(X,Z)\big) \big]\Big)^{\frac{1}{2}} \tilde{B}_f^{(n)} \\
			&\phantom{=} + 2 \exp\bigg(\frac{C}{2}t^2\bigg) \Big(\E_P \big[ |\rho_P(X,Z)|^4 \exp\big(-2 t \rho_P(X,Z)\big) \big]\Big)^{\frac{1}{2}} \tilde{B}_f^{(n)}\\
                &\leq \big(C + 2C^2t^2\big)\exp\bigg(\frac{C}{2}t^2\bigg) \Big(\E_P \big[ \exp\big(-2 t \rho_P(X,Z)\big) \big]\Big)^{\frac{1}{2}} \tilde{B}_f^{(n)} \\
			&\phantom{=} + 2 \exp\bigg(\frac{C}{2}t^2\bigg) \Big(\E_P \big[ \rho^8_P(X,Z)\big]\Big)^{1/4} \Big(\E_P\big[ \exp\big(-4 t \rho_P(X,Z)\big) \big]\Big)^{\frac{1}{4}} \tilde{B}_f^{(n)}\\
			&\leq \exp\bigg(\frac{C}{2}t^2\bigg) \Bigg\{ C + 2C^2t^2+  2 \Big(\E_P \big[ \rho^{8}_P(X,Z)\big]\Big)^{\frac{1}{4}} \Bigg\} \\
   &\phantom{=} \cdot \Bigg(\E_P \bigg[ \exp\bigg(- 4 t \rho_P(X,Z)\bigg) \bigg]\Bigg)^{\frac{1}{4}} \tilde{B}_f^{(n)}.
		\end{align*}
		We are now in a position to apply Theorem~\ref{thm:scoresubg}. Recalling the moment generating function bound for sub-Gaussian random variables, we have 
		\begin{align*}
			|\phi_{P,n}''(t)| &\leq \exp\bigg(\frac{C}{2}t^2\bigg) \Bigg\{ C + 2C^2t^2+  2 \Big(C^{4}7!!\Big)^{\frac{1}{4}} \Bigg\}  \Big(\exp\big(4^2 C^2 t^2\big)\Big)^{\frac{1}{4}} \tilde{B}_f^{(n)} \\
			&\leq c_1 \big(1 + t^2 \big) \exp\big(c_2 t^2 \big) \tilde{B}_f^{(n)},
		\end{align*}
		for some constants $c_1,c_2>0$ depending on $C$ (but not on $P$ or $n$).
		
		Returning to equations (\ref{eqn:phihwtaylor}, \ref{eqn:w2phihwtaylor}), we have \begin{align*}
			\E\big[\phi_{P,n}(hW)\;\big|\; D^{(n,1)}\big]
			&\leq \tilde{A}_f^{(n)} + c_1 h^2 \tilde{B}_f^{(n)} \int_\R w^2(1+h^2w^2)\exp(c_2h^2w^2)K(w) \; dw.\\
			h^{-2}\E\big[W^2\phi_{P,n}(hW)\;\big|\; D^{(n,1)}\big]
			&\leq h^{-2}\tilde{A}_f^{(n)} + c_1 \tilde{B}_f^{(n)} \int_\R w^4(1+h^2w^2)\exp(c_2h^2w^2)K(w) \; dw.
		\end{align*}
		For all $0<h<\frac1{2\sqrt{c_2}}$, the final integrals are bounded by a constant.
		
		Putting everything together, we have that \begin{align*}
		  \E_P\big[\{\Delta_{P,n}(\cdot,Z)*K_h\}^2(X) \;\big|\; D^{(n,1)} \big] &= O_{\mathcal{P}}\big(\tilde{A}_f^{(n)} + h^2 \tilde{B}_f^{(n)}\big) = O_{\mathcal{P}}\big(n^{-\alpha} + h^2 n^{-\alpha/2}\big);\\
			\E_P\big[\{\Delta_{P,n}(\cdot,Z)*K_h\}'^2(X) \;\big|\; D^{(n,1)} \big] &= O_{\mathcal{P}}\big(h^{-2}\tilde{A}_f^{(n)} + \tilde{B}_f^{(n)}\big)= O_{\mathcal{P}}\big(h^{-2} n^{-\alpha} + n^{-\alpha/2}\big). 
		\end{align*}
		Hence the choice of $h = cn^{-\gamma}$ for any $\gamma \in [\alpha/4, \alpha/2)$ yields the desired rates on (\ref{eqn:g*K}, \ref{eqn:g*Kderiv}). Since this choice also satisfies the bound \eqref{eqn:gammaalpha4}, we also achieve the desired rates in (\ref{eqn:Afbound},~\ref{eqn:Bfbound}). This completes the proof.     
	\end{proof}
	\subsection{Auxiliary lemmas}
	\begin{lemma}
		\label{lem:convolution}
		Let $W\sim N(0,1)$ be a standard Gaussian random variable independent of $(X,Z)$, and fix $h>0$. Let $g:\R \times \mathcal{Z} \to \R$ be such that $\E|g(x+hW,z)| < \infty$ for all $(x,z)$. Then we have that for each $z$,
		\begin{equation*}
			x \mapsto \{g(\cdot,z)*K_h\}(x) =  \E[g(x+hW,z)]
		\end{equation*}
		is differentiable.
		Moreover, if $\E\big[|g(x+hW,z)|^{1+\eta}\big]~<~\infty$ for some $\eta>0$ then the derivative is given by \begin{equation*}
			\{g(\cdot,z)*K_h\}'(x) =  \frac1h\E[Wg(x+hW,z)].
		\end{equation*}
	\end{lemma}
	\begin{proof}
		Recall that the convolution operator is \begin{equation*}
			\{g(\cdot,z)*K_h\}(x) = \int_\R g(u,z) K_h(x-u)\; du.
		\end{equation*}
		We check the conditions for interchanging differentiation and integration operators \citep[Thm.~20.4]{AliprantisAnalysis}. The integrand $g(u,z) K_h(x-u)$ is integrable in $u$ with respect to the Lebesgue measure for each $(x,z)$, since \begin{align*}
			\int_\R |g(u,z)| K_h(x-u)\; du &= \int_{\R} |g(u,z)| \;\frac1h\; K\bigg(\frac{x-u}{h}\bigg) \; du\\
			&= \int_{\R} |g(x+hw,z)| K(-w) \; dw\\
			&= \int_{\R} |g(x+hw,z)| K(w) \; dw\\
			&= \E[|g(x+hW,z)|] < \infty.
		\end{align*} 
		Due to the smoothness of the Gaussian kernel, $g(u,z) K_h(x-u)$ is differentiable in $x$ with derivative
		absolutely continuous in $x$ for each $(u,z)$. Furthermore it has $x$-derivative 
		\begin{equation*}
			g(u,z) K_h'(x-u) = - g(u,z) \bigg(\frac{x-u}{h^2}\bigg) K_h(x-u).
		\end{equation*}
		
		Fix $x_0$ and $V=[x_0-h, x_0+h]$. It remains to find a Lebesgue integrable function $G:\R\to\R$ such that \begin{equation*}
			\bigg| g(u,z) \bigg(\frac{x-u}{h^2}\bigg) K_h(x-u) \bigg| \leq G(u)
		\end{equation*}
		for all $x\in V$ and $u\in\R$. Now for any $x\in V$, $u\in\R$,
		\begin{align*}
			\bigg| g(u,z) \bigg(\frac{x-u}{h^2}\bigg) &K_h(x-u) \bigg| \\ &=\bigg| g(u,z) \bigg(\frac{x-x_0+x_0-u}{h^2}\bigg) \frac{K_h(x-u)}{K_h(x_0-u)}K_h(x_0-u) \bigg|\\
			&= | g(u,z)| \bigg|\frac{x-x_0+x_0-u}{h^2}\bigg| \exp\bigg(-\frac{(x-x_0)^2}{2} - (x-x_0)(x_0-u)\bigg) K_h(x_0-u) \\
			&\leq |g(u,z)| \frac{h+|x_0-u|}{h^2} \exp\bigg(-\frac{h^2}2 + h|x_0-u|\bigg) K_h(x_0-u)\\
			&=:G(u).
		\end{align*}
		Moreover, recalling the symmetry of $K$ and using a change of variables $w=(u-x_0)/h$, \begin{equation*}
			\int_\R G(u)\; du = \frac1h \exp\bigg(-\frac{h^2}{2}\bigg) \int_\R |g(x_0+hw,z)| (1+|w|) \exp\big(|w|h^2\big)K(w)\;dw
		\end{equation*}
		We now apply H\"older's inequality twice. Pick $q_1, q_2 > 1$ be such that $1/q_1 + 1/q_2 = \eta/(1+\eta)$. Now,
		\begin{align*}
			\int_\R G(u)\; du &\leq \frac1h \exp\bigg(-\frac{h^2}{2}\bigg) \Big(\E\big[|g(x_0+hW,z)|^{1+\eta}\big]\Big)^{\frac1{1+\eta}} \\
			& \phantom{\leq} \cdot \bigg(\int_\R (1+|w|)^{\frac{1+\eta}{\eta}} \exp\bigg({\frac{1+\eta}{\eta}}|w|h^2\bigg)K(w)\;dw\bigg)^{\frac{\eta}{1+\eta}}\\
			&\leq \frac1h  \exp\bigg(-\frac{h^2}{2}\bigg) \Big(\E\big[|g(x_0+hW,z)|^{1+\eta}\big]\Big)^{\frac1{1+\eta}} \Big(\E\big[(1+|W|)^{q_1}\big]\Big)^{\frac1{q_1}} \Big(\E\big[ \exp\big(q_2|w|h^2\big)\big]\Big)^{\frac{1}{q_2}}.
		\end{align*}
		We have that $\E\big[|g(x_0+hW,z)|^{1+\eta}\big]$ is finite by assumption, $\E\big[(1+|W|)^{q_1}\big]$ is a Gaussian moment so is finite, and $(\E\big[ \exp\big(q_2|w|h^2\big)\big]$ is bounded in terms of the Gaussian moment generating function. Hence $G$ is Lebesgue integrable.

		Finally we check the claimed identities. Using a change of variables $u=x+hw$, and recalling the symmetry of $K$, we have that
		\begin{align*}
			\{g(\cdot,z)*K_h\}(x) &= \int_{\R} g(u,z) \;\frac1h\; K\bigg(\frac{x-u}{h}\bigg) \; du\\
			&= \int_{\R} g(x+hw,z) K(w) \; dw\\
			&= \E[g(x+hW,z)],
		\end{align*} 
		and
		\begin{align*}
			\{g(\cdot,z)*K_h\}'(x) &= \int_{\R} g(u,z)  K_h'(x-u) \; du\\
			&= -\frac1{h^{2}}\int_{\R} g(u,z) \bigg(\frac{x-u}{h}\bigg) K\bigg(\frac{x-u}{h}\bigg) \; du\\
			&= \frac1h\int_{\R} g(x+hw,z) w K(w) \; dw\\
			&= \frac1h \E[Wg(x+hW,z)].
		\end{align*}
	\end{proof}
	
	\begin{lemma}
		\label{lem:pderivlocallybounded}
		Let $p$ be a twice differentiable density on $\R$, with $\sup_{x\in\R} |\partial_x \log p(x)| =: \sup_{x\in\R} |\rho'(x)|\leq C<\infty$. Then for every $t_0\in\R$ there exists a Lebesgue integrable function $g$ such that 
		\begin{equation*}
			|p'(x-t)|,\; |p''(x-t)| \leq g(x)
		\end{equation*}
		for all $x\in\mathcal\R$ and $t\in [t_0-1, t_0+1]$. 
	\end{lemma}
	\begin{proof}
		Fix $t_0$ and set $V:=[t_0-1, t_0+1]$. We will make use of Lemma~\ref{lem:densityratiobound}. Indeed for any $t\in V$, \begin{align*}
			|p'(x-t)| &= |\rho(x-t)| p(x-t)\\
			&=|\rho(x-t)-\rho(x-t_0)+\rho(x-t_0)| \frac{p(x-t)}{p(x-t_0)} p(x-t_0)\\
			&\leq \big\{C|t-t_0| + |\rho(x-t_0)|\big\} \frac{p(x-t)}{p(x-t_0)} p(x-t_0)\\
			&\leq \big\{C|t-t_0| + |\rho(x-t_0)|\big\} \exp\bigg(|t-t_0|\;|\rho(x-t_0)| + \frac{(t-t_0)^2C}{2}\bigg) p(x-t_0)\\
			&\leq \big\{C + |\rho(x-t_0)|\big\} \exp\bigg(|\rho(x-t_0)| + \frac{C}{2}\bigg) p(x-t_0).
		\end{align*}
		Similarly,
		\begin{align*}
			|p''(x-t)| &= |\rho'(x-t)+\rho^2(x-t)| p(x-t)\\
			&\leq \big[C + \{\rho(x-t)-\rho(x-t_0)+\rho(x-t_0)\}^2 \big]\frac{p(x-t)}{p(x-t_0)} p(x-t_0)\\
			&\leq \big[C + 2C^2(t-t_0)^2 + 2\rho^2(x-t_0) \big]\frac{p(x-t)}{p(x-t_0)} p(x-t_0)\\
			&\leq \big[C + 2C^2(t-t_0)^2 + 2\rho^2(x-t_0) \big] \exp\bigg(|t-t_0|\;|\rho(x-t_0)| + \frac{(t-t_0)^2C}{2}\bigg) p(x-t_0)\\
			&\leq \big[C + 2C^2 + 2\rho^2(x-t_0) \big] \exp\bigg(|\rho(x-t_0)| + \frac{C}{2}\bigg) p(x-t_0).
		\end{align*}
		In the third line we have used the inequality $(a+b)^2 \leq 2(a^2+b^2)$.
		
		Taking $g$ to be the maximum of the two bounds, it suffices to check that the function $|\rho(x-t_0)|^k\exp(|\rho(x-t_0)|)p(x-t_0)$ is Lebesgue integrable with respect to $x$ for $k=0,1,2$. Using the change of variables $y=x-t_0$ and the Cauchy--Schwarz inequality,
		\begin{align*}
			\int_\R |\rho(x-t_0)|^k\exp(|\rho(x-t_0)|)p(x-t_0)\; dx &= \int_\R |\rho(y)|^k\exp(|\rho(y)|)p(y)\; dx\\
			&= \E\big[|\rho(X)|^k\exp(|\rho(X)|)\big]\\
			&\leq \Big(\E\big[\rho^{2k}(X)\big] \Big)^{\frac12} \Big(\E\big[\exp(2|\rho(X)|)\big] \Big)^{\frac12},
		\end{align*} 
		where $X\sim p$. By Theorem~\ref{thm:scoresubg}, \begin{equation*}
			\E\big[\rho^{2k}(X)\big] \leq C^k (2k-1)!!
		\end{equation*} for $k=1,2$ and moreover $\rho(X)$ is sub-Gaussian with parameter $\sqrt{2C}$, so \begin{align*}
			\E\big[\exp(2|\rho(X)|)\big] &\leq \E\big[\exp(2\rho(X))\big] + \E\big[\exp(-2\rho(X))\big]\\
			&\leq 2\exp(4C).
		\end{align*}
		This completes the proof.
	\end{proof}
	
	\begin{lemma} \label{lem:eta_exists}
		Consider the setup of Theorem~\ref{thm:hrate}. We have
		\[
		\tilde{B}_f^{(n)} := \Big(\E_P\Big[ \big|f_P(X,Z)-\tilde f^{(n,1)}(X,Z)\big|^{4} \;\Big| \; D^{(n,1)} \Big]\Big)^{\frac{1}{2}} = O_{\mathcal{P}}(n^{-\alpha/2}).
		\]
	\end{lemma}
	\begin{proof}
		Let us write $\Delta_{P,n} :=|f_P(X,Z)-\tilde f^{(n,1)}(X,Z)|$. Note that $\Delta_{P,n}$ is a non-negative random variable which, conditionally on $D^{(n,1)}$, is bounded above by $\sup_{x,z}|f_P(x,z)-\tilde f^{(n,1)}(x,z)|$.
		Given $\epsilon > 0$, we know by assumption that there exist $M_1, M_2, N \in \mathbb{N}$ such that for all $n \geq N$ and for all $P \in \mathcal{P}$ both of the following hold:
        \begin{align*}
            \pr_P\big(\sup_{x,z} |f_P(x,z)-\tilde f^{(n,1)}(x,z)| > M_1\big) &< \epsilon/2;\\
            \pr_P\big(n^{\alpha} \E_P(\Delta_{P,n}^2 \given D^{(n,1)}) > M_2 \big) &< \epsilon/2.
        \end{align*}
        Taking $M= \max\big\{1+M_1^4, M_2\big\}$, a union bound gives that
		\begin{align*}
		 \pr_P\big(\big\{\sup_{x,z} |f_P(x,z)-\tilde f^{(n,1)}(x,z)| > (M-1)^{1/4}\big\}\, \cup \,  \big\{n^{\alpha} \E_P(\Delta_{P,n}^2 \given D^{(n,1)}) > M \big\}\big) < \epsilon.
		\end{align*}
		Now we have
		\begin{align*}
		(\tilde{B}_f^{(n)})^{2}  &= \E_P [\Delta_{P,n}^{4}(\ind_{\{\Delta_{P,n} \leq 1\}} +  \ind_{\{\Delta_{P,n} > 1\}}) \given D^{(n,1)} ]   \\
        &\leq \E_P [\Delta_{P,n}^{2}\ind_{\{\Delta_{P,n} \leq 1\}} \given D^{(n,1)} ] + \E_P[ \Delta_{P,n}^{4} \ind_{\{\Delta^2_{P,n} > 1\}} \given D^{(n,1)} ]   \\
		&\leq \E_P(\Delta_{P,n}^2 \given D^{(n,1)})  + \sup_{x,z} |f_P(x,z)-\tilde f^{(n,1)}(x,z)|^{4} \,\pr_P( \Delta_{P,n}^2 > 1 \given D^{(n,1)}) \\
		&\leq  \E_P(\Delta_{P,n}^2 \given D^{(n,1)}) ( 1 + \sup_{x,z} |f_P(x,z)-\tilde f^{(n,1)}(x,z)|^{4}).
		\end{align*}
		The second line uses the observation that for any $t\geq 0$ we either have $t^2\leq t\leq 1$ or $1\leq t\leq t^2$. third line uses the H\"older inequality and the fourth line the Markov inequality. Thus, for $P \in \mathcal{P}$ and $n \geq N$,
		\begin{align*}
			&\pr_P( n^{\alpha/2} \tilde{B}_f^{(n)} > M) =  \pr_P( n^{\alpha} (\tilde{B}_f^{(n)})^2 > M^2) \\
                &\qquad \leq \pr_P\Big(n^{\alpha}\E_P(\Delta_{P,n}^2 \given D^{(n,1)}) ( 1 + \sup_{x,z} |f_P(x,z)-\tilde f^{(n,1)}(x,z)|^{4}) > M^2 \Big)\\
                &\qquad \leq \pr_P\big( \big\{1 + \sup_{x,z} |f_P(x,z)-\tilde f^{(n,1)}(x,z)|^4 > M\big\}\, \cup \,  \big\{n^{\alpha} \E_P(\Delta_{P,n}^2 \given D^{(n,1)}) > M\big\} \big)\\
			&\qquad = \pr_P\big( \sup_{x,z} |f_P(x,z)-\tilde f^{(n,1)}(x,z)| > (M-1)^{1/4}\, \cup \,  n^{\alpha} \E_P(\Delta_{P,n}^2 \given D^{(n,1)}) > M \big) < \epsilon,
		\end{align*}
		as required.
	\end{proof}
	
	\section{Proofs relating to Section~\ref{sect:score}}
	Our proofs make use of the following representations of $	p_{\hat\varepsilon}(\epsilon)$ and $\rho_{\hat\varepsilon}(\epsilon)$. We first note that
	\begin{equation*}
		\varepsilon_P = \hat\varepsilon^{(n)} + u^{(n)}_\sigma(Z)\;\hat\varepsilon^{(n)} + u^{(n)}_m(Z),
	\end{equation*}
	where we recall
	\begin{equation*}
		u^{(n)}_\sigma(z) := \frac{\hat \sigma^{(n)}(z) - \sigma_P(z)}{\sigma_P(z)};\quad 
		u^{(n)}_m(z) := \frac{\hat m^{(n)}(z) - m_P(z)}{\sigma_P(z)}.
	\end{equation*}
	Recall that since we do not have access to samples of $\varepsilon_P$, only $\hat\varepsilon^{(n)}$, our goal is to show that the score functions of these two variables are similar. Conditionally on $D^{(n)}$, and for each fixed $\epsilon\in\R$ and $z\in\mathcal Z$, the estimated residual $\hat\varepsilon^{(n)}$ and covariates $Z$ have joint density
	\begin{align*}
		p_{\hat\varepsilon, Z}(\epsilon, z)   &= p_{e, Z}\Big(\epsilon + u_\sigma^{(n)}(z)\epsilon + u_m^{(n)}(z), z\Big)\\
		&= p_{e}\Big(\epsilon + u_\sigma^{(n)}(z)\epsilon + u_m^{(n)}(z)\Big)\;p_Z(z),
	\end{align*}
	where the first equality is via a change-of-variables and the second is using the independence of $\varepsilon$ and $Z$. Integrating over $z$, we have that the marginal density of $\hat\varepsilon^{(n)}$, conditionally on $D^{(n)}$, is
	\begin{equation*}
		p_{\hat\varepsilon}(\epsilon) = \E_P\Big[p_{e}\Big(\epsilon + u_\sigma^{(n)}(Z)\epsilon + u_m^{(n)}(Z)\Big) \;\Big|\; D^{(n)}\Big].
	\end{equation*}
	If $p_\epsilon'$ is bounded and $\E[|u_\sigma(Z)| \given  D^{(n)}]<\infty$ then the estimated residual score function is
	\begin{align*}
		\rho_{\hat\varepsilon}(\epsilon) &= \frac{p'_{\hat\varepsilon}(\epsilon)}{p_{\hat\varepsilon}(\epsilon)}\\
		&= \frac{\E_P\Big[\Big\{1+u_\sigma^{(n)}(Z)\Big\}p_{e}'\Big(\epsilon + u_\sigma^{(n)}(Z)\epsilon + u_m^{(n)}(Z)\Big) \;\Big|\; D^{(n)}\Big]}{\E_P\Big[p_{e}\Big(\epsilon + u_\sigma^{(n)}(Z)\epsilon + u_m^{(n)}(Z)\Big) \;\Big|\; D^{(n)}\Big]},
	\end{align*}
	by differentiating under the integral sign (see, for example, \citet[Thm.~20.4]{AliprantisAnalysis}). 
	\subsection{Proof of Theorem~\ref{thm:scoresubg}}
	\begin{proof}
		By \citet[Thm.~2.6]{WainwrightHighDim}, the moment bound is sufficient to show sub-Gaussianity. Note that when $X$ is symmetrically distributed, its density $p(\cdot)$ is anti-symmetric. Thus its score function $\rho(\cdot)$ is anti-symmetric, and so the random variable $\rho(X)$ is symmetrically distributed.
		
		We prove the moment bound by induction. Suppose it is true for all $1\leq j<k$ for some $k\geq 1$. By the product rule, \begin{align*}
			\big(\rho^{2k-1}(x)p(x)\big)' &= \rho^{2k-1}(x)p'(x) + (2k-1) \rho'(x)\rho^{2k-2}(x)p(x)\\
			&= \rho^{2k}(x)p(x) + (2k-1) \rho'(x)\rho^{2k-2}(x)p(x).
		\end{align*}
		Therefore for any $-\infty<a<b<\infty$ we have \begin{equation}
			\int_a^b \rho^{2k}(x)p(x)\; dx = \rho^{2k-1}(b)p(b)-\rho^{2k-1}(a)p(a)  - (2k-1)\int_a^b  \rho'(x)\rho^{2k-2}(x)p(x)\; dx. \label{eqn:scoreintbyparts}
		\end{equation}
		
		We have that $\E[\rho^{2k-2}(X)]<\infty$ by the induction hypothesis if $k\geq 2$ and trivially if $k=1$. By Lemma~\ref{lem:scoreboundaryconvergent} we can choose sequences $a_n\to-\infty$, $b_n\to\infty$ such that \begin{equation*}
			\lim_{n\to\infty} \big\{\rho^{2k-1}(b_n)p(b_n)-\rho^{2k-1}(a_n)p(a_n)\big\} = 0.
		\end{equation*} By H\"older's inequality, we have that \begin{align*}
			\int_\R  \big|\rho'(x)\rho^{2k-2}(x)p(x) \big|\; dx &\leq C \int_\R  \rho^{2k-2}(x)p(x) \; dx\\
			&\leq \begin{cases}
				C^k (2k-3)!! \text{ if $k\geq 2$ by the induction hypothesis};\\
				C \text{ if $k=1$}.
			\end{cases}
		\end{align*}
		Therefore dominated convergence gives \begin{align*}
			\lim_{n\to\infty} \bigg|(2k-1)\int_{a_n}^{b_n}  \rho'(x)\rho^{2k-2}(x)p(x)\; dx\bigg| &= \bigg|(2k-1)\int_\R  \rho'(x)\rho^{2k-2}(x)p(x)\; dx\bigg|\\
			&\leq C^k(2k-1)!!.
		\end{align*}
		
		Finally, we can assume without loss of generality that the sequences $(a_n)$ and $(b_n)$ are both monotone, for example by relabelling their monotone sub-sequences. Now, for each $x\in\R$ the sequence $\ind_{[a_n,b_n]}(x)\rho^{2k}(x)p(x)$ is increasing in $n$. The monotone convergence theorem thus gives \begin{equation*}
			\lim_{n\to\infty} \int_{a_n}^{b_n} \rho^{2k}(x)p(x)\; dx = \E\big[\rho^{2k}(X)\big].
		\end{equation*}
		Taking the limit in equation \eqref{eqn:scoreintbyparts} yields \begin{align*}
			\E\big[\rho^{2k}(X)\big] &= \lim_{n\to\infty} \int_{a_n}^{b_n} \rho^{2k}(x)p(x)\; dx \\
			&= \lim_{n\to\infty} \bigg\{\rho^{2k-1}(b_n)p(b_n)-\rho^{2k-1}(a_n)p(a_n) - (2k-1)\int_{a_n}^{b_n}  \rho'(x)\rho^{2k-2}(x)p(x)\; dx\bigg\}\\
			&\leq C^k(2k-1)!!,
		\end{align*}
		as claimed.
	\end{proof}

	
	\subsection{Proof of Theorem~\ref{thm:knownscoreestimation}}
	\begin{proof}
		Let us write $C_\rho := \sup_{\epsilon \in \R} |\rho_{e}'(\epsilon)|$.
		Define \begin{equation*}
			\bar \rho^{(n)}_P (x,z) := \frac{1}{\sigma_P(z)}\; \rho_{e}\bigg(\frac{x-\hat m^{(n)}(z)}{\sigma_P(z)}\bigg).
		\end{equation*}
		Using the inequality $(a+b)^2\leq 2(a^2+b^2)$, we have
		\begin{align*}
			A_\rho^{(n)} &= \E_P\Big[\big\{\rho_P(X,Z)-\bar \rho ^{(n)}_P(X,Z) + \bar \rho ^{(n)}_P(X,Z) - \hat\rho^{(n)}(X,Z)\big\}^2 \;\Big|\; D^{(n)}\Big]\\
			&\leq 2 \E_P\Big[\big\{ \rho_P(X,Z)-\bar \rho ^{(n)}_P(X,Z)\big\}^2 \;\Big|\; D^{(n)}\Big] + 2 \E_P\Big[\big\{\bar \rho ^{(n)}_P(X,Z)-\hat\rho^{(n)}(X,Z)\big\}^2 \;\Big|\; D^{(n)}\Big].
		\end{align*}
		The first term readily simplifies using H\"older's inequality and the Lipschitz property of $\rho_{e}$. Writing $c_{\sigma} := \inf_{P \in \mathcal{P}} \inf_z \sigma_P(z)$, we have 
		\begin{align*}
			\E_P\Big[\big\{\rho_P(X,Z)-\bar \rho ^{(n)}_P(X,Z)\big\}^2 \;\Big|\; D^{(n)}\Big] & = \E_P\Bigg[\frac{1}{\sigma_P^2(Z)}\bigg\{\rho_{e}\bigg(\frac{X-m_P(Z)}{\sigma_P(Z)}\bigg) - \rho_{e}\bigg(\frac{X-\hat m^{(n)}(Z)}{\sigma_P(Z)}\bigg)\bigg\}^2 \;\Bigg|\; D^{(n)}\Bigg]\\
			&\leq \bigg(\frac{C_{\rho}}{c_{\sigma}^2}\bigg)^2 \; \E_P\big[ u_m^{(n)2}(Z) \;\big|\; D^{(n)}\big]\\
			&= \bigg(\frac{C_{\rho}}{c_{\sigma}^2}\bigg)^2 \; A_m^{(n)}.
		\end{align*}

		Let us write
		\[
		\check{\rho}_P^{(n)} (x, z) := \frac{1}{\sigma_P(z)}\; \rho_{e}\bigg(\frac{x-\hat m^{(n)}(z)}{\hat{\sigma}^{(n)}(z)}\bigg).
		\]
		Then
		\begin{align*}
			& \E_P\Big[\big\{\bar \rho ^{(n)}_P(X,Z)-\hat\rho^{(n)}(X,Z)\big\}^2 \;\Big|\; D^{(n)}\Big] \\
			& \qquad \leq 2\E_P\Big[\big\{\bar \rho ^{(n)}_P(X,Z)-\check\rho^{(n)}(X,Z)\big\}^2 \;\Big|\; D^{(n)}\Big] + 2\E_P\Big[\big\{\check \rho ^{(n)}_P(X,Z)-\hat\rho^{(n)}(X,Z)\big\}^2 \;\Big|\; D^{(n)}\Big].
		\end{align*}
		Now, given any $\epsilon > 0$, let $M>0$ and $N$ be such that for all $n \geq N$,
		\[
		\sup_{P \in \mathcal{P}} \pr_P\left(\sup_z \frac{\sigma(z)}{\hat{\sigma}(z)} > M \right) < \epsilon/2
		\]
		and
		\[
		\sup_{P \in \mathcal{P}} \pr_P(\sup_{z} u_m^{(n)}(z) >M ) < \epsilon/2.
		\]
		In the following, we work on the $D^{(n)}$ measurable event
		\[
		\Omega_n := \left\{ \sup_z \frac{\sigma_P(z)}{\hat{\sigma}^{(n)}(z)} \leq  M \right\} \cap \{\sup_{z}u_m^{(n)}(z) \leq M\}
		\]
		which by the above, has probability at least $1-\epsilon$ under all $P \in \mathcal{P}$ and for all $n \geq N$.
		Note that on this event, for each $P \in \mathcal{P}$
		\[
		M \geq \frac{\inf_{z} \sigma_P(z)}{ \inf_{z} \hat{\sigma}^{(n)}(z)} \geq \frac{c_\sigma}{\inf_{z} \hat{\sigma}^{(n)}(z)}, 
		\]
		so $\sup_z \hat{\sigma}^{(n),-1} (z) \leq M / c_{\sigma}$. 
		Now, noting that
		\[
		\frac{1}{\sigma_P(Z)} - \frac{1}{\hat{\sigma}^{(n)}(Z)} = \frac{1}{\hat{\sigma}^{(n)}(Z)}u_{\sigma}^{(n)},
		\]
		we have
		\begin{align*}
			\{\check{\rho}_P^{(n)}(X,Z) - \hat{\rho}^{(n)}(X, Z)\}^2 
            &=\left(\frac{1}{\sigma_P(Z)} - \frac{1}{\hat{\sigma}^{(n)}(Z)}\right)^2 \rho_{e}^2\left( \frac{\sigma_P(Z)}{\hat{\sigma}^{(n)}(Z)} \varepsilon - u_m^{(n)}(Z) \frac{\sigma_P(Z)}{\hat{\sigma}^{(n)}(Z)} \right)   \\
			& = \frac{u_{\sigma}^{(n),2}(Z)}{\hat{\sigma}^{(n),2}(Z)} \left\{ \rho_{e}\left( \frac{\sigma_P(Z)}{\hat{\sigma}^{(n)}(Z)} \varepsilon + u_m^{(n)}(Z) \frac{\sigma_P(Z)}{\hat{\sigma}^{(n)}(Z)} \right) - \rho_{e}(\varepsilon) + \rho_{e}(\varepsilon) \right\}^2 \\
				&\leq 2\frac{M^2}{c_{\sigma}^2} u_{\sigma}^{(n),2}(Z)\left\{2(M + 1)^2(\varepsilon^2 + M^2)C_\rho^2 + \rho_{e}^2 (\varepsilon)\right\}.
		\end{align*}
		Thus since $\epsilon$ was arbitrary, and using that $\varepsilon \independent Z$, we have that
		\[
		\E_P \left( 	\{\check{\rho}_P^{(n)}(X,Z) - \hat{\rho}^{(n)}(X, Z)\}^2 \given D^{(n)}\right) = O_{\mathcal{P}}(A_{\sigma}^{(n)}).
		\]
		
		Next,
		\begin{align*}
			\{\check{\rho}^{(n)}_P(X, Z) - \bar{\rho}^{(n)}_P(X, Z)\}^2 &= \frac{1}{\sigma_P^2(Z)}\left\{ \rho_{e}\left( \frac{X - \hat{m}^{(n)}(Z)}{\hat{\sigma}^{(n)}(Z)}\right) - \rho_{e}\left( \frac{X - \hat{m}^{(n)}(Z)}{\sigma_P(Z)}\right)\right\}^2 \\
			&\leq \frac{C_\rho^2}{\sigma_P^2(Z) }\left(\frac{1}{\hat{\sigma}^{(n)}(Z)} - \frac{1}{\sigma_P(Z)} \right)^2 \{m_P(Z) - \hat{m}^{(n)}(Z) + \sigma_{P}(Z) \varepsilon \}^2\\
			&\leq \frac{2C_\rho^2}{\hat{\sigma}^{(n),2}(Z)} u_{\sigma}^{(n),2}(Z) ( M^2 + \varepsilon^2) \\
			&\leq \frac{2 C_\rho M^2}{c_{\sigma}^2} u_{\sigma}^{(n),2}(Z) ( M^2 + \varepsilon^2).
		\end{align*}
		Thus using  that $\varepsilon \independent Z$, we have that
		\[
		\E_P \left( 	\{\check{\rho}^{(n)}_P(X, Z) - \bar{\rho}^{(n)}_P(X, Z)\}^2 \given D^{(n)}\right) = O_{\mathcal{P}}(A_{\sigma}^{(n)}).
		\]
		Putting things together gives the result.
	\end{proof}

	\subsection{Proof of Theorem~\ref{thm:subgscoreestimation}}
	
	\begin{proof}
		Let us set $C_{\sigma} := \{18 C_{\varepsilon} \sqrt{C_{\rho}}\}^{-1}$.
		The assumptions on $u_m^{(n)}$ and $u_\sigma^{(n)}$ mean that for any $\epsilon > 0$ we can find $N \in  \mathbb{N}$ and $M, C_m >0 $ such that for any $n\geq N$, with uniform probability at least $1-\epsilon$, the data $D^{(n)}$ is such that
		\begin{equation}
			\label{eqn:locscaleoundederrors}
			\sup_{z}\big|u_m^{(n)}(z)\big| \leq C_m, \quad \sup_{z}\big|u_\sigma^{(n)}(z)\big| \leq C_\sigma \quad \text{and} \quad \sup_z \hat{\sigma}^{(n),-1}(z) \leq M;
		\end{equation} 
		see the proof of Theorem~\ref{thm:knownscoreestimation}.
		It suffices to show that on this sequence  of events, we can find a uniform constant $C$ (not depending on $P$ or $n$) such that \begin{equation*}
			A_\rho^{(n)} \leq C\big(A_m^{(n)} + A_\sigma^{(n)} + A_{\hat\varepsilon}^{(n)}\big).
		\end{equation*}
		
		Fix $P\in\mathcal{P}$ and $D^{(n)}$ such that \eqref{eqn:locscaleoundederrors} holds. We decompose $A_\rho^{(n)}$ so as to consider the various sources of error separately.
		\begin{align*}
			A_\rho^{(n)} &= \E_P\Big[\big\{\rho_{P}(X,Z)-\hat \rho^{(n)}(X,Z)\big\}^2 \; \Big| \; D^{(n)}\Big]\\
			&=\E_P\Bigg[\bigg\{\frac1{\sigma_P(Z)}\rho_{e}(\varepsilon_P)-\frac{1}{\hat\sigma^{(n)}(Z)} \hat \rho_{\hat\varepsilon}^{(n)}\big(\hat\varepsilon^{(n)}\big)\bigg\}^2 \; \Bigg| \; D^{(n)}\Bigg]\\
			&=\E_P\Bigg[\frac1{\hat\sigma^{(n)2}(Z)}\bigg\{\bigg(\frac{\hat\sigma^{(n)}(Z)}{\sigma_P(Z)}-1\bigg) \rho_{e}(\varepsilon_P)+ \rho_{e}(\varepsilon_P) - \hat \rho_{\hat\varepsilon}^{(n)}\big(\hat\varepsilon^{(n)}\big)\bigg\}^2 \; \Bigg| \; D^{(n)}\Bigg]\\
			&=\E_P\Bigg[\frac1{\hat\sigma^{(n)2}(Z)}\bigg\{u_\sigma^{(n)}(Z) \rho_{e}(\varepsilon_P)+ \rho_{e}(\varepsilon_P) - \rho_{\hat\varepsilon}(\varepsilon_P) + \rho_{\hat\varepsilon}(\varepsilon_P) - \rho_{\hat\varepsilon}(\hat\varepsilon^{(n)}) \\
			&\phantom{=} + \rho_{\hat\varepsilon}(\hat\varepsilon^{(n)}) -\hat \rho_{\hat\varepsilon}^{(n)}\big(\hat\varepsilon^{(n)}\big)\bigg\}^2 \; \Bigg| \; D^{(n)}\Bigg].
		\end{align*}
		Applying H\"older's inequality and $(a+b+c+d)^2\leq 4(a^2+b^2+c^2+d^2)$, we deduce
		\begin{align}
			A_\rho^{(n)} &\leq 4 M^2 \bigg\{\E_P\big[u_\sigma^{(n)2}(Z) \rho_{e}^2(\varepsilon_P) \; \big|\; D^{(n)}\big] \nonumber\\
			&\phantom{\leq 4 M^2 \bigg\{} + \E_P\Big[ \big\{ \rho_{e}(\varepsilon_P) - \rho_{\hat\varepsilon}(\varepsilon_P) \big\}^2 \; \Big| \; D^{(n)} \Big] \nonumber\\
			& \phantom{\leq 4 M^2 \bigg\{} + \E_P\Big[ \big\{  \rho_{\hat\varepsilon}(\varepsilon_P) - \rho_{\hat\varepsilon}(\hat\varepsilon^{(n)}) \big\}^2 \; \Big| \; D^{(n)} \Big] \nonumber\\
			&\phantom{\leq 4 M^2 \bigg\{} + \E_P\Big[ \big\{  \rho_{\hat\varepsilon}(\hat\varepsilon^{(n)}) -\hat \rho_{\hat\varepsilon}^{(n)}\big(\hat\varepsilon^{(n)}\big)\big\}^2 \; \Big| \; D^{(n)} \Big]\bigg\}\label{eqn:locscalesubgdecomp}.
		\end{align}
		
		We consider the expectations in \eqref{eqn:locscalesubgdecomp} separately. For the first term, the independence of $\varepsilon_P$ and $Z$ and Theorem~\ref{thm:scoresubg} imply \begin{equation*}
			\E_P\big[u_\sigma^{(n)2}(Z) \rho_{e}^2(\varepsilon_P) \; \big|\; D^{(n)}\big] = \E_P\big[\rho_{e}^2(\varepsilon_P)\big] \; A_\sigma^{(n)} \leq C_\rho A_\sigma^{(n)}. 
		\end{equation*} Lemma~\ref{lem:estimatedresidualscore} applies to the second term. To apply Lemma~\ref{lem:estimatedresidual} to the third term, we note that \begin{equation*}
			\big(\E_P(\varepsilon_P^8)\big)^{\frac18} \leq (768 C_\varepsilon^8)^{\frac18} < 3C_\varepsilon
		\end{equation*}
		by Lemma~\ref{lem:subgmoment}. The fourth term is equal to $A_{\hat\varepsilon}^{(n)}$ by definition. This completes the proof.
	\end{proof}

    \subsection{Proof of Theorem~\ref{thm:homscoreestimation}}
	
	\begin{proof}
		The assumption on $u_m^{(n)}$ means that for any $\epsilon > 0$ we can find $N, C_m$ such that for any $n\geq N$, with uniform probability at least $1-\epsilon$, the data $D^{(n)}$ is such that \begin{equation}
			\label{eqn:locationerrors}
			\sup_{z}\big|u_m^{(n)}(z)\big| \leq C_m.
		\end{equation} 
		It suffices to show that under this event, we can find a uniform constant $C$ (not depending on $P$ or $n$) such that \begin{equation*}
			A_\rho^{(n)} \leq C\big(A_m^{(n)} + A_{\hat\varepsilon}^{(n)}\big).
		\end{equation*}
		
		Fix $P\in\mathcal{P}$ and $D^{(n)}$ such that \eqref{eqn:locationerrors} holds. We decompose $A_\rho^{(n)}$ so as to consider the various sources of error separately.
		\begin{align}
			A_\rho^{(n)} &= \E_P\Big[\big\{\rho_{e}(\varepsilon_P)-\hat \rho_{\hat\varepsilon}^{(n)}\big(\hat\varepsilon^{(n)}\big)\big\}^2 \; \Big| \; D^{(n)}\Big] \nonumber\\
			&= \E_P\Big[\big\{\rho_{e}(\varepsilon_P) - \rho_{\hat\varepsilon}(\varepsilon_P) + \rho_{\hat\varepsilon}(\varepsilon_P) - \rho_{\hat\varepsilon}(\hat\varepsilon^{(n)}) + \rho_{\hat\varepsilon}(\hat\varepsilon^{(n)}) -\hat \rho_{\hat\varepsilon}^{(n)}\big(\hat\varepsilon^{(n)}\big)\big\}^2 \; \Big| \; D^{(n)}\Big] \nonumber\\
			&= 3 \E_P\Big[\big\{\rho_{e}(\varepsilon_P)-\rho_{\hat\varepsilon}(\varepsilon_P)\big\}^2 \; \Big| \; D^{(n)}\Big] \nonumber\\
			&\phantom{=} + 3 \E_P\Big[\big\{\rho_{\hat\varepsilon}(\varepsilon_P)-\rho_{\hat\varepsilon}(\hat\varepsilon^{(n)})\big\}^2 \; \Big| \; D^{(n)}\Big]  \nonumber\\
			&\phantom{=} + 3 \E_P\Big[ \big\{  \rho_{\hat\varepsilon}(\hat\varepsilon^{(n)}) -\hat \rho_{\hat\varepsilon}^{(n)}\big(\hat\varepsilon^{(n)}\big)\big\}^2 \; \Big| \; D^{(n)} \Big] \label{eqn:locationdecomp},
		\end{align}
		where the final inequality is $(a+b+c)^2\leq 3(a^2+b^2+c^2)$.
		
		We consider the expectations in \eqref{eqn:locationdecomp} separately. Lemma~\ref{lem:estimatedresidualscore_locationonly} applies to the first term. Lemma~\ref{lem:estimatedresidual_locationonly} applies to the second term. The third term is equal to $A_{\hat\varepsilon}^{(n)}$ by definition. This completes the proof.
	\end{proof}
	
	\subsection{Auxiliary lemmas}

	\begin{lemma}
		\label{lem:estimatedresidualscore}
		Let $P$ be such that $ p_{e}$ is twice differentiable on $\R$, with \begin{equation*}
			\sup_{\epsilon\in\R} |\partial_\epsilon^2 \log p_{e}(\epsilon)| = \sup_{\epsilon\in\R} |\rho_{e}'(\epsilon)|\leq C_\rho,
		\end{equation*}
		$p_{e}'$ is bounded, and $\varepsilon_P$ is sub-Gaussian with parameter $C_\varepsilon$. Further assume that $D^{(n)}$ is such that  $\sup_{z\in\mathcal{Z}} |u_m^{(n)}(z)|\leq C_m$  and $\sup_{z\in\mathcal{Z}} |u_\sigma^{(n)}(z)|\leq C_\sigma$. If
		\begin{equation*}
			\sqrt{C_\rho}C_\sigma C_\varepsilon \leq \frac1{18}
		\end{equation*}
		then there exists a constant $C$, depending only on $C_\rho, C_m, C_\sigma, C_\varepsilon$, such that
		\begin{equation*}
			\E_P\Big[\big\{\rho_{e}(\varepsilon_P)-\rho_{\hat\varepsilon}(\varepsilon_P)\big\}^2 \; \Big| \; D^{(n)}\Big] \leq C \big(A_m^{(n)} + A_\sigma^{(n)}\big).
		\end{equation*}
	\end{lemma}
	\begin{proof}
		For ease of notation, write $Q^{(n)}$ for the distribution of $\big(u_m^{(n)}(Z), u_\sigma^{(n)}(Z)\big)$ conditionally on $D^{(n)}$, and let $(U_m, U_\sigma)\sim Q^{(n)}$.
		Therefore
		\begin{equation*}
			A_m^{(n)} = \E_{Q^{(n)}} \big(U_m^2\big) \quad \text{and} \quad A_\sigma^{(n)} = \E_{Q^{(n)}} \big(U_\sigma^2\big).
		\end{equation*}

		The conditions on $p_{e}$ and $U_\sigma$ are sufficient to interchange differentiation and expectation operators as follows \citep[Thm.~20.4]{AliprantisAnalysis}. 
		\begin{align*}
			\rho_{\hat\varepsilon}(\epsilon)  &= \frac{\frac{\partial}{\partial \epsilon}\E_{Q^{(n)}}\big[p_{e}(\epsilon + U_\sigma\epsilon + U_m) \big]}{\E_{Q^{(n)}}\big[p_{e}(\epsilon + U_\sigma\epsilon + U_m) \big]}\\
			&=  \frac{\E_{Q^{(n)}}\big[(1+U_\sigma)p_{e}'(\epsilon + U_\sigma\epsilon + U_m) \big]}{\E_{Q^{(n)}}\big[p_{e}(\epsilon + U_\sigma\epsilon + U_m) \big]}.
		\end{align*}
		We may then decompose the approximation error as follows.
		\begin{align*}
			|\rho_{\hat\varepsilon}(\epsilon) - \rho_{e}(\epsilon)| &= \bigg|\frac{\E_{Q^{(n)}} [(1+U_\sigma)p_{e}'(\epsilon + U_\sigma\epsilon + U_m)] }{\E_{Q^{(n)}} [p_{e}(\epsilon + U_\sigma\epsilon + U_m)]} - \rho_{e}(\epsilon) \bigg| \\
			&=\bigg|\frac{\E_{Q^{(n)}} [(1+U_\sigma)\{\rho_{e}(\epsilon + U_\sigma\epsilon + U_m) - \rho_{e}(\epsilon) \}\;p_{e}(\epsilon + U_\sigma\epsilon + U_m)] }{\E_{Q^{(n)}} [p_{e}(\epsilon + U_\sigma\epsilon + U_m)]} \\
			&\phantom{=\bigg|} + \frac{\E_{Q^{(n)}} [U_\sigma\;p_{e}(\epsilon + U_\sigma\epsilon + U_m)] }{\E_{Q^{(n)}} [p_{e}(\epsilon + U_\sigma\epsilon + U_m)]} \rho_{e}(\epsilon)\bigg| \\
			&\leq C_\rho \frac{\E_{Q^{(n)}} [|(1+U_\sigma)( U_\sigma\epsilon + U_m)| \;p_{e}(\epsilon + U_\sigma\epsilon + U_m)] }{\E_{Q^{(n)}} [p_{e}(\epsilon + U_\sigma\epsilon + U_m)]} \\
			&\phantom{\leq} + \frac{\E_{Q^{(n)}} [|U_\sigma|\;p_{e}(\epsilon + U_\sigma\epsilon + U_m)] }{\E_{Q^{(n)}} [p_{e}(\epsilon + U_\sigma\epsilon + U_m)]} |\rho_{e}(\epsilon)| \\
			&\leq \bigg\{C_\rho \Big(\E_{Q^{(n)}} \big[(1+U_\sigma)^2( U_\sigma\epsilon + U_m)^2\big]\Big)^{1/2} + |\rho_{e}(\epsilon)|\Big( \E_{Q^{(n)}}\big(U_\sigma^2\big)\Big)^{1/2} \bigg\} \\
			&\phantom{\leq} \cdot \frac{\Big(\E_{Q^{(n)}} \big[p^2_\varepsilon(\epsilon + U_\sigma\epsilon + U_m)\big]\Big)^{1/2} }{\E_{Q^{(n)}} [p_{e}(\epsilon + U_\sigma\epsilon + U_m)]}\\
			&=: R_1(\epsilon) R_2(\epsilon).
		\end{align*}
		The first inequality uses the Lipschitz property of $\rho_{e}$. The second applies the Cauchy--Schwarz inequality.
		
		We will show that the first term in the product is dominated by $\E_{Q^{(n)}}\big(U_\sigma^2\big) + \E_{Q^{(n)}}\big(U_m^2\big)$, and that the second term is bounded. Indeed,
		\begin{align*}
			R_1^2(\epsilon) &\leq 2C_\rho^2 \E_{Q^{(n)}} \big[(1+U_\sigma)^2( U_\sigma\epsilon + U_m)^2\big] + 2\rho_{e}^2(\epsilon)\E_{Q^{(n)}}\big(U_\sigma^2\big)\\
			&\leq 2C_\rho^2(1+C_\sigma)^2 \E_{Q^{(n)}} \big[( U_\sigma\epsilon + U_m)^2\big] + 2\rho_{e}^2(\epsilon)\E_{Q^{(n)}}\big(U_\sigma^2\big)\\
			&\leq  4C_\rho^2(1+C_\sigma)^2 \E_{Q^{(n)}} \big( U_\sigma^2\big)\epsilon^2 + 4C_\rho^2(1+C_\sigma)^2 \E_{Q^{(n)}} \big(U_m^2\big) + 2\rho_{e}^2(\epsilon)\E_{Q^{(n)}}\big(U_\sigma^2\big).
		\end{align*}
		The first and third inequalities are $(a+b)^2\leq 2(a^2+b^2)$, and the second uses the bound $|U_\sigma|\leq C_\sigma$.
		
		For any $\epsilon\in\R$ 
		and for any constant $c_1>0$ (to be chosen later),
		\begin{align*}
			R_2^2(\epsilon) &\leq \bigg(\frac{\sup_{|u_m|\leq C_m\;, \; |u_\sigma|\leq C_\sigma} p_{e}(\epsilon+u_\sigma \epsilon + u_m) / p_{e}(\epsilon)}{\inf_{|u_m|\leq C_m\;, \; |u_\sigma|\leq C_\sigma} p_{e}(\epsilon+u_\sigma \epsilon + u_m) / p_{e}(\epsilon)}\bigg)^2\\
			&\leq \exp\big\{4|C_m+|C_\sigma\epsilon|| \;|\rho_{e}(\epsilon)| + 2C_\rho(C_m+|C_\sigma\epsilon|)^2\big\}\\
			&\leq \exp\bigg\{ \frac{\rho_{e}^2(\epsilon)}{C_\rho c_1} + (4c_1+2)C_\rho(C_m+|C_\sigma\epsilon|)^2\bigg\}.
		\end{align*}
		The first line is a supremum bound for the ratio of expectations, the second is the application of Lemma~\ref{lem:densityratiobound}, and the third uses that for all $c>0$,
		\begin{equation*}
			0\leq \bigg(\frac{a}{\sqrt{c}} - 2\sqrt{c}b \bigg)^2 \implies 4ab = 2\bigg(\frac{a}{\sqrt{c}}\bigg)(2\sqrt{c}b) \leq \frac{a^2}{c} + 4cb^2.
		\end{equation*}
		
		Using the above and H\"older's inequality, we have that for any $c_2>0$ (to be chosen later),
		\begin{align*}
			\E_P\Big[\big\{\rho_{e}(\varepsilon_P)-\rho_{\hat\varepsilon}(\varepsilon_P)\big\}^2 \; \Big| \; D^{(n)}\Big] &\leq \E_P\big[R_1^2(\varepsilon_P)R_2^2(\varepsilon_P) \; \big| \; D^{(n)}\big]\\
			&\leq \bigg(\E_P\bigg[R_1^{\frac{2(1+c_2)}{c_2}}(\varepsilon_P)  \; \bigg| \; D^{(n)} \bigg]\bigg)^{\frac{c_2}{1+c_2}} \Big(\E_P\big[R_2^{2(1+c_2)}(\varepsilon_P)  \; \big| \; D^{(n)} \big]\Big)^{\frac1{1+c_2}}.
		\end{align*}
		
		By the triangle inequality (for the $L_{(1+c_2)/c_2}(P)$ norm), 
		\begin{align*}
			\bigg(\E_P\bigg[R_1^{\frac{2(1+c_2)}{c_2}}(\varepsilon_P)  \; \bigg| \; D^{(n)} \bigg]\bigg)^{\frac{c_2}{1+c_2}} &\leq 4C_\rho^2(1+C_\sigma)^2 \E_{Q^{(n)}} \big( U_\sigma^2\big)\bigg(\E_P\bigg[\varepsilon_P^{\frac{2(1+c_2)}{c_2}}\bigg]\bigg)^{\frac{c_2}{1+c_2}} \\
			&\phantom{\leq} + 4C_\rho^2(1+C_\sigma)^2 \E_{Q^{(n)}} \big(U_m^2\big) \\
			&\phantom{\leq} +  2\E_{Q^{(n)}}\big(U_\sigma^2\big)\bigg(\E_P\bigg[\rho_{e}^{\frac{2(1+c_2)}{c_2}}(\varepsilon_P)\bigg]\bigg)^{\frac{c_2}{1+c_2}}.
		\end{align*}
		
		By H\"older's inequality, for any $c_3>0$ (to be chosen later),
		\begin{align*}
			\E_P\big[R_2^{2(1+c_2)}(\varepsilon_P)  \; \big| \; D^{(n)} \big] & \leq \E_P\bigg[\exp\bigg\{ \frac{(1+c_2)\rho_{e}^2(\varepsilon_P)}{C_\rho c_1} + (1+c_2)(4c_1+2)C_\rho(C_m+|C_\sigma\varepsilon_P|)^2\bigg\}\bigg]\\
			&\leq \bigg(\E_P\bigg[\exp\bigg\{ \frac{(1+c_3)(1+c_2)\rho_{e}^2(\varepsilon_P)}{C_\rho c_1} \bigg\}\bigg]\bigg)^{\frac1{1+c_3}} \\
			&\phantom{=} \cdot \bigg(\E_P\bigg[ \exp\bigg\{\frac{(1+c_3)(1+c_2)(4c_1+2)C_\rho(C_m+|C_\sigma\varepsilon_P|)^2}{c_3}\bigg\}\bigg]\bigg)^{\frac{c_3}{1+c_3}}\\
			&\leq \bigg(\E_P\bigg[\exp\bigg\{ \frac{(1+c_3)(1+c_2)\rho_{e}^2(\varepsilon_P)}{C_\rho c_1} \bigg\}\bigg]\bigg)^{\frac1{1+c_3}} \\
			&\phantom{=} \cdot \bigg(\E_P\bigg[ \exp\bigg\{\frac{2(1+c_3)(1+c_2)(4c_1+2)C_\rho(C_m^2+C_\sigma^2\varepsilon_P^2)}{c_3}\bigg\}\bigg]\bigg)^{\frac{c_3}{1+c_3}}\\
			&=: \Bigg(\E_P\Bigg[\exp\Bigg\{ \frac{\lambda_\rho \rho_{e}^2(\varepsilon_P)}{2\big(\sqrt{2C_\rho}\big)^2} \Bigg\} \Bigg]\Bigg)^{\frac1{1+c_3}}  \exp\big\{2(1+c_2)(4c_1+2)C_\rho C_m^2)\big\}\\
			&\phantom{=} \cdot  \bigg(\E_P\bigg[ \exp\bigg\{\frac{\lambda_\varepsilon \varepsilon_P^2}{2C_\varepsilon^2}\bigg\}\bigg]\bigg)^{\frac{c_3}{1+c_3}};
		\end{align*}
		the final inequality uses $(a+b)^2 \leq 2(a^2+b^2)$ and the monotonicity of the exponential function; and in the final equality the newly defined quantities are
		\begin{align*}
			\lambda_\rho &:= \frac{4(1+c_3)(1+c_2)}{c_1},\\
			\lambda_\varepsilon &:= \frac{4C_\varepsilon^2 (1+c_3)(1+c_2)(4c_1+2)C_\rho C_\sigma^2}{c_3}.
		\end{align*}
		To apply Lemma~\ref{lem:subgsquaremgf}, we must choose $c_1,c_2,c_3>0$ such that both $\lambda_\rho, \lambda_\varepsilon \in [0,1)$. The choice
		\begin{equation*}
			(c_1, c_2, c_3) = \bigg(9, \frac{1}{16}, 1\bigg)
		\end{equation*}
		suffices for 
		\begin{equation*}
			\lambda_\rho = \frac{17}{18},\quad \lambda_\varepsilon \leq 1-\frac1{18^2}.
		\end{equation*}
		Hence \begin{align*}
			\E_P\Big[\big\{\rho_{e}(\varepsilon_P)-\rho_{\hat\varepsilon}(\varepsilon_P)\big\}^2 \; \Big| \; D^{(n)}\Big] &\leq 
			\Bigg\{ 4C_\rho^2(1+C_\sigma)^2 \E_{Q^{(n)}} \big( U_\sigma^2\big)\Big(\E_P\big(\varepsilon_P^{34}\big)\Big)^{\frac{1}{17}} \\
			&\phantom{\leq \bigg\{ }+ 4C_\rho^2(1+C_\sigma)^2 \E_{Q^{(n)}} \big(U_m^2\big) \\
			&\phantom{\leq \bigg\{ }+ 2\E_{Q^{(n)}}\big(U_\sigma^2\big)\Big(\E_P\big[\rho_{e}^{34}(\varepsilon_P)\big]\Big)^{\frac{1}{17}} \bigg\} \\
			&\phantom{\leq \Bigg\{ } \cdot \Bigg(\E_P\Bigg[\exp\Bigg\{ \frac{\lambda_\rho \rho_{e}^2(\varepsilon_P)}{2\big(\sqrt{2C_\rho}\big)^2} \Bigg\} \Bigg]\Bigg)^{\frac8{17}}  \exp\big(76 C_\rho C_m^2\big)  \\
			&\phantom{\leq \Bigg\{ } \cdot\bigg(\E_P\bigg[ \exp\bigg\{\frac{\lambda_\varepsilon \varepsilon_P^2}{2C_\varepsilon^2}\bigg\}\bigg]\bigg)^{\frac8{17}}.
		\end{align*}
		Finally, by Theorem~\ref{thm:scoresubg} and Lemmas~\ref{lem:subgmoment} and \ref{lem:subgsquaremgf} we have the bounds \begin{align*}
			\Big(\E_P\big[\rho_{e}^{34}(\varepsilon_P)\big]\Big)^{\frac{1}{17}} &\leq C_\rho (33!!)^\frac{1}{17} < 13 C_\rho;\\
			\Big(\E_P\big(\varepsilon_P^{34}\big)\Big)^{\frac{1}{17}}  &\leq \big(34 \cdot 2^{17} C_\varepsilon^{34} \Gamma(17)\big)^\frac1{17} = 2C_\varepsilon^2(34\cdot 16!)^\frac1{17} < 15C_\varepsilon^2;\\
			\Bigg(\E_P\Bigg[\exp\Bigg\{ \frac{\lambda_\rho \rho_{e}^2(\varepsilon_P)}{2\big(\sqrt{2C_\rho}\big)^2} \Bigg\} \Bigg]\Bigg)^{\frac8{17}} &\leq \bigg(\frac{1}{\sqrt{1-\lambda_\rho}}\bigg)^{\frac8{17}} <4;\\
			\bigg(\E_P\bigg[ \exp\bigg\{\frac{\lambda_\varepsilon \varepsilon_P^2}{2C_\varepsilon^2}\bigg\}\bigg]\bigg)^{\frac8{17}} & \leq \bigg(\frac{1}{\sqrt{1-\lambda_\varepsilon}}\bigg)^{\frac8{17}} <2.
		\end{align*}
		This gives the final bound 
		\begin{align*}
			\E_P\Big[\big\{\rho_{e}(\varepsilon_P)-\rho_{\hat\varepsilon}(\varepsilon_P)\big\}^2 \; \Big| \; D^{(n)}\Big] &\leq 
			\big\{480C_\rho^2(1+C_\sigma)^2C_\varepsilon^2 + 208 C_\rho\big\} \exp\big(76 C_\rho C_m^2\big) \;\E_{Q^{(n)}}\big(U_\sigma^2\big)\\
			&\phantom{\leq} + 32C_\rho^2(1+C_\sigma)^2 \exp\big(76 C_\rho C_m^2\big) \;\E_{Q^{(n)}} \big(U_m^2\big)\\
			&=
			\big\{480C_\rho^2(1+C_\sigma)^2C_\varepsilon^2 + 208C_\rho\big\} \exp\big(76 C_\rho C_m^2\big) \;A_\sigma^{(n)}\\
			&\phantom{\leq} + 32C_\rho^2(1+C_\sigma)^2 \exp\big(76 C_\rho C_m^2\big) \;A_m^{(n)}.
		\end{align*}
		
	\end{proof}

	\begin{lemma}
		\label{lem:estimatedresidual}
		Let $P$ be such that $p_{e}$ is twice differentiable on $\R$, with \begin{equation*}
			\sup_{\epsilon\in\R} |\partial_\epsilon^2 \log p_{e}(\epsilon)| = \sup_{\epsilon\in\R} |\rho_{e}'(\epsilon)|\leq C_\rho,
		\end{equation*}
		$p_{e}'$ and $ p_{e}''$ are both bounded, and $\big(\E_P(\varepsilon_P^8)\big)^{\frac18} = C_\varepsilon < \infty$. Further assume that $D^{(n)}$ is such that $\sup_{z\in\mathcal{Z}} |u_m^{(n)}(z)|\leq C_m$  and $\sup_{z\in\mathcal{Z}} |u_\sigma^{(n)}(z)|\leq C_\sigma$ for almost every $z\in\mathcal{Z}$. 
		Then there exists a constant $C$, depending only on $C_\rho, C_m, C_\sigma, C_\varepsilon$, such that
		\begin{equation*}
			\E_P\Big[\big\{\rho_{\hat\varepsilon}(\varepsilon_P)-\rho_{\hat\varepsilon}(\hat\varepsilon^{(n)})\big\}^2 \; \Big| \; D^{(n)}\Big] \leq C \big(A_m^{(n)} + A_\sigma^{(n)}\big).
		\end{equation*}
	\end{lemma}
	\begin{proof}
		For ease of notation, write $Q^{(n)}$ for the distribution of $\big(u_m^{(n)}(Z), u_\sigma^{(n)}(Z)\big)$ conditionally on $D^{(n)}$, and let $(U_m, U_\sigma)\sim Q^{(n)}$. Therefore \begin{equation*}
			A_m^{(n)} = \E_{Q^{(n)}} \big(U_m^2\big); \quad A_\sigma^{(n)} = \E_{Q^{(n)}} \big(U_\sigma^2\big).
		\end{equation*}
		
		The proof proceeds by first bounding the derivative of $\rho_{\hat\varepsilon}$. The conditions on $p_{e}$ and $U_\sigma$ are sufficient to interchange differentiation and expectation operators as follows \citep[Thm.~20.4]{AliprantisAnalysis}. 
		\begin{align*}
			\rho_{\hat\varepsilon}(\epsilon)  &= \frac{\frac{\partial}{\partial \epsilon}\E_{Q^{(n)}}\big[p_{e}(\epsilon + U_\sigma\epsilon + U_m) \big]}{\E_{Q^{(n)}}\big[p_{e}(\epsilon + U_\sigma\epsilon + U_m) \big]}\\
			&=  \frac{\E_{Q^{(n)}}\big[(1+U_\sigma)p_{e}'(\epsilon + U_\sigma\epsilon + U_m) \big]}{\E_{Q^{(n)}}\big[p_{e}(\epsilon + U_\sigma\epsilon + U_m) \big]},
		\end{align*}
		and further,
		\begin{align*}
			\rho_{\hat\varepsilon}'(\epsilon)  &=  \frac{\partial}{\partial \epsilon} \frac{\E_{Q^{(n)}}\big[(1+U_\sigma)p_{e}'(\epsilon + U_\sigma\epsilon + U_m) \big]}{\E_{Q^{(n)}}\big[p_{e}(\epsilon + U_\sigma\epsilon + U_m) \big]}\\
			&=  \frac{\frac{\partial}{\partial \epsilon} \E_{Q^{(n)}}\big[(1+U_\sigma)p_{e}'(\epsilon + U_\sigma\epsilon + U_m) \big]}{\E_{Q^{(n)}}\big[p_{e}(\epsilon + U_\sigma\epsilon + U_m) \big]} - \rho_{\hat\varepsilon}^2(\epsilon)\\
			&= \frac{\E_{Q^{(n)}}\big[ (1+U_\sigma)^2p_{e}''(\epsilon+U_\sigma \epsilon + U_m) \big]}{\E_{Q^{(n)}}\big[ p_{e}(\epsilon+U_\sigma \epsilon + U_m) \big]}\\
			&\phantom{=} - \frac{\Big(\E_{Q^{(n)}}\big[ (1+U_\sigma)p_{e}'(\epsilon+U_\sigma \epsilon + U_m) \big]\Big)^2}{\Big(\E_{Q^{(n)}}\big[ p_{e}(\epsilon+U_\sigma \epsilon + U_m) \big]\Big)^2} \\
			&=\frac{\E_{Q^{(n)}}\big[ (1+U_\sigma)^2\big\{\rho_{e}'(\epsilon+U_\sigma \epsilon + U_m) + \rho_{e}^2(\epsilon+U_\sigma \epsilon + U_m)\big\} p_{e}(\epsilon+U_\sigma \epsilon + U_m) \big]}{\E_{Q^{(n)}}\big[ p_{e}(\epsilon+U_\sigma \epsilon + U_m) \big]}\\
			&\phantom{=} - \frac{\Big(\E_{Q^{(n)}}\big[ (1+U_\sigma)\rho_{e}(\epsilon+U_\sigma \epsilon + U_m) p_{e}(\epsilon+U_\sigma \epsilon + U_m) \big]\Big)^2}{\Big(\E_{Q^{(n)}}\big[ p_{e}(\epsilon+U_\sigma \epsilon + U_m) \big]\Big)^2} 
		\end{align*}
		In the third line we have made use of the identities \begin{align*}
			p'_\varepsilon(\epsilon) &= \rho_{e}(\epsilon)p_{e}(\epsilon);\\
			p_{e}''(\epsilon)&= \big\{\rho_{e}'(\epsilon)  +\rho_{e}^2(\epsilon)\big\} p_{e}(\epsilon).
		\end{align*}
		We now apply both the triangle and H\"older inequalities to deduce
		\begin{align*}
			\big| \rho_{\hat\varepsilon}'(\epsilon) \big|&\leq \sup_{|u_m|\leq C_m, |u_\sigma|\leq C_\sigma} (1+u_\sigma)^2\big\{ \big|\rho_{e}'(\epsilon+u_\sigma \epsilon + u_m)\big| + 2\rho_{e}^2 (\epsilon+u_\sigma \epsilon + u_m)\big\}\\
			&\leq (1+C_\sigma)^2\Bigg\{C_\rho + 2 \sup_{|u_m|\leq C_m, |u_\sigma|\leq C_\sigma} \rho_{e}^2 (\epsilon+u_\sigma \epsilon + u_m)\Bigg\}.
		\end{align*}
		
		Now we apply a Taylor expansion as follows.
		\begin{align*}
			\E_P\Big[\big\{\rho_{\hat\varepsilon}(\varepsilon_P)-\rho_{\hat\varepsilon}(\hat\varepsilon)\big\}^2 \; \Big| \; D^{(n)}\Big] &= \E_P\Big(\E_{Q^{(n)}}\Big[\big\{\rho_{\hat\varepsilon}(\varepsilon_P)-\rho_{\hat\varepsilon}(\varepsilon_P + U_\sigma\varepsilon_P + U_m)\big\}^2 \; \Big| \; \varepsilon_P, \; D^{(n)}\Big] \; \Big| \; D^{(n)} \Big)\\
			&\leq \E_P\Bigg(\E_{Q^{(n)}}\big[(U_\sigma\varepsilon_P+U_m)^2\; \big| \; \varepsilon_P\big] \\
			&\phantom{=\E_P\Bigg(}\cdot\bigg\{\sup_{|u_m|\leq C_m, |u_\sigma|\leq C_\sigma} \rho'_{\hat\varepsilon} (\varepsilon_P+u_\sigma \varepsilon_P + u_m)\bigg\}^2  \; \Bigg| \; D^{(n)} \Bigg)\\
			&\leq (1+C_\sigma)^2\E_P\Bigg(\E_{Q^{(n)}}\big[(U_\sigma\varepsilon_P+U_m)^2\; \big| \; \varepsilon_P \big] \\
			&\phantom{
				\leq \E_P\Bigg(
			}  \cdot\bigg\{C_\rho + 2\sup_{|\eta_m|\leq 2C_m+C_\sigma C_m, |\eta_\sigma|\leq 2C_\sigma+C_\sigma^2} \rho_{e}^2 (\varepsilon_P+\eta_\sigma \varepsilon_P + \eta_m)\bigg\}^2  \; \Bigg| \; D^{(n)} \Bigg)\\
			&\leq 2(1+C_\sigma)^2\E_P\Bigg[\Big\{\E_{Q^{(n)}}\big(U_\sigma^2\big)\varepsilon_P^2+\E_{Q^{(n)}}\big(U_m^2\big)\Big\} \\
			&\phantom{ \leq \E_P\Bigg( } \cdot \bigg\{C_\rho  + 2\sup_{|\eta_m|\leq 2C_m+C_\sigma C_m, |\eta_\sigma|\leq 2C_\sigma+C_\sigma^2}  \rho_{e}^2 (\varepsilon_P+\eta_\sigma \varepsilon_P + \eta_m)\bigg\}^2  \; \Bigg| \; D^{(n)} \Bigg]\\
			&\leq 2(1+C_\sigma)^2\bigg(\E_P\bigg[\Big\{\E_{Q^{(n)}}\big(U_\sigma^2\big)\varepsilon_P^2+\E_{Q^{(n)}}\big(U_m^2\big)\Big\}^2 \; \bigg|\; D^{(n)} \bigg]\bigg)^{\frac12}\\
			&\phantom{\leq}\cdot \Bigg(\E_P\Bigg[\bigg\{C_\rho + 2\sup_{|\eta_m|\leq 2C_m+C_\sigma C_m, |\eta_\sigma|\leq 2C_\sigma+C_\sigma^2} \rho_{e}^2 (\varepsilon_P+\eta_\sigma \varepsilon_P + \eta_m)\bigg\}^4 \Bigg]\Bigg)^{\frac12}\\
			&\leq 2(1+C_\sigma)^2\Big[\E_{Q^{(n)}}\big(U_\sigma^2\big)\big\{\E_P\big(\varepsilon_P^4\big)\big\}^{\frac12}+\E_{Q^{(n)}}\big(U_m^2\big)\Big]\\
			&\phantom{\leq}\cdot \Bigg(\E_P\Bigg[\bigg\{C_\rho + 2\sup_{|\eta_m|\leq 2C_m+C_\sigma C_m, |\eta_\sigma|\leq 2C_\sigma+C_\sigma^2}  \rho_{e}^2 (\varepsilon_P+\eta_\sigma \varepsilon_P + \eta_m)\bigg\}^4 \Bigg]\Bigg)^{\frac12}\\
			&\leq 2(1+C_\sigma)^2\Big[\E_{Q^{(n)}}\big(U_\sigma^2\big)\big\{\E_P\big(\varepsilon_P^4\big)\big\}^{\frac12}+\E_{Q^{(n)}}\big(U_m^2\big)\Big]\\
			&\phantom{\leq}\cdot \bigg(\E_P\bigg[\Big\{C_\rho + 2\big(|\rho_{e} (\varepsilon_P)|  + C_\rho(2 C_m+C_\sigma C_m) + C_\rho (2C_\sigma+C_\sigma^2)\varepsilon_P\big)^2\Big\}^4 \bigg]\bigg)^{\frac12}\\
			&\leq 2(1+C_\sigma)^2\Big[\E_{Q^{(n)}}\big(U_\sigma^2\big)\big\{\E_P\big(\varepsilon_P^4\big)\big\}^{\frac12}+\E_{Q^{(n)}}\big(U_m^2\big)\Big]\\
			&\phantom{\leq}\cdot \Big(\E_P\Big[\big\{C_\rho + 6\rho^2_\varepsilon (\varepsilon_P)  + 6C_\rho^2(2 C_m+C_\sigma C_m)^2 + 6C_\rho^2 (2C_\sigma+C_\sigma^2)^2\varepsilon_P^2\big\}^4 \Big]\Big)^{\frac12}\\
			&\leq 2(1+C_\sigma)^2\Big[\E_{Q^{(n)}}\big(U_\sigma^2\big)\big\{\E_P\big(\varepsilon_P^4\big)\big\}^{\frac12}+\E_{Q^{(n)}}\big(U_m^2\big)\Big]\\
			&\phantom{\leq}\cdot \bigg(C_\rho + 6\Big(\E_P\big[\rho^8_\varepsilon (\varepsilon_P)\big]\Big)^\frac14  \\
			&\phantom{\leq}+ 6C_\rho^2(2 C_m+C_\sigma C_m)^2 + 6C_\rho^2 (2C_\sigma+C_\sigma^2)^2\Big(\E_P\big[\varepsilon_P^8\big]\Big)^{\frac14}\bigg)^2
		\end{align*}
		where we have made use of the triangle inequalities for $L_2(P)$ and $L_4(P)$, and also the inequalities, $\{(a+b)/2\}^2 \leq (a^2+b^2)/2$ and $\{(a+b+c)/3\}^2 \leq (a^2+b^2+c^2)/3$.
		
		Finally, by Theorem~\ref{thm:scoresubg} and the assumed eighth moment of $\varepsilon_P$, we have that \begin{align*}
			6\Big(\E_P\big[\rho^8_\varepsilon (\varepsilon_P)\big]\Big)^\frac14 &\leq 6\cdot 105^\frac14C_\rho < 20 C_\rho;\\
			\big\{\E_P\big(\varepsilon_P^4\big)\big\}^{\frac12} & \leq \{\E_P\big(\varepsilon_P^8\big)\big\}^{\frac14} = C_\varepsilon^2.
		\end{align*}
		Hence \begin{align*}
			\E_P\Big[\big\{\rho_{\hat\varepsilon}(\varepsilon_P)-\rho_{\hat\varepsilon}(\hat\varepsilon)\big\}^2 \; \Big| \; D^{(n)}\Big] &\leq 2(1+C_\sigma)^2 C_\rho^2\big\{21  + C_\rho(2+C_\sigma)^2(C_m^2 +C_\sigma^2C_\varepsilon^2)\big\}^2 \\
			&\phantom{\leq} \cdot \Big\{4C_\varepsilon^2 \E_{Q^{(n)}}\big(U_\sigma^2\big)+\E_{Q^{(n)}}\big(U_m^2\big)\Big\}.
		\end{align*}
	\end{proof}

	\begin{lemma}
		\label{lem:estimatedresidualscore_locationonly}
		Let $P$ be such that $p_{e}$ is twice differentiable on $\R$, with \begin{equation*}
			\sup_{\epsilon\in\R} |\partial_\epsilon^2 \log p_{e}(\epsilon)| = \sup_{\epsilon\in\R} |\rho_{e}'(\epsilon)|\leq C_\rho
		\end{equation*}
		and $p_{e}'$ bounded. Further assume that $D^{(n)}$ is such that  $\sup_{z\in\mathcal{Z}} |u_m^{(n)}(z)|\leq C_m$. Then there exists a constant $C$, depending only on $C_\rho, C_m$, such that
		\begin{equation*}
			\E_P\Big[\big\{\rho_{e}(\varepsilon_P)-\rho_{\hat\varepsilon}(\varepsilon_P)\big\}^2 \; \Big| \; D^{(n)}\Big] \leq C A_m^{(n)} .
		\end{equation*}
	\end{lemma}
	\begin{proof}
		For ease of notation, write $Q^{(n)}$ for the distribution of $u_m^{(n)}(Z)$ conditionally on $D^{(n)}$, and let $U\sim Q^{(n)}$. Therefore \begin{equation*}
			A_m^{(n)} = \E_{Q^{(n)}} \big(U^2\big).
		\end{equation*}

		The condition on $p_{e}$ is sufficient to interchange differentiation and expectation operators as follows \citep[Thm.~20.4]{AliprantisAnalysis}. 
		\begin{align*}
			\rho_{\hat\varepsilon}(\epsilon)  &= \frac{\frac{\partial}{\partial \epsilon}\E_{Q^{(n)}}[p_{e}(\epsilon + U) ]}{\E_{Q^{(n)}}[p_{e}(\epsilon + U)]}\\
			&=  \frac{\E_{Q^{(n)}}\big[p_{e}'(\epsilon + U) \big]}{\E_{Q^{(n)}}[p_{e}(\epsilon + U) ]}.
		\end{align*}
		We may decompose the approximation error as follows.
		\begin{align*}
			|\rho_{\hat\varepsilon}(\epsilon) - \rho_{e}(\epsilon)| &= \bigg|\frac{\E_{Q^{(n)}} \big[p_{e}'(\epsilon+U)\big] }{\E_{Q^{(n)}} [p_{e}(\epsilon+U)]} - \rho_{e}(\epsilon) \bigg| \\
			&=\bigg|\frac{\E_{Q^{(n)}} [\{\rho_{e}(\epsilon+U)-\rho_{e}(\epsilon)\}\;p_{e}(\epsilon+U)] }{\E_{Q^{(n)}} [p_{e}(\epsilon+U)]}\bigg| \\
			&\leq C_\rho \frac{\E_{Q^{(n)}} [|U|p_{e}(\epsilon+U)] }{\E_{Q^{(n)}} [p_{e}(\epsilon+U)]}\\
			&\leq C_\rho \Big(\E_{Q^{(n)}} \big(U^2\big)\Big)^{1/2} \frac{\Big( \E_{Q^{(n)}}\big[p_{e}^2(\epsilon+U)\big]\Big)^{1/2}}{\E_{Q^{(n)}} [p_{e}(\epsilon+U)]}\\
			&=: C_\rho \big(A_m^{(n)}\big)^{1/2} R(\epsilon).
		\end{align*}
		The first inequality uses the Lipschitz property of $\rho_{e}$. The second applies the Cauchy--Schwarz inequality.
		
		Now, for every $\epsilon\in\R$ with $p_{e}(\epsilon)>0$,
		\begin{align*}
			R^2(\epsilon) &\leq \bigg(\frac{\sup_{|u|\leq C_m} p_{e}(\epsilon + u) / p_{e}(\epsilon)}{\inf_{|u|\leq C_m} p_{e}(\epsilon + u) / p_{e}(\epsilon)}\bigg)^2\\
			&\leq \exp\big\{4C_m|\rho_{e}(\epsilon)| + 2C_\rho C_m^2\big\}.
		\end{align*}
		The first line is a supremum bound for the ratio of expectations, the second is the application of Lemma~\ref{lem:densityratiobound}. Since $\exp(|x|)\leq \exp(x)+\exp(-x)$, this yields the bound
		\begin{equation*}
			\E_P\Big[\big\{\rho_{e}(\varepsilon_P)-\rho_{\hat\varepsilon}(\varepsilon_P)\big\}^2 \; \Big| \; D^{(n)}\Big] \leq C_\rho^2 \exp\big(2C_\rho C_m^2\big) \;\big(\E_P[\exp\{4C_m \rho_{e}(\varepsilon_P)\}] + \E_P[\exp\{-4C_m \rho_{e}(\varepsilon_P)\}]\big) \;A_m^{(n)}.
		\end{equation*}
		By Theorem \ref{thm:scoresubg}, $\rho_{e}(\varepsilon_P)$ is sub-Gaussian with parameter $\sqrt{2C_\rho}$, so for all $\lambda\in\R$ we have \begin{equation*}
			\E_P[\exp\{\lambda \rho_{e}(\varepsilon_P)\}] \leq \exp(\lambda^2 C_\rho).
		\end{equation*}
		Thus
		\begin{equation*}
			\E_P\Big[\big\{\rho_{e}(\varepsilon_P)-\rho_{\hat\varepsilon}(\varepsilon_P)\big\}^2 \; \Big| \; D^{(n)}\Big] \leq 2C_\rho^2 \exp(18C_\rho C_m^2)  \;A_m^{(n)}.
		\end{equation*}
	\end{proof}

	\begin{lemma}
		\label{lem:estimatedresidual_locationonly}
		Let $P$ be such that $p_{e}$ is twice differentiable on $\R$, with \begin{equation*}
			\sup_{\epsilon\in\R} |\partial_\epsilon^2 \log p_{e}(\epsilon)| = \sup_{\epsilon\in\R} |\rho_{e}'(\epsilon)|\leq C_\rho,
		\end{equation*}
		and $p_{e}'$ and $ p_{e}''$ both bounded. Further assume that $D^{(n)}$ is such that $\sup_{z\in\mathcal{Z}} |u_m^{(n)}(z)|\leq C_m$ for almost every $z\in\mathcal{Z}$. 
		Then there exists a constant $C$, depending only on $C_\rho, C_m$, such that
		\begin{equation*}
			\E_P\Big[\big\{\rho_{\hat\varepsilon}(\varepsilon_P)-\rho_{\hat\varepsilon}(\hat\varepsilon^{(n)})\big\}^2 \; \Big| \; D^{(n)}\Big] \leq C \big(A_m^{(n)} + A_\sigma^{(n)}\big).
		\end{equation*}
	\end{lemma}
	\begin{proof}
		For ease of notation, write $Q^{(n)}$ for the distribution of $u_m^{(n)}(Z)$ conditionally on $D^{(n)}$, and let $U\sim Q^{(n)}$. Therefore \begin{equation*}
			A_m^{(n)} = \E_{Q^{(n)}} \big(U^2\big).
		\end{equation*}
		
		The proof proceeds by first bounding the derivative of $\rho_{\hat\varepsilon}$. The conditions on $p_{e}$ are sufficient to interchange differentiation and expectation operators as follows \citep[Thm.~20.4]{AliprantisAnalysis}:
		\begin{align*}
			\rho_{\hat\varepsilon}(\epsilon)  &= \frac{\frac{\partial}{\partial \epsilon}\E_{Q^{(n)}}[p_{e}(\epsilon + U)]}{\E_{Q^{(n)}}[p_{e}(\epsilon + U)]}\\
			&=  \frac{\E_{Q^{(n)}}\big[p_{e}'(\epsilon + U) \big]}{\E_{Q^{(n)}}[p_{e}(\epsilon + U) \big]},
		\end{align*}
		and further,
		\begin{align*}
			\rho_{\hat\varepsilon}'(\epsilon)  &=  \frac{\partial}{\partial \epsilon} \frac{\E_{Q^{(n)}}\big[p_{e}'(\epsilon + U) \big]}{\E_{Q^{(n)}}[p_{e}(\epsilon + U) \big]}\\
			&=  \frac{\frac{\partial}{\partial \epsilon} \E_{Q^{(n)}}\big[p_{e}'(\epsilon + U) \big]}{\E_{Q^{(n)}}[p_{e}(\epsilon + U) ]} - \rho_{\hat\varepsilon}^2(\epsilon)\\
			&= \frac{\E_{Q^{(n)}}\big[ p_{e}''(\epsilon + U) \big]}{\E_{Q^{(n)}}[ p_{e}(\epsilon + U)]} - \frac{\Big(\E_{Q^{(n)}}\big[ p_{e}'(\epsilon + U) \big]\Big)^2}{\Big(\E_{Q^{(n)}}\big[ p_{e}(\epsilon+U_\sigma \epsilon + U_m) \big]\Big)^2} \\
			&=\frac{\E_{Q^{(n)}}\big[ \big\{\rho_{e}'(\epsilon + U) + \rho_{e}^2(\epsilon + U)\big\} p_{e}(\epsilon + U) \big]}{\E_{Q^{(n)}}\big[ p_{e}(\epsilon + U) \big]} - \frac{\big(\E_{Q^{(n)}}[ \rho_{e}(\epsilon + U) p_{e}(\epsilon + U)]\big)^2}{\big(\E_{Q^{(n)}}[ p_{e}(\epsilon + U) ]\big)^2}.
		\end{align*}
		In the third line we have made use of the identities \begin{align*}
			p'_\varepsilon(\epsilon) &= \rho_{e}(\epsilon)p_{e}(\epsilon);\\
			p_{e}''(\epsilon)&= \big\{\rho_{e}'(\epsilon)  +\rho_{e}^2(\epsilon)\big\} p_{e}(\epsilon).
		\end{align*}
		We now apply both the triangle and H\"older inequalities to deduce
		\begin{align*}
			\big| \rho_{\hat\varepsilon}'(\epsilon) \big|&\leq \sup_{|u|\leq C_m} \big\{ \big|\rho_{e}'(\epsilon + u)\big| + 2\rho_{e}^2 (\epsilon + u)\big\}\\
			&\leq C_\rho + 2 \sup_{|u|\leq C_m} \rho_{e}^2 (\epsilon + u).
		\end{align*}
		
		Now we apply a Taylor expansion as follows, noting that $\varepsilon_P$ is independent of $U$ conditionally on $D^{(n)}$.
		\begin{align*}
			\E_P\Big[\big\{\rho_{\hat\varepsilon}(\varepsilon_P)-\rho_{\hat\varepsilon}(\hat\varepsilon^{(n)})\big\}^2 \; \Big| \; D^{(n)}\Big] &= \E_P\big(\E_{Q^{(n)}}\big[\{\rho_{\hat\varepsilon}(\varepsilon_P)-\rho_{\hat\varepsilon}(\varepsilon_P + U)\}^2 \; \big| \; \varepsilon_P, \; D^{(n)}\big] \; \big| \; D^{(n)} \big)\\
			&\leq \E_P\Bigg(\E_{Q^{(n)}}(U^2) \bigg\{\sup_{|u|\leq C_m} \rho'_{\hat\varepsilon} (\varepsilon_P + u)\bigg\}^2  \; \Bigg| \; D^{(n)} \Bigg)\\
			&= A_m^{(n)} \E_P\Bigg( \bigg\{\sup_{|u|\leq C_m} \rho'_{\hat\varepsilon} (\varepsilon_P + u)\bigg\}^2  \; \Bigg| \; D^{(n)} \Bigg)\\
			&\leq A_m^{(n)} \E_P\Bigg(\bigg\{C_\rho + 2\sup_{|\eta|\leq 2C_m} \rho_{e}^2 (\varepsilon_P + \eta)\bigg\}^2  \; \Bigg| \; D^{(n)} \Bigg)\\
			&\leq A_m^{(n)} \E_P\bigg[\Big\{C_\rho + 2\big(|\rho_{e} (\varepsilon_P)|  + 2C_\rho C_m \big)^2\Big\}^2 \bigg]\\
			&\leq  A_m^{(n)} \; \E_P\Big[\big\{C_\rho + 4\rho^2_\varepsilon(\varepsilon_P) + 16C_\rho^2C_m^2\big\}^2 \Big]\\
			&\leq  A_m^{(n)} \; \Big(3C_\rho^2 + 48 \E_P\big[\rho^4_\varepsilon(\varepsilon_P)\big] + 768 C_\rho^4C_m^4\Big),
		\end{align*}
		where we have made use of the inequalities $(a+b)^2 \leq 2(a^2+b^2)$ and $(a+b+c)^2 \leq 3(a^2+b^2+c^2)$. Finally, by Theorem~\ref{thm:scoresubg}, $\E_P\big[\rho^4_\varepsilon (\varepsilon_P)\big] \leq 4 C_\rho^2$. Hence 
		\begin{equation*}
			\E_P\Big[\big\{\rho_{\hat\varepsilon}(\varepsilon_P)-\rho_{\hat\varepsilon}(\hat\varepsilon^{(n)})\big\}^2 \; \Big| \; D^{(n)}\Big] \leq   A_m^{(n)} \; \big(147 C_\rho^2  + 768 C_\rho^4C_m^4\big). \qedhere
		\end{equation*}
	\end{proof}

	\section{Auxiliary lemmas} \label{app:aux_Lemmas}
	\begin{lemma}
		\label{lem:pbounded}
		If $p$ is a twice differentiable density function on $\R$ with score $\rho$ defined everywhere and $\sup_{x\in\R} |\rho'(x)| \leq C$, then $\sup_{x\in\R} p(x) \leq 2\sqrt{2C}$. 
	\end{lemma}
	\begin{proof}
		Suppose, for a contradiction, that $\sup_{x\in\R} p(x) > 2\sqrt{2C}$. Pick $x_0 < x_1'$ such that $0<p(x_0) < \sqrt{2C}$ and $p(x_1')>2\sqrt{2C}$, and further that $x_1'$ is not the maximiser of $p$ on the interval $[x_0,x_1']$. We set $x_1$ to be the maximiser of $p$ in $(x_0, x_1')$, and observe that $p(x_1)=:M > 2\sqrt{2C}$, $p'(x_1) = 0$, and  $p''(x_1) < 0$.
		
		Now let $x_- = \sup\{x<x_1: p(x)\leq M/2\} > x_0$, the final inequality following from the intermediate value theorem. Note that as
		\[
		1 \geq \int_{x_-}^{x_1} p(x) \, dx \geq \frac{M}{2} (x_1 - x_-),
		\]
		$x_1 - x_- \leq 2/M$.
		We also have that $p(x_1)-p(x_-)\geq M/2$, so there must be a point $\tilde x \in [x_-, x_1]$ where $p'(\tilde x) \geq M^2/4$. Now because $p'(x_1)=0$ there must also be a point $x_* \in [\tilde x, x_1]$ with $p''(x_*)\leq -M^3/8$.
		
		Finally we may employ the assumption on $|\rho'|$ to bound $p''(x_*)$ from below. Noting that $p(x_*)\leq M$ as $x_* \in [x_0, x_1]$, we have
		\begin{align*}
			-M^3/8 \geq p''(x_*) &= \rho'(x_*)p(x_*) + \rho^2(x_*)p(x_*)\\
			&\geq \rho'(x_*)p(x_*) \geq -CM. \qedhere
		\end{align*}
	\end{proof}
	
	\begin{corollary}
		\label{lem:p''bounded}
		If $p$ is a twice differentiable density function on $\R$ with score $\rho$ defined everywhere and $\sup_{x\in\R} |\rho'(x)| \leq C$ then $\inf_{x\in\R} p''(x) \geq -2\sqrt{2}C^{3/2}$. 
	\end{corollary}
	\begin{proof}
		This follows from Lemma~\ref{lem:pbounded} and       $p''(x) = \rho'(x)p(x) + \rho^2(x)p(x)\geq \rho'(x)p(x)$.
	\end{proof}
	
	
	\begin{lemma}
		\label{lem:pconvergent}
		If $p$ is a twice differentiable density function on $\R$ and $\sup_{x\in\R} |\rho'(x)| \leq C$, then $p(x) \to 0$ as $|x|\to\infty$.
	\end{lemma}
	\begin{proof}
		Note first that by Lemma~\ref{lem:pbounded} we know that $p(x)$ is uniformly bounded.
		Suppose then, for contradiction, that $\limsup_{|x|\to\infty} p(x)=:2\epsilon>0$. Then for any $M\geq 0$ we can find $x_0$ with $|x_0|>M$ and $p(x_0) \geq \epsilon$. We will show that the integral of $p(x)$ over a finite interval containing $x_0$ is bounded below. This means that we can choose non-overlapping intervals $I_1,\ldots, I_N$ such that \begin{equation*}
			\int_{\R}p(x)\;dx \geq \sum_{n=1}^N \int_{I_n} p(x)\;dx >1,
		\end{equation*}
		a contradiction.

		Since $p'$ is continuous, we have that $|p'(x_0)| < \infty$. By Corollary~\ref{lem:p''bounded}, $\inf_{x\in\R} p''(x) \geq  -2\sqrt{2}C^{3/2}$. Using a Taylor expansion, we can fit a negative quadratic beneath the curve $p$ at $x_0$. Integrating this quadratic over the region where it is positive gives the bound. Indeed,  
		\begin{align*}
			p(x) &\geq p(x_0) + (x-x_0)p'(x_0) - \sqrt{2} C^{3/2}(x-x_0)^2\\
			&= p(x_0) + \frac{\big(p'(x_0)\big)^2}{4\sqrt2 C^{3/2}} - \sqrt{2} C^{3/2}\bigg(x-x_0-\frac{p'(x_0)}{2\sqrt{2} C^{3/2}}\bigg)^2\\
			&\geq \epsilon - \sqrt{2} C^{3/2}\bigg(x-x_0-\frac{p'(x_0)}{2\sqrt{2} C^{3/2}}\bigg)^2 =: f(x).
		\end{align*}
		The quadratic $f(x)$ has roots
		\begin{align*}
			a &:= x_0 + \frac{p'(x_0)}{2\sqrt{2} C^{3/2}} - \frac{\sqrt{\epsilon}}{2^{1/4}C^{3/4}};\\
			b &:= x_0 + \frac{p'(x_0)}{2\sqrt{2} C^{3/2}} + \frac{\sqrt{\epsilon}}{2^{1/4}C^{3/4}}.
		\end{align*}
		Thus $(a,b)$ is a finite interval containing $x_0$ and
		\begin{equation*}
			\int_a^b p(x)\;dx \geq \int_a^b f(x)\,dx = \frac{2^{7/4} \epsilon^{3/2}}{3C^{3/4}}. \qedhere
		\end{equation*}
	\end{proof}
	
	\begin{lemma}
		\label{lem:boundaryconvergent}
		Let $f:\R\to\R$ be a continuous function. Then at least one of the following holds.
		\begin{enumerate}[label=(\alph*)]
			\item There exists a sequence $a_n \to \infty$ such that $f(a_n) \to 0$.
			\item There exists $A\in\R$ and $\epsilon > 0$ such that $f(x) >\epsilon$ for all $x\geq A$, and in particular $\int_A^\infty f(x)\;dx = \infty$.
			\item There exists $A\in\R$ and $\epsilon > 0$ such that $f(x) <-\epsilon $ for all $x\geq A$, and in particular $\int_A^\infty (-f(x))\;dx = \infty$.
		\end{enumerate}
	\end{lemma}
	\begin{proof}
		If $\liminf_{x \to \infty} f(x) > 0$ then clearly (b) occurs while if $\limsup_{x \to \infty} f(x) < 0$ then (c) occurs. Thus we may assume that $\liminf_{x \to \infty} f(x) \leq 0 \leq \limsup_{x \to \infty} f(x)$. If either of these inequalities are equalities, then (a) occurs, so we may assume they are both strict. However in this case, as $\{f(x) : x \geq A\}$ has infinitely many positive points and negative points for all $A \geq 0$, by the intermediate value theorem, we must have that (a) occurs.
	\end{proof}

	\begin{lemma}
		\label{lem:scoreboundaryconvergent}
		Let $p$ be a twice differentiable density function on $\R$ with score $\rho$ defined everywhere, and let $k$ be a non-negative integer. If $\sup_{x\in\R} |\rho'(x)| \leq C$ and $\E[\rho^{2k}(X)]<\infty$, then there exist sequences $a_n \to -\infty$ and $b_n \to \infty$ such that $\rho^{2k+1}(a_n)p(a_n) \to 0$ and $\rho^{2k+1}(b_n)p(b_n) \to 0$.
	\end{lemma}
	\begin{proof}
		Write $f(x)=\rho^{2k+1}(x)p(x)$. Since $f$ is continuous, we may apply Lemma~\ref{lem:boundaryconvergent} to both $f$ and $x \mapsto f(-x)$ to conclude that either the statement of the lemma holds, or one of the following hold for some $B \in \R$ and $\epsilon > 0$:
		\begin{enumerate}[label=(\alph*)]
			\item $f(x) > \epsilon$ for all $x \geq B$,
			\item $f(x) < - \epsilon$ for all $x \geq B$
		\end{enumerate}
		or one of the above with $x \geq B$ replaced with $x \leq B$. Let us suppose for a contradiction that (a) occurs (the other cases are similar), so in particular
		\begin{equation}
			\int_{B}^\infty f(x)\; dx = \infty. \label{eqn:rho2k+1contradictioninfinite}
		\end{equation}
		
		
		If $k=0$, then \begin{equation*}
			\int_B^\infty f(x)\;dx = \int_B^\infty p'(x)\;dx = \lim_{b \to \infty} \int_B^b p'(x)\;dx = \lim_{b\to\infty} p(b) - p(B);
		\end{equation*}
		here the penultimate equality follows from monotone convergence and the final equality follows from the fundamental theorem of calculus.
		By Lemma~\ref{lem:pconvergent} however, this is finite, a contradiction. If instead $k\geq 1$, then for any $b\geq B$ we have that \begin{align}
			\rho^{2k}(b)p(b)-\rho^{2k}(B)p(B) &= \int_B^b \rho^{2k}(x)p'(x)\; dx + 2k\int_B^b \rho'(x)\rho^{2k-1}(x)p(x)\;dx.\nonumber\\
			&= \int_B^b \rho^{2k+1}(x)p(x)\; dx + 2k\int_B^b \rho'(x)\rho^{2k-1}(x)p(x)\;dx.\label{eqn:rho2k+1contradictionfinite}
		\end{align}
		We will take the limit as $b\to\infty$. Since $\rho^{2k}(x)p(x)$ is non-negative and we have that $\E[\rho^{2k}(X)]<\infty$, we can choose an increasing sequence $b_n\to\infty$ satisfying $\rho^{2k}(b_n)p(b_n) \leq 1$ for every $n$. 
		
		Note that for each $n$ and for every $x\in\R$,
		\begin{equation*}
			\big|\ind_{[B, b_n]}(x)\rho'(x) \rho^{2k-1}(x)p(x)\big|
			\leq C |\rho(x)|^{2k-1}p(x).
		\end{equation*}
		By Jensen's inequality, $\E\big[|\rho(X)|^{2k-1}\big]< \infty$. Thus, by dominated convergence theorem,
		\begin{align*}
			\lim_{n \to \infty} 2k \int_B^{b_n} \rho'(x) \rho^{2k-1}(x) p(x)\;dx &= 2k \int_B^{\infty} \rho'(x) \rho^{2k-1}(x) p(x)\;dx\\
			&\leq 2kC \int_B^{\infty} |\rho(x)|^{2k-1} p(x)\;dx\\
			&\leq 2kC \;\E\big[|\rho(X)|^{2k-1}\big]< \infty.
		\end{align*}
		Now \eqref{eqn:rho2k+1contradictionfinite} implies that \begin{equation*}
			\lim_{n\to\infty} \int_B^{b_n}  f(x)\; dx < \infty.
		\end{equation*}
		But we assumed that $f(x) \geq \epsilon > 0$ for all $x\geq B$, so for each fixed $x\in\R$ the integrand $\ind_{[B,b_n]}(x) f(x) $ is increasing as a function of $n$. Therefore monotone convergence implies that \begin{equation*}
			\int_B^{\infty}  f(x)\; dx < \infty,
		\end{equation*}
		contradicting \eqref{eqn:rho2k+1contradictioninfinite}.
	\end{proof}


	\begin{lemma}
	\label{lem:subgmoment}
	Let $X$ be mean-zero and sub-Gaussian with parameter $\sigma>0$. Then for any $p>0$,\begin{align*}
		\E\big(|X|^p\big)\leq p2^{\frac{p}2}\sigma^p \Gamma\Big(\frac{p}2\Big),
	\end{align*}
	where $\Gamma(x) = \int_0^\infty u^{x-1} \exp(-u)\;du$ is the gamma function.
\end{lemma}
\begin{proof}
	By the Chernoff bound we have that \begin{equation*}
		\Pr(|X|>t)\leq 2\exp\bigg(-\frac{t^2}{2\sigma^2}\bigg).
	\end{equation*}
	We are now able to make use of the tail probability formula for expectation.
	\begin{align*}
		\E\big(|X|^p\big) &= \int_0^\infty \Pr\big(|X|^p>s)\;ds\\
		&= \int_0^\infty \Pr\big(|X|>s^{-p})\;ds\\
		&= \int_0^\infty p t^{p-1} \Pr\big(|X|>t)\;dt\\
		&\leq \int_0^\infty p t^{p-1} 2 \exp\bigg(-\frac{t^2}{2\sigma^2}\bigg)\;dt\\
		&= \int_0^\infty \sigma^2 p (2\sigma^2 u)^{\frac{p}2-1} 2\exp(-u)\;du\\
		&= p2^{\frac{p}2}\sigma^p \int_0^\infty u^{\frac{p}2-1} \exp(-u)\;du.
	\end{align*}
	The third line makes the substitution $t=s^{-p}$, the fifth $u=t^2/2\sigma^2$. Recalling the definition of the Gamma function, we are done.
\end{proof}

\begin{lemma}[\cite{WainwrightHighDim} Thm.~2.6]
	\label{lem:subgsquaremgf}
	Let $X$ be mean-zero and sub-Gaussian with parameter $\sigma>0$. Then \begin{align*}
		\E\bigg[\exp\bigg(\frac{\lambda X^2}{2\sigma^2}\bigg)\bigg]\leq \frac1{\sqrt{1-\lambda}} \text{ for all $\lambda\in[0,1)$.}
	\end{align*}
\end{lemma}

	\section{Additional points}
	\subsection{On the semiparametric efficient variance bound and sub-Gaussian scores} \label{sec:var_bound}
	Considering the setting of Section~\ref{sect:general} in the case where $d=1$, for simplicity (and dropping the subscript $P$ for notational ease), if $\rho(X, Z)$ is a sub-Gaussian random variable, then for any $\eta > 0$, the semiparametric efficient variance bound is satisfies
	\begin{align*}
		\E\left[\left(f'(X, Z) - \theta - \rho(X, Z)\{Y - f(X, Z)\}\right)^2 \right] &= \Var[f'(X, Z)^2] + \E \left[ \rho(X, Z)^2(Y - f(X, Z))^2\right] \\
		&\leq \Var[f'(X, Z)^2] + C \E|Y - f(X, Z)|^{2+\eta}
	\end{align*}
	for some $C > 0$. Here we have used H\"older's inequality and appealed to the fact that sub-Gaussian random variables have moments of all orders. Recall from Theorem~\ref{thm:scoresubg}, $\rho(X, Z)$ will be sub-Gaussian under a uniform bound on $\rho'(x, z)$.
	
	When the tails of the $p(x|z)$ density are lighter than that of a Gaussian, it may be the case that the efficiency bound above is finite, but no uniform bound on $\rho'(x, z)$ will exist. Consider, for example, the case where $p(x|z) \propto \exp(-x^{4})$. Then $\rho(x, z) = -4x^3$ and $\rho'(x, z) = -12 x^2$. On the other hand, it is clear that $\rho(X, Z)$ will have moments of all orders, so a version of the bound above will hold.
	
	Returning to the previous display, in the case that the errors $Y - f(X, Z)$ are independent of $(X, Z)$, then the only assumption on the score $\rho(X, Z)$ in order to obtain a semiparametric efficient variance bound, is $\E[\rho(X, Z)^2] < \infty$. In this case, if $p(x|z)$ has support $(a, b)$ and $p(x|z) \propto |x-c|^m \ind_{(a, b)}(x)$ for all $x$ sufficiently close to $c$, $c \in \{a,b\}$ and $m > 1$, then we will have $\E[\rho(X, Z)^2] < \infty$. Indeed, then we will have $(p'(x|z))^2 / p'(x|z) \propto |x-c|^{m-2}\ind_{(a, b)}(x)$ which has a finite integral in an interval around $c$.

	\subsection{Linear score functions}
	Some works have made the simplifying assumption that
	\begin{equation}
				\rho_P(x,z) = \beta_P ^{\top} b(x,z) \label{eqn:linearscore}
			\end{equation} for some known basis $b(x,z)$. This has some theoretical appeal, since any $\rho_P$ can be represented in this way for some bases, and the score estimation problem is made parametric. Practically, however, even with domain knowledge it can be hard to choose a good basis. When $(X,Z)$ are of moderate to large dimension, there are limited interactions that one can practically allow --- for instance a quadratic basis may be feasible, but a multivariate kernel basis not. If the chosen basis contains the vector $x$, then it transpires that the linearity assumption \eqref{eqn:linearscore} is equivalent to assuming a certain conditional Gaussian linear model for $x$ given the other basis elements (see Theorem~\ref{thm:basisgaussian} below). This provides additional insight into the method of \cite{RothenhauslerICE}, which is based on the debiased Lasso \citep{ZhangDebiased, vandeGeerDebiased}.
	\begin{theorem}
		\label{thm:basisgaussian}
		Let $b(x,z)=\big(x,g^{\top}(x,z)\big)^{\top} \in \R^m$ for some $g:\R \times \mathcal{Z}\to \R^{m-1}$ be such that $\E\big[b(X,Z)b(X,Z)^{\top}\big]$ is positive definite, $\E\big[\|b(X,Z)\|^2_2\big]<\infty$, $\E|\partial_x b(X,Z)|<\infty$, and for almost every $z\in\mathcal{Z}$ we have that $b(\cdot,z)$ and $\partial_x b(\cdot,z)$ are absolutely continuous and $\lim_{|x|\to\infty}b(x,z)p(x\given  z)=0$. Define the linearly transformed basis functions
		\begin{equation*}
			\bar g(x,z) := g(x,z) - x \big(\E\big[ \partial_x g(X,Z)\big]\big) \in \R^{m-1}.
		\end{equation*}
		We have that $\rho(x,z)=\beta^{\top} b(x,z)$ for some $\beta \in \R^{m}$ if and only if $\rho(x,z)=\tilde \rho(x,\bar g(x,z))$, where $\tilde \rho$ is the score function corresponding to the related multivariate Gaussian linear model:
		\begin{equation*}
			(X,g) \eqdist \Big(X, \bar g(X,Z) \Big);\quad X\given  g \sim N\big(\gamma^{\top} g, S).
		\end{equation*}
		Here $\gamma\in \R^{(m-1)}$ and $S>0$ do not depend on $(X,g)$. 
	\end{theorem}
	\begin{proof}
		First assume that $X\given  g$ has the stated conditional distribution. Then
		\begin{align*}
			\tilde \rho(x,g) &= \partial_x \log \tilde p(x\given  g)\\
			&= -S^{-1}(x - \gamma^{\top}g)\\
			&= \begin{pmatrix}
				-S^{-1} & S^{-1} \gamma
			\end{pmatrix}^{\top} \begin{pmatrix}
				x \\ g
			\end{pmatrix},
		\end{align*}
		so indeed $\tilde \rho(x,\bar g(x,z))$ is in the linear span of $\{x,\bar g(x,z)\}$, and hence that of $b(x,z)$.
		
		Now let $\rho(x,z)=\beta^{\top} b(x,z)$, and denote by $\beta_x$ and $\beta_g$ the first and last $m-1$ components of $\beta$ respectively. Define the transformed variables
		\begin{align}
			\bar b(x,z) &:= \begin{pmatrix}
				x\\\bar g(x,z)
			\end{pmatrix} = \begin{pmatrix}
				1 & 0_{1 \times (m-1)}\\
				- \E\big[ \partial_x g(X,Z)\big] & I_{(m-1)\times (m-1)}
			\end{pmatrix}b(x,z)\label{eqn:bbardefinition};\\
			\bar \beta_x &:= \beta_x +  \beta_g^{\top} \E\big[\partial_x g(X,Z)\big];\nonumber\\
			\bar \beta &:= \begin{pmatrix}
				\bar \beta_x \\ \beta_g
			\end{pmatrix}.\nonumber
		\end{align}
		By the decomposition \eqref{eqn:bbardefinition} we see that $\E\big[\bar b(X,Z) \bar b^{\top}(X,Z)\big]$ inherits the positive definiteness of $\E\big[ b(X,Z) b^{\top}(X,Z)\big]$. Then we have that
		\begin{equation*}
			\rho(x,z) = \bar\beta^{\top} \bar b(x,z);\quad 
			\E\big[ \partial_x \bar b^{\top}(X,Z)\big]=\begin{pmatrix}
				1 & 0_{(m-1)\times 1}
			\end{pmatrix}.\label{eqn:scorebasiscentering}
		\end{equation*}

		The conditions on $b$ mean that $\rho$ satisfies the conditions of \citet[Prop.~1]{CoxSpline} conditionally on $Z$, so $\bar \beta$ minimises
		\begin{equation*}
			\E\big[(\bar \beta^{\top} \bar b(X,Z))^2+2\partial_x \bar\beta^{\top} \bar b(X,Z)\big]= \bar \beta^{\top} \E \big[\bar b(X,Z)\bar b^{\top}(X,Z)\big] \bar \beta + 2\bar \beta^{\top} \Big(\E\big[\partial_x \bar b^{\top}(X,Z)\big]\Big)^{\top}.
		\end{equation*}
		Hence \begin{equation*}
			\E \big[\bar b(X,Z)\bar b^{\top}(X,Z)\big] \bar\beta + \begin{pmatrix}
				1 \\ 0_{(m-1)\times 1}
			\end{pmatrix} = 0 \label{eqn:betadefiningequation}.
		\end{equation*}
		Using the Schur complement identity for the inverse, we find that $\bar\beta$ takes the following form:
		\begin{align*}
			\bar \beta &= -\begin{pmatrix}
				1 \\ \gamma  
			\end{pmatrix}S^{-1},\\
			\gamma &= \Big(\E\big[\bar g(X,Z)\bar g^{\top}(X,Z)\big]\Big)^{-1} \E\big[\bar g(X,Z)X^{\top}\big],\\
			S &= \E\big[\big\{X-\gamma^{\top} \bar g(X,Z)\big\}X^{\top}\big].
		\end{align*}
		Therefore we have that
		\begin{equation*}
			\rho(x,z) = - S^{-1} \{x - \gamma^{\top} \bar g(x,z)\}.
		\end{equation*}
		
		Finally, note that $\gamma\in \R^{m-1}$ satisfies
		\begin{equation*}
			\E\big[\big\{X-\gamma^{\top}\bar g(X,Z)\big\}\bar g^{\top}(X,Z)\big] = 0.
		\end{equation*}
		This implies that $\gamma$ minimises $\E\big[\{X-\gamma^{\top}\bar g(X,Z)\}^2\big] = S$.
		This suffices to prove that
		\begin{equation*}
			\rho(x,z)= - S^{-1} \{x - \gamma^{\top} \bar g(x,z)\} = \tilde \rho(x,\bar g(x,z)). \qedhere
		\end{equation*} 
	\end{proof}

	\subsection{Explicit estimators for numerical experiments}
	\label{sect:simulationimplementation}
	
	In order reduce the computational burden, we pre-tune all hyperparameters on 1000 datasets, each of which we split into training and testing. This includes all gradient boosting regression parameters, the various spline degrees of freedom and the Lasso tuning parameters of the basis approaches.
	
	\subsubsection{Resmooth and spline}
	\label{sect:methodus}
	Let $\tilde f^{(n,k)}$ and $\hat m^{(n,k)}$ be gradient boosting regressions (\texttt{xgboost} package \citep{xgboost}) of $Y$ on $(X,Z)$ and $X$ on $Z$ respectively, using the out-of-fold data $D^{(n,k)}$. Further let $\hat\sigma^{(n,k)}$ be the a decision tree (\texttt{partykit} package \citep{partykit}) regression of the squared in-sample residuals of $X$ on $Z$, and $\hat \rho_{\hat\varepsilon}^{(n,k)}$ be a univariate spline score estimate (our implementation) using the scaled in-sample residuals.
	
	Let $\hat \theta^{(n)}, \hat\Sigma^{(n)}$ be as in \eqref{eqn:est} where \begin{align}
		\hat f^{(n,k)}(x,z) &= \sum_{j=1}^J\tilde f^{(n,k)}(x+hw_j,z)q_j; \label{eqn:resmooth}\\
		\nabla \hat f^{(n,k)}(x,z) &= \frac1h \sum_{j=1}^J w_j\tilde f^{(n,k)}(x+hw_j,z)q_j ; \label{eqn:resmoothderiv}\\
		\hat \rho^{(n,k)}(x,z) &= \frac{1}{\hat \sigma^{(n,k)}(z)}\; \hat \rho_{\hat\varepsilon}^{(n,k)}\bigg(\frac{x-\hat m^{(n,k)}(z)}{\hat\sigma^{(n,k)}(z)}\bigg). \label{eqn:splinescore}
	\end{align}
	Here we approximate Gaussian expectations via numerical integration, using a deterministic set of pairs $(w_j, q_j)$ such that, for functions $g$, \begin{equation*}
		\E [g(W)] \approx \sum_{j=1}^J g(w_j)q_j.
	\end{equation*}
	We have used $J=101$, $\{w_j\}$ to be an evenly spaced grid on $[-5,5]$, and $q_j$ to be proportional to the standard normal density at $w_j$, scaled so that $\sum_{j=1}^J q_j=1$.
	
	We took the set of bandwidths $\mathcal{H}$ in Algorithm~\ref{alg:resmooth} to be
	\[
	\frac{\exp(-5)}{2\sqrt{3}} \hat{\sigma}_X, \, \frac{\exp(-4.8)}{2\sqrt{3}} \hat{\sigma}_X, \ldots, \frac{\exp(2)}{2\sqrt{3}} \hat{\sigma}_X,
	\]
	where $\hat{\sigma}_X$ denotes the empirical standard deviation of the $X$-variable.
	
	\subsubsection{Difference and basis}
	\label{sect:methodchern}
	
	We form an estimator as in \eqref{eqn:est}. Let $\tilde f^{(n,k)}$ be a gradient boosting regression (\texttt{xgboost} package \citep{xgboost}) of $Y$ on $(X,Z)$ using the out-of-fold data $D^{(n,k)}$. Set basis $b$ to the quadratic basis for $(X,Z)\in\R^{p+1}$, omitting the $X$ term:\begin{equation*}
		b(x,z) = (1, x^2, x z_1, \ldots,x z_{p}, z_1, z_1^2, z_1 z_2,\ldots, z_1 z_{p}, z_2, z_2^2, z_2 z_3,\ldots,\ldots, z_{p}, z_{p}^2).
	\end{equation*}
	Let $\hat \beta^{(n,k)}$ be the Lasso coefficient (\texttt{glmnet} package) when regressing $X$ on $b(X,Z)$ using $D^{(n,k)}$, and $\hat\sigma^{(n,k)}$ be the in-sample variance estimate, computed using the product of $X$ and the $X$ on $Z$ residuals.
	
	Let $\hat \theta^{(n)}, \hat\Sigma^{(n)}$ be as in \eqref{eqn:est} where \begin{align}
		\hat f^{(n,k)}(x,z) &= \tilde f^{(n,k)}(x,z); \label{eqn:original}\\
		\nabla \hat f^{(n,k)}(x,z) &= \frac{\tilde f^{(n,k)}\big(x+\frac D2,z\big)- \tilde f^{(n,k)}\big(x-\frac D2,z\big)}{D} ; \label{eqn:difference}\\
		\hat \rho^{(n,k)}(x,z) &= -\frac1{\big(\hat\sigma^{(n,k)}\big)^2}\Big(x_i - b(x_i,z_i)^{\top} \hat\beta^{(n,k)}\Big). \label{eqn:basisscore}
	\end{align}
	Here $D$ is set to one quarter of the (population) marginal standard deviation of $X$.

	\subsubsection{Partially linear regression}
	\label{sect:methodpl}
	
	We consider a doubly-robust partially linear regression as in \citet[\S 4.1]{ChernozhukovTreatment}, implemented in the \texttt{DoubleMLPLR} function of the \texttt{DoubleML} R package. The partially linear regression makes the simplifying assumption that $\E_P(Y\given  X, Z)  =  \theta_P X + g_P(Z)$. When this relationship is misspecified, the target is given by
	\begin{equation}
		\theta_P^* = \frac{\E_P[\Cov_P\{X, Y \given  Z\}]}{\E_P[\Var_P(X\given  Z)]} \label{eqn:partiallylinearparameter},
	\end{equation}
	\citep{VansteelandtLean}; this does not equal the average partial effect $\theta_P = \E_P[f_P'(X,Z)]$ in general.
	
	The nuisance functions $g_P$ and $\E_P(X\given  Z)$ may be modelled via plug-in machine learning, so again we use gradient boosting (\texttt{xgboost} package \citep{xgboost}). Hyperparameter pre-tuning for $g_P$ estimation is done by regressing $Y-\theta_PX$ on $Z$. Here we have used $\theta_P$ instead of the unknown $\theta_P^*$ for convenience, but we do not expect this to be critical.
	
	\subsubsection{\cite{RothenhauslerICE}}
	\label{sect:methodroth}
	The estimator of \cite{RothenhauslerICE} is based on the debiased Lasso \citep{ZhangDebiased, vandeGeerDebiased}. As they recommend, we use a quadratic basis for $Z\in\R^{p}$,\begin{equation*}
		b(z) = (1, z_1, z_1^2, z_1 z_2,\ldots, z_1 z_{p-1}, z_2, z_2^2, z_2 z_3,\ldots,\ldots, z_{p}, z_{p}^2).
	\end{equation*} 
	We perform the Lasso regressions using \texttt{glmnet} \citep{glmnet}.

	\subsubsection{Spline score estimation}
	We use the univariate estimator of \cite{CoxSpline}, which we implemented according to \cite{NgSpline, NgSplineImplementation}.

\subsection{Additional numerical results}
	\label{sect:nuisancemse}

 We report mean-square error results for the nuisance function estimators (Section~\ref{sect:simulationimplementation}) used in our numerical experiments (Section~\ref{sect:numerical}). We find that our location-scale spline estimator \eqref{eqn:splinescore} achieves lower mean-squared error to the truth $\rho_P$ than the basis estimator \eqref{eqn:basisscore} in all of our simulation settings (Table~\ref{tab:scoremse}). We further find that our resmoothing procedure \eqref{eqn:resmoothderiv} achieves lower mean-squared error as a derivative estimator to the truth $f_P'$ than numerical differencing \eqref{eqn:basisscore} in all of our simulation settings (Table~\ref{tab:derivmse}), at little cost in terms of mean-squared regression error to $f_P$.

 \begin{table}[h]
     \centering
     \begin{tabular}{c|c|c}
        $\varepsilon_P$ & basis & spline \\
        \hline
           normal  & 0.27 (0.01) & \textbf{0.16 (0.01)} \\
           mixture2  & 0.36 (0.01)  & \textbf{0.22 (0.01)} \\
           mixture3  & 0.88 (0.01)  & \textbf{0.41 (0.02)} \\
           logistic & 0.34 (0.01)  & \textbf{0.26 (0.01)} \\
           t4 & 0.70 (0.02) & \textbf{0.23 (0.01)}
     \end{tabular}
     \caption{Out-of-sample mean-squared error estimates (standard errors) for the score function estimates $\hat\rho$ in our simulation settings, based on 100 repeats with training size $800$ and test size $1000$. The basis approach is \eqref{eqn:basisscore}, the spline approach is \eqref{eqn:splinescore}.}
     \label{tab:scoremse}
 \end{table}

 \begin{table}[h]
     \centering
     \begin{tabular}{c|c|c|c||c|c|c}
        $f_P$ & OLS & original & resmoothing & OLS & difference & resmoothing \\
        \hline
    plm & 0.089 (0.001) & 0.055 (0.001) & \textbf{0.047 (0.001)} & \textbf{0.005 (0.001)} & 0.501 (0.011) & 0.026 (0.001) \\
    additive & 0.398 (0.002) & \textbf{0.062 (0.001)} & 0.068 (0.002) & 0.333 (0.001) & 0.189 (0.007) & \textbf{0.106 (0.033)} \\
    interaction & 1.624 (0.012) & 0.227 (0.003) & \textbf{0.225 (0.004)} & 3.596 (0.014) & 1.334 (0.019) & \textbf{0.843 (0.017)}
\end{tabular}
     \caption{Out-of-sample mean-squared error estimates (standard errors) for the regression estimates $\hat f_P$ (first three columns) and derivative estimates $\hat f'$ (last three columns) in our simulation settings, based on 100 repeats with training size $800$ and test size $1000$. For regression the original approach is \eqref{eqn:original} and the resmoothing approach is \eqref{eqn:resmooth}, for derivative estimation the difference approach is \eqref{eqn:difference} and the resmoothing approach is \eqref{eqn:resmoothderiv}.}
     \label{tab:derivmse}
 \end{table}
	
\end{document}